\documentclass{amsart}

\usepackage{amssymb, amsmath}
\usepackage{mathrsfs}
\usepackage{amscd}
\usepackage{verbatim}

\usepackage{tikz-cd}

\usepackage{wasysym}

\usepackage[utf8]{inputenc}
\usepackage{enumerate}
\usepackage{url}

\usepackage[colorinlistoftodos,prependcaption,textsize=tiny]{todonotes}

\usepackage[colorlinks,linkcolor={blue},citecolor={blue},urlcolor={red},]{hyperref}

\allowdisplaybreaks

\theoremstyle{plain}
\newtheorem{theorem}{Theorem}
\theoremstyle{remark}
\newtheorem{remark}[theorem]{Remark}
\newtheorem{example}[theorem]{Example}
\newtheorem{problem}[theorem]{Problem}
\theoremstyle{plain}
\newtheorem{corollary}[theorem]{Corollary}
\newtheorem{lemma}[theorem]{Lemma}
\newtheorem{proposition}[theorem]{Proposition}
\newtheorem{definition}[theorem]{Definition}

\numberwithin{theorem}{section}

\numberwithin{equation}{section}

\def\N{{\mathbb N}}
\def\Z{{\mathbb Z}}

\def\R{{\mathbb R}}
\def\C{{\mathbb C}}

\def\M{{\mathbb M}}
\def\B{{\mathbb B}}

\newcommand{\A}{{\mathscr A}}
\newcommand{\F}{{\mathscr F}}

\newcommand{\Longra}{\ensuremath{\Longrightarrow}}
\newcommand{\Longlra}{\ensuremath{\Longleftrightarrow}}

\newcommand{\longra}{\ensuremath{\longrightarrow}}

\newcommand{\calL}{{\mathcal L}}

\newcommand{\one}{{{\bf 1}}}

\newcommand{\lb}{\langle}
\newcommand{\rb}{\rangle}

\newcommand{\wh}{\widehat}
\newcommand{\supp}{\text{\rm supp\,}}

\newcommand{\wt}{\widetilde}

\newcommand{\norm}[1]{||#1||}

\newcommand{\ip}[2]{\langle #1, #2 \rangle}

\newcommand{\sign}{\text{sign}}

\newcommand{\tr}{\text{Tr}}
\newcommand{\tx}{\tilde{x}}

\newcommand{\ext}{{\rm ext}}
\newcommand {\Schw}{\mathcal{S}}
\newcommand {\Distr}{\mathcal{D}}

\newcommand{\ud}{\, d}
\newcommand{\Dir}{{\rm Dir}}
\newcommand{\DD}{\Delta_{\Dir}}

\newcommand{\Eo}{E_{\rm odd}}

\newcommand{\odd}{{\rm odd}}

\newcommand{\dist}{{\rm dist}}
\newcommand{\Rad}{{\rm Rad}}

\newcommand{\dom}{\mathscr{O}}
\newcommand{\BIP}{\rm{BIP}}
\renewcommand{\subset}{\subseteq}

\begin{document}

\author{Nick Lindemulder}
\address{Delft Institute of Applied Mathematics\\
Delft University of Technology \\ P.O. Box 5031\\ 2600 GA Delft\\The
Netherlands}
\email{N.Lindemulder@hotmail.nl}

\address{Institute of Analysis \\
Karlsruhe Institute of Technology \\
Englerstraße 2 \\
76131 Karlsruhe\\
Germany}
\email{N.Lindemulder@hotmail.nl}

\author{Mark Veraar}
\address{Delft Institute of Applied Mathematics\\
Delft University of Technology \\ P.O. Box 5031\\ 2600 GA Delft\\The
Netherlands} \email{M.C.Veraar@tudelft.nl}

\thanks{The authors are supported by the VIDI subsidy 639.032.427 of the Netherlands Organisation for Scientific Research (NWO)}

\date\today




\title[The heat equation and functional calculus]{The heat equation with rough boundary conditions and holomorphic functional calculus}

\maketitle
\begin{abstract}
In this paper we consider the Laplace operator with Dirichlet boundary conditions on a smooth domain. We prove that it has a bounded $H^\infty$-calculus on weighted $L^p$-spaces for power weights which fall outside the classical class of $A_p$-weights. Furthermore, we characterize the domain of the operator and derive several consequences on elliptic and parabolic regularity. In particular, we obtain a new maximal regularity result for the heat equation with rough inhomogeneous boundary data.
\end{abstract}

\setcounter{tocdepth}{1}
\tableofcontents

\section{Introduction}

Often solutions to PDEs can have blow-up behavior near the boundary of an underlying domain $\dom\subseteq \R^d$. Using weighted spaces with weights of the form $w_{\gamma}^{\dom}(x) := \dist(x,\partial \dom)^{\gamma}$ for appropriate values of $\gamma$, allows for additional flexibility and even obtain well-posedness for problems which appear ill-posed at first sight. PDEs in weighted spaces have been considered by many authors (see e.g.\ \cite{DongKimweighted15,KimK-H08,Krylovheat,Krylovheatpq}). Moreover, the $H^\infty$-functional calculus properties of differential operators on weighted space have been treated in several papers as well (see e.g.\ \cite{Assaad13,AusMar07,AuMarIII,LRW, Martell04}.

The development of the $H^\infty$-calculus was motivated by the Kato square root problem (see \cite{McI90} for a survey) which was eventually solved in \cite{AHLT}. An $H^\infty$-calculus approach to the solution was obtained later in \cite{AKM06}. Since the work \cite{KWcalc} it has turned out that the $H^\infty$-calculus is an extremely efficient tool in the $L^p$-theory of partial differential equations (see the monographs \cite{DenkKaip, PrSibook} and references therein).

In this paper we study the boundedness of the $H^\infty$-calculus of the Laplace operator with Dirichlet boundary conditions $\Delta_{\Dir}$ for bounded $C^2$-domains $\dom$. This operator and its generalizations have been studied in many papers (see \cite{DDHPV,DHP,KuWe}. Our contribution is that we study $\DD$ and its functional calculus on weighted spaces which do not fall into the classical setting, but which are useful for certain partial differential equations. In particular,
we prove the following result.
\begin{theorem}\label{thm:boundeddomainLaplaceintro}
Let $\dom$ be a bounded $C^2$-domain. Let $p\in (1, \infty)$, $\gamma\in (-1,2p-1)\setminus\{p-1\}$ and set $w^{\dom}_{\gamma}(x) = \dist(x,\partial \dom)^{\gamma}$.
Then the operator $-\DD$ on  $L^p(\dom,w^{\dom}_{\gamma})$ with $D(\DD) = W^{2,p}_{\Dir}(\dom,w^{\dom}_{\gamma})$, has a bounded $H^\infty$-calculus of angle zero. In particular, $\DD$ generates an analytic $C_0$-semigroup on $L^p(\dom,w^{\dom}_{\gamma})$.
\end{theorem}
A similar result holds on the half space $\R^d_+$ or small deformations of the half space.
The range $\gamma\in (p-1, 2p-1)$ falls outside the classical $A_p$-setting and Theorem \ref{thm:boundeddomainLaplaceintro} is new in this range. The range $\gamma\in (-1, p-1)$ can be treated by classical methods, and it can be derived from the general $A_p$-case which will be considered in Section \ref{sec:heatApsetting}.

The boundedness of the $H^\infty$-calculus has many interesting consequences for the operator $\DD$ on $L^p(\dom,w^{\dom}_{\gamma})$.
Loosely speaking, the boundedness of the $H^\infty$-calculus can be used as a black box to ensure existence of certain singular integrals.
In particular, the boundedness of the $H^\infty$-calculus implies:
\begin{itemize}
\item Continuous and discrete square function estimates (see \cite[Theorems 10.4.4 $\&$ 10.4.23]{HNVW2}), which are closely related to the classical Littlewood--Paley inequalities.
\item Well-posedness and maximal regularity of the Laplace equation and the heat equation on $L^p(\dom,w^{\dom}_{\gamma})$ (see Corollaries \ref{cor:ellipticregnonAp}, \ref{cor:weightedmaxnonAp}, \ref{cor:max-reg_bdd_domain}).
\item Maximal regularity for the stochastic heat equation on $L^p(\dom,w^{\dom}_{\gamma})$ (see \cite[Theorem 1.1]{NVW12a}).
\end{itemize}
On bounded domains we analyse the spectrum of $\DD$ and in particular we show that the analytic semigroup generated by $\DD$ is exponentially stable. Additionally we use the functional calculus to characterize several of the fractional domain spaces.

The main difficulty in the proof of Theorem \ref{thm:boundeddomainLaplaceintro} in the non-$A_p$ setting is that standard tools from harmonic analysis are not available. For instance, the boundedness of the Hilbert transform, the boundedness of the Hardy-Littlewood maximal function operator, and the Littlewood--Paley decomposition all hold on $L^p(\R^d,w_{\gamma}^{\dom})$ if and only if $\gamma\in (-1, p-1)$ (see \cite[Chapter 9]{GrafakosM} and \cite{Rychkov01}). Here one also needs to use the fact that the $A_p$-condition holds if and only if $\gamma\in (-1, p-1)$.
As a consequence, we have to find a new approach to obtain the domain characterizations, sectoriality estimates and the boundedness of the functional calculus.

We have already mentioned that Theorem \ref{thm:boundeddomainLaplaceintro} implies maximal regularity results. As a further application we will derive a maximal regularity result for the heat equation on weighted spaces with {\em rough inhomogeneous boundary conditions}. The main reason we can allow much rougher boundary data than in previous works is that we allow $\gamma\in (p-1, 2p-1)$. Maximal regularity results can be used to study nonlinear equations in an effective way (see e.g.\ \cite{PSW18} and references therein). The result below is a special case of Theorem \ref{thm:max-reg_bdd_domain_bd-data;interval}.
In order to make the result transparent without losing the main innovative part of the result, we state the result in the special case $u_0 = 0$, $f = 0$ and $p=q$ and without weights in time.
\begin{theorem}\label{thm:inhmainintro}
Let $\dom$ be a bounded $C^{2}$-domain. Let $\lambda\geq 0$. Let $p\in (1, \infty)$ and $\gamma\in (-1, 2p-1)\setminus \{p-1,2p-3\}$ and set $\delta = 1-\frac{1+\gamma}{2p}$.
Assume
\[g \in B^{\delta}_{p,p}(\R_+;L^{p}(\partial\dom)) \cap
L^{p}(\R_+;B^{2\delta}_{p,p}(\partial\dom)),\]
with $g(0,\cdot) = 0$ in the case $\gamma\in (-1, 2p-3)$. Then there exists a unique $u \in W^{1,p}(\R_+;L^p(\dom,w_{\gamma}^\dom))\cap L^{p}(\R_+;W^{2,p}(\dom,w_{\gamma}^{\dom}))$ such that
\begin{align*}
\left\{\begin{array}{rlll}
u' + (\lambda-\Delta) u &=& 0, & \text{on $\R_+\times\dom$}, \\
\tr_{\partial\dom}u &=&g, & \text{on $\R_+\times\partial \dom$}, \\
u(0) &=& 0, & \text{on $\dom$}.
\end{array}\right.
\end{align*}
\end{theorem}
Conversely, the conditions on $g$ are necessary in order for $u$ to be in the intersection space.
Note that $\delta\in (0,1)$ can be taken arbitrarily close to zero by taking $\gamma$ arbitrarily close to $2p-1$. Moreover, if $\gamma\in (2p-3, 2p-1)$ then the compatibility condition $g(0,\cdot) = 0$ also vanishes.

Theorem \ref{thm:inhmainintro} was proved in \cite{DHPinh} and \cite{Weid02} for $\gamma = 0$, and in this case the smoothness parameter equals $\delta = 1-\tfrac1{2p}$. In \cite{DHPinh} actually the general setting of higher order operators $A$ with boundary conditions of Lopatinskii-Shapiro was consider. In \cite{lindemulder2017maximal} the first author extended the latter result to the weighted situation with $\gamma\in (-1, p-1)$, in which case $\delta\in (\tfrac12, 1)$ can only be taken arbitrarily close to $\tfrac12$ by taking $\gamma$ close to $p-1$. It would be interesting to investigate if one can extend special cases of \cite{lindemulder2017maximal} to other values of $\gamma$. In ours proofs the main technical reason that we can extend the range of $\gamma$'s in the Dirichlet setting is that the heat kernel on a half space has a zero of order one at the boundary. The heat kernel in the case of Neumann boundary conditions does not have this property. Moreover, the Neumann trace operator is not well-defined for $\gamma\in (p-1, 2p-1)$. It is a natural question to ask for which kernels associated to higher order elliptic operators with different boundary conditions one has similar behavior at the boundary. In such cases one might be able to allow for rougher boundary data as well.

There exist several theories of elliptic and parabolic boundary value problems on other classes of function spaces than the $L^q(L^p)$-framework of the above. The case that $L^p$ is replaced by a weighted Besov or Triebel-Lizorkin space is considered by the first named author in \cite{LindF} in the elliptic setting and in \cite{LindFpara} in the parabolic setting. The advantage in that setting is that one can use Fourier multiplier theorems for $A_{\infty}$-weights. The results in \cite{LindF,LindFpara}  are independent from the results presented here since in the non-$A_p$ setting Triebel-Lizorkin spaces do not coincide with Sobolev spaces. Results in the framework of tent spaces have been obtained in \cite{amenta2016abstract, AxKeiMc06,AuAxMc} for elliptic equations and in \cite{auscher20162} for parabolic equations. Here in some cases the boundary data is allowed to be in $L^p$ or $L^2$.

The paper is organized as follows. In Section \ref{sec:Hardy} we present some results on traces, Hardy inequalities and interpolation inequalities which will be needed. In Section \ref{sec:heatApsetting} we consider the half space case with $A_p$-weights. In Section \ref{sec:heatnonAp} we consider the half space case for non-$A_p$-weights. We extend the results to bounded domains in Section \ref{sec:loc},  where Theorem \ref{thm:boundeddomainLaplaceintro} can be derived from Corollary \ref{cor:boundeddomainLaplaceC}.
In Section \ref{sec:inhheat} we consider the heat equation with inhomogeneous boundary conditions and, in particular, we will derive Theorem \ref{thm:inhmainintro}.
In many of our considerations we consider the vector-valued situation. This is mainly because it can be convenient to write Sobolev spaces as the intersection of several simpler vector-valued Sobolev spaces.

\subsubsection*{Acknowledgment}  The authors would like to thank Dorothee Frey and Bas Nieraeth for helpful discussions on Section \ref{subs:extra}.

\subsubsection*{Notation}

$\R^d_+ = (0,\infty)\times\R^{d-1}$ denotes the half space. We write $x = (x_1, \tilde{x})\in \R^{d}$ with $x_1\in \R$ and $\tilde{x}\in \R^{d-1}$.
The following shorthand notation will be used throughout the paper
\[w_{\gamma}(x) = |x_1|^{\gamma} \ \ \text{and} \ \ w_{\gamma}^{\dom}(x) = \dist(x, \partial\dom)^{\gamma}.\]
For two topological vector spaces $X$ and $Y$ (usually Banach spaces), $\calL(X,Y)$ denotes the space of continuous linear operators. We write $A \lesssim_p B$ whenever $A \leq C_p B$ where $C_p$ is a constant which depends on the parameter $p$. Similarly, we write $A\eqsim_p B$ if $A\lesssim_p B$ and $B\lesssim_p A$. Unless stated otherwise in the rest of the paper $X$ is assumed to be a Banach space.

\section{Preliminaries}

\subsection{Function spaces and weights}
Let $X$ be a Banach space. For an open set $\dom\subseteq \R^d$ let $\mathcal{D}(\dom;X)$ denote the space of compactly supported smooth functions from $\dom$ into $X$ equipped with its usual inductive limit topology.
Let $\Distr'(\dom;X) = \calL(\mathcal{D}(\dom), X)$ be the space of $X$-valued distributions.
Let $C^\infty_c(\overline{\dom};X)$ be the space of infinitely differentiable functions which vanish outside a compact set $K\subseteq \overline{\dom}$.
Furthermore, $\Schw(\R^d;X)$ denotes the space of Schwartz functions and $\Schw'(\R^d;X) = \calL(\Schw(\R^d), X)$ is the space of $X$-valued tempered distributions. We refer to \cite{amann2003vector,Amann09} for introductions to the theory of vector-valued distribution.

A locally integrable function $w:\dom\to (0,\infty)$ is called a weight.
A weight $w$ on $\R^d$ will be called {\em even} if $w(-x_1, \tx) = w(x_1, \tx)$ for $x_1>0$ and $\tx\in \R^{d-1}$.

Although we will be mainly interested in a special class of weights, it will be natural to formulate some of the result for the class of Muckenhoupt $A_p$-weights. For $p\in (1, \infty)$ and a weight $w:\R^d\to (0,\infty)$, we say that $w\in A_p$ if
\[[w]_{A_p} = \sup_{Q} \frac{1}{|Q|} \int_Q w(x) \ud x \cdot \Big(\frac{1}{|Q|} \int_Q w(x)^{-\frac{1}{p-1}}\ud x\Big)^{p-1}<\infty.\]
Here the supremum is taken over all cubes $Q\subseteq \R^d$ with sides parallel to the coordinate axes. For $p\in (1, \infty)$ and a weight $w:\R^d\to (0,\infty)$ one has $w\in A_p$ if and only if the Hardy--Littlewood maximal function is bounded on $L^p(\R^d,w)$.
We refer the reader to \cite[Chapter 9]{GrafakosM} for standard properties of $A_p$-weights. For a fixed $p$ and a weight $w\in A_p$, the weight $w' = w^{-1/(p-1)}\in A_{p'}$ is the $p$-dual weight.
Define $A_{\infty} =\bigcup_{p>1} A_p$. Recall that $w_{\gamma}(x) := |x_1|^{\gamma}$ is in $A_p$ if and only if $\gamma\in (-1, p-1)$.

For a weight $w:\dom\to (0,\infty)$ and $p\in [1, \infty)$, let $L^p(\dom,w;X)$ denote the Bochner space of all strongly measurable functions $f:\dom\to X$ such  that
\[\|f\|_{L^p(\dom,w;X)} =\Big(\int_\dom |f(x)|^p w(x) dx\Big)^{1/p} < \infty.\]
For a set $\Omega\subseteq \R^d$ with nonempty interior and $w:\Omega\to (0,\infty)$ let $L^1_{\rm loc}(\Omega;X)$ denote the set of all functions such that for all bounded open sets $\Omega_0$ with $\overline{\Omega_0}\subseteq \Omega$, we have $f|_{\Omega_0}\in L^1(\Omega_0,w;X)$. In this case $f$ is called {\em locally integrable} on $\Omega$.
If the $p$-dual weight $w'=w^{-1/(p-1)}$ ($w'=1$ when $p=1$) is locally integrable on $\dom$, then $L^p(\dom,w;X) \hookrightarrow \Distr'(\dom;X)$.

For $p \in (1,\infty)$, an integer $k\geq 0$ and a weight $w$ with $w'=w^{-1/(p-1)} \in L^{1}_{\rm loc}(\dom)$, let $W^{k,p}(\dom,w;X)\subseteq \Distr'(\dom;X)$ be the {\em Sobolev space} of all $f\in L^p(\dom,w;X)$ with $D^\alpha f \in L^p(\dom,w;X)$ for all $|\alpha|\leq k$ and set
\begin{align*}
\|f\|_{W^{k,p}(\dom,w;X)} &= \sum_{|\alpha|\leq k} \|D^{\alpha} f\|_{L^p(\dom,w;X)}, \\ [f]_{W^{k,p}(\dom,w;X)} &= \sum_{|\alpha| = k} \|D^{\alpha} f\|_{L^p(\dom,w;X)}.
\end{align*}
$W^{k,p}(\dom,w;X)$ is a Banach space. We refer to \cite{Kufner80, KufnerOpic84} for a detailed study of weighted Sobolev spaces. Finally, for a set $\Omega\subseteq \R^d$ with nonempty interior we let $W^{k,1}_{\rm loc}(\Omega,w;X)$ denote the space of functions such that $D^\alpha f\in L^1_{\rm loc}(\Omega,w;X)$ for all $|\alpha|\leq k$.

Let us mention that density of $C^\infty_c(\overline{\dom};X)$ in $W^{1,p}(\dom,w;X)$ is not true in general, not even for $w\in A_{\infty}$. A sufficient condition class is $w\in A_p$ (see \cite[Corollary 2.1.6]{Tur00}). Further examples and counterexamples can be found in \cite[Chapter 7 $\&$ 11]{Kufner80} and \cite{Zhikov98}.

We further would like to point out that in general $W^{k,p}(\dom,w)$ does not coincide with a Triebel-Lizorkin space $F^{k}_{p,2}(\dom,w)$ if $w\notin A_p$. Moreover, in the $X$-valued setting this is even wrong for $w = 1$ unless $X$ is isomorphic to a Hilbert space (see \cite{HanMeyer}).

\subsection{Localization and $C^k$-domains\label{subsec:locCk}}

\begin{definition}\label{DBVP:def:special_domain}
Let $\dom \subset \R^{d}$ be a domain and let $k \in \N_0\cup\{\infty\}$.
Then $\dom$ is called a \emph{special $C^{k}$-domain} when, after rotation and translation, it is of the form
\begin{equation}\label{DBVP:eq:def:special_domain;above_graph}
\dom = \{ x= (y,x') \in \R^{d} : y > h(x') \}
\end{equation}
for some $C^{k}$-function $h:\R^{d-1} \longra \R$. If $h$ can be chosen with compact support, then $\mathscr{O
}$ is called a \emph{special $C^{k}_{c}$-domain}.
\end{definition}

For later it will be convenient to define, given a special $C^{k}_{c}$-domain $\dom$ with $k\in \N_0$, the numbers
\begin{equation}\label{DBVP:eq:special_domain;above_graph;C^{k}-number}
[\dom]_{C^{k}} := \inf_{h}\norm{h}_{C^{k}_{b}(\R^{d-1})}
\end{equation}
where the infimum is taken over all $h \in C^{k}_{c}(\R^{d-1};\R)$ for which $\dom$, after rotation and translation, can be represented as \eqref{DBVP:eq:def:special_domain;above_graph}.

\begin{definition}\label{DBVP:def:sec:localization;smooth_domain}
Let $k\in \N_0\cup\{\infty\}$. A domain $\dom \subset \R^{d}$ is said to be a \emph{$C^{k}$-domain} when every boundary point $x \in \partial \dom$ admits an open neighborhood $V$ with the property that
\[
\dom \cap V = W \cap V \quad\mbox{and}\quad \partial \dom \cap V = \partial W \cap V
\]
for some special $C^{k}$-domain $W \subset \R^{d}$.
\end{definition}

Note that, in the above definition, $V$ may be replaced by any smaller open neighborhood of $x$. Hence, we may without loss of generality assume that $W$ is a $C^{k}_{c}$-domain. Moreover, if $k\in \N_0$ then for any $\epsilon > 0$ we can arrange that $[W]_{C^{k}} < \epsilon$.

If $U,V\subseteq \R^d$ are open and $\Phi:U\to V$ is a $C^1$-diffeomorphism, then we define $\Phi_*:L^1_{\rm loc}(U) \to L^1_{\rm loc}(V)$ by
\[\lb \Phi_* f,g\rb := \lb f, j_{\Phi} g\circ\Phi\rb, \ \  f\in L^1_{\rm loc}(U), g\in C_c(V),\]
where $j_{\Phi} = \det(\nabla \Phi)$ denotes the Jacobian. In this way $\Phi_{*}f = f\circ \Phi^{-1}$.

Now assume $h\in C^k_{c}(\R^{d-1})$ with $k\geq 1$ and
\begin{equation}\label{eq:Omegaspecial}
\dom = \{(x_1,\tx): \tx\in \R^{d-1}, x_1>h(\tx)\}.
\end{equation}
Define a $C^k$-diffeomorphism $\Phi:\dom\to \R^d_+$ by
\begin{align}\label{eq:diffeo_special_domain}
\Phi(x) = (x_1-h(\tx), \tx).
\end{align}
Obviously, $\det(\nabla \Phi) = 1$. For a weight $w:\R^d\to (0,\infty)$, let $w_{\Phi}:\dom\to (0,\infty)$ be defined by $w_{\Phi}(x) = w(\Phi(x))$.
In the important case that $w(x) =  |x_1|^{\gamma}$, we have
\[w_{\Phi}(x) =  |x_1-h(\tx)|^{\gamma} \eqsim \dist(x,\partial \dom)^{\gamma},  \ \ x\in \dom.\]
In this way for $k\in \N_0$, the mapping $\Phi_*$ defines a bounded isomorphism
\[\Phi_*:W^{k,p}(\dom,w_{\Phi})\to W^{k,p}(\R_+^d,w_{\gamma})\]
with inverse $(\Phi^{-1})_*$.

In the paper we will often use a localization procedure. We will usually leave out the details as they are standard. In the localization argument for the functional calculus (see Theorem \ref{thm:boundeddomainLaplaceX}) we do give the full details as a precise reference with weighted spaces seems unavailable.

Given a bounded $C^k$-domain $\dom$ with $k\geq 1$, then we can find $\eta_{0} \in C^{\infty}_{c}(\dom)$ and $\{\eta_{n}\}_{n=1}^{N} \subset C^{\infty}_{c}(\R^{d})$ such that $\supp(\eta_{n}) \subset V_{n}$ for each $n \in \{1,\ldots,N\}$ and $\sum_{n=0}^{N}\eta_{n}^{2} = 1$ (see \cite[Ch.8, Section 4]{Krylov2008_book}).
These functions can be used to decompose the space $E_k:= W^{k,p}(\dom,w^{\dom}_{\gamma};X)$  as
\[
F_k:= W^{k,p}(\R^{d};X)  \oplus \bigoplus_{n=1}^{N} W^{k,p}(\dom_{n},w^{\dom_{n}}_{\gamma};X)
\]
The mappings $\mathcal{I}: E_k \longra F_k$ and $\mathcal{P}: F_k \longra E_k$ given by
\begin{equation}\label{eq:domaindecomoperator}
\mathcal{I} f  = (\eta_{n}f)_{n=0}^{N} \qquad \text{and} \qquad \mathcal{P} (f_{n})_{n=0}^{N} = \sum_{n=0}^{N}\eta_{n}f_{n}.
\end{equation}
satisfy $\mathcal{P}\mathcal{I} = I$, thus $\mathcal{P}$ is a retraction with coretraction $\mathcal{I}$.

\subsection{Functional calculus}

Let $\Sigma_{\varphi} = \{z\in \C: |\arg(z)|<\varphi\}$.
We say that an unbounded operator $A$ on a Banach space $X$ is a {\em sectorial operator} if $A$ is injective, closed, has dense range and there exists a $\varphi\in (0,\pi)$ such that $\sigma(A)\subseteq \overline{\Sigma_{\varphi}}$ and
\begin{align*}
\sup_{\lambda\in \C\setminus \Sigma_{\varphi}} \|\lambda R(\lambda,A)\|<\infty.
\end{align*}
The infimum over all possible $\varphi$ is called the {\em angle of sectoriality} and denoted by $\omega(A)$. In this case we also say that $A$ is {\em sectorial of angle} $\omega(A)$. The condition that $A$ has dense range is automatically fulfilled if $X$ is reflexive (see \cite[Proposition 10.1.9]{HNVW2}).

Let $H^\infty(\Sigma_{\omega})$ denote the space of all bounded holomorphic functions $f:\Sigma_{\omega}\to \C$ and let $\|f\|_{H^\infty(\Sigma_{\omega})} = \sup_{z\in \Sigma_{\omega}} |f(z)|$. Let $H^\infty_0(\Sigma_{\omega})\subseteq H^\infty(\Sigma_{\omega})$  be the set of all $f$ for which there exists an $\varepsilon>0$ and $C>0$ such that $|f(z)|\leq C \frac{|z|^{\varepsilon}}{1+|z|^{2\varepsilon}}$.

If $A$ is sectorial, $\omega(A)<\nu<\omega$, and $f\in H^\infty_0(\Sigma_{\omega})$ we let
\[f(A) = \frac{1}{2\pi i} \int_{\partial \Sigma_{\nu}} f(\lambda) R(\lambda,A) \ud \lambda,\]
where $\partial \Sigma_{\nu}$ is oriented downwards. The operator $A$ is said to have a {\em bounded $H^\infty(\Sigma_{\omega})$-calculus} if there exists a constant $C$ such that for all $f\in H^\infty_0(\Sigma_{\omega})$
\[\|f(A)\|\leq C\|f\|_{H^\infty(\Sigma_{\omega})}.\]
The infimum over all possible $\omega> \omega(A)$ is called the angle of the $H^\infty$-calculus and is denoted by $\omega_{H^\infty}(A)$. In this case we also say that $A$ has a bounded $H^\infty$-calculus of angle $\omega_{H^\infty}(A)$.

For details on the $H^\infty$-functional calculus we refer the reader to \cite{Haase:2} and \cite{HNVW2}.

The following well-known result on the domains of fractional powers and complex interpolation will be used frequently. For the definitions of the powers $A^{\alpha}$ with $\alpha\in \C$ we refer to \cite[Chapter 3]{Haase:2}. For details on complex interpolation we refer to \cite{BeLo,HNVW1,Tr1}.

We say that $A$ has {\em BIP (bounded imaginary powers)} if for every $s\in \R$, $A^{is}$ extends to a bounded operator on $X$. In this case one can show that there exists $M,\sigma\geq 0$ such that (see \cite[Corollary 3.5.7]{Haase:2})
\begin{equation}\label{eq:Aisesigmaetc}
\|A^{is}\|\leq M e^{\sigma s}, \qquad s\in\R.
\end{equation}
Let $\omega_{\BIP}(A) = \inf\{\omega\in \R: \exists M>0 \ \text{such that for all} \ s\in \R \ \|A^{is}\|\leq Me^{\omega |s|}\}$.
One can easily check that $\omega_{\BIP}(A)\leq \omega_{H^\infty}(A)$.

The next result can be found in \cite[Theorem 6.6.9]{Haase:2} and \cite[Theorem 1.15.3]{Tr1}.
\begin{proposition}\label{prop:BIP}
Assume $A$ is a sectorial operator such that $A$ has BIP. Then for all $\theta\in (0,1)$ and $0\leq \alpha<\beta$ we have
\[[D(A^{\alpha}), D(A^{\beta})]_{\theta} = D(A^{(1-\theta)\alpha+\theta\beta}),\]
where the constant in the norm equivalence depends $\alpha, \beta,\theta$, the sectoriality constants and on the constant $M$ and $\sigma$ in \eqref{eq:Aisesigmaetc}.
\end{proposition}

For two closed operators $(A,D(A))$ and $(B,D(B))$ on $X$ we define $D(A + B) := D(A)\cap D(B)$ and $(A+B) u = Au + B u$. Often it is difficult to determine whether $A+B$ with the above domain is a closed operator. Sufficient conditions are given in the following theorem which will be used several times throughout this paper (see \cite{DoreVenni,PrSo}).
\begin{theorem}[Dore--Venni]\label{thm:operatorsum}
Let $X$ be a UMD space.
Assume $A$ and $B$ are sectorial operators on $X$ with commuting resolvents and assume $A$ and $B$ both have BIP with $\omega_{\BIP}(A)+\omega_{\BIP}(B)<\pi$. Then the following assertions hold:
\begin{enumerate}[(1)]
\item\label{it:sum1} $A+B$ is a closed sectorial operator with
with $\omega(A+B)\leq \max\{\omega_{\BIP}(A),\omega_{\BIP}(B)\}$
\item\label{it:sum2} There exists a constant $C\geq 0$ such that for all $x\in D(A)\cap D(B)$,
\[\|Ax\|+ \|Bx\|\leq C \|Ax+Bx\|,\]
and if $0\in \rho(A)$ or $0\in \rho(B)$, then $0\in \rho(A+B)$.
\end{enumerate}
\end{theorem}

The following can be used to obtain boundedness of the $H^\infty$-calculus for translated operators (\eqref{it:firsttransH} is straightforward and \eqref{it:secondttransH} follows from \cite[Corollary 5.5.5]{Haase:2}):
\begin{remark}\label{rem:Hinftytranslate}
Let $\sigma\in (0,\pi)$ and assume $A$ is a sectorial operator of angle $\leq \sigma$
\begin{enumerate}[$(1)$]
\item\label{it:firsttransH} If $A$ has a bounded $H^\infty$-calculus of angle $\leq \sigma$, then for all $\lambda\geq 0$, $A+\lambda$ has a bounded $H^\infty$-calculus of angle $\leq \sigma$.
\item\label{it:secondttransH} If there exists a $\wt{\lambda}>0$ such that $A+\wt{\lambda}$ has a bounded $H^\infty$-calculus of angle $\leq \sigma$, then for all $\lambda>0$, $A+\lambda$ has a bounded $H^\infty$-calculus of angle $\leq \sigma$.
\end{enumerate}
\end{remark}

\subsection{UMD spaces and Fourier multipliers}
Below the geometric condition UMD will often be needed for $X$. UMD stands for unconditional martingale differences. One can show that a Banach space $X$ is a UMD space if and only if the Hilbert transform is bounded if and only if the vector-valued analogue of the Mihlin multiplier theorem holds. For details we refer to \cite[Chapter 5]{HNVW1}. Here we recall the important examples for our considerations.
\begin{itemize}
\item Every Hilbert space is a UMD space;
\item If $X$ is a UMD space, $(S, \Sigma, \mu)$ is $\sigma$-finite and $p\in (1, \infty)$, then $L^p(S;X)$ is a UMD space.
\item UMD spaces are reflexive.
\end{itemize}

For $m\in L^\infty(\R^d)$ define
\[T_m: \Schw(\R^d;X) \to \Schw'(\R^d;X), \qquad T_m f = \mathcal \F^{-1} (m \wh f).\]
For $p\in [1,\infty)$ and $w\in A_\infty$ the Schwartz class $\Schw(\R^d;X)$ is dense in $L^p(\R^d,w;X)$ (see Lemma \ref{lem:densitynonAp}).

The following is a weighted version of Mihlin's type multiplier theorem and can be found in \cite[Proposition 3.1]{MeyVerpoint}
\begin{proposition}\label{prop:weightedmihlin}
Let $X$ be a UMD space, $p\in (1, \infty)$ and $w\in A_p$. Assume that $m\in C^{d+2}(\R^d\setminus\{0\})$ satisfies
\begin{equation}\label{Mihlin-assum}
C_m = \sup_{|\alpha| \leq d+2}\sup_{\xi \neq 0} |\xi|^{|\alpha|} |D^{\alpha} m(\xi)| <\infty.
\end{equation}
Then $T_m$ extends to a bounded operator on $L^p(\R^d,w;X)$, and its operator norm only depends on $d$, $X$, $p$, $[w]_{A_p}$ and $C_m$.
\end{proposition}

\begin{proposition}\label{propddtHinfty}
Let $X$ be a UMD space. Let $q\in (1, \infty)$ and $v\in A_q(\R)$. Then the following assertions hold:
\begin{enumerate}
\item The operator $\tfrac{d}{dt}$ with $D(\tfrac{d}{dt}) = W^{1,q}(\R,v;X)$ has a bounded $H^\infty$-calculus with $\omega_{H^\infty}(\tfrac{d}{dt}) \leq \tfrac{\pi}{2}$.
\item The operator $\tfrac{d}{dt}$ with $D(\tfrac{d}{dt}) = W^{1,q}_0(\R_+,v;X)$ has a bounded $H^\infty$-calculus with $\omega_{H^\infty}(\tfrac{d}{dt}) \leq \tfrac{\pi}{2}$.
\end{enumerate}
\end{proposition}
Here $W^{1,q}_0(\R_+,v;X)$ denotes the closed subspace of $W^{1,q}(\R_+,v;X)$ of functions which are zero at $t=0$.
\begin{proof}
(1) follows from Proposition \ref{prop:weightedmihlin} and \cite[Theorem 10.2.25]{HNVW2}. (2) can be derived as a consequence by repeating part of the proof of \cite[Theorem 6.8]{LMV} where the case $v(t) = |t|^{\gamma}$ was considered.
\end{proof}

For $p\in (1, \infty)$, $w\in A_p$ and $s\in \R$, we define the {\em Bessel potential space} $H^{s,p}(\R^d,w;X)$ as the space of all $f\in \Schw'(\R^d;X)$ for which $\F^{-1} [(1+|\cdot|^2)^{s/2} \wh{f} \in L^p(\R^d,w;X)$. This is a Banach space when equipped with the norm
\[\|f\|_{H^{s,p}(\R^d,w;X)} = \|\F^{-1} [(1+|\cdot|^2)^{s/2} \wh{f} ]\|_{L^p(\R^d,w;X)}.\]
For an open subset $\dom\subseteq \R^d$ the space $H^{s,p}(\dom,w;X)$ is defined as all restriction $f|_{\dom}$ where $f\in H^{s,p}(\dom,w;X)$. This is a Banach space when equipped with the norm
\[\|f\|_{H^{s,p}(\dom,w;X)} = \inf\{\|g\|_{H^{s,p}(\R^d,w;X)}: g|_{\dom} = f \ \text{where} \ g\in H^{s,p}(\R^d,w;X)\}.\]

The next result can be found in \cite[Propositions 3.2 $\&$ 3.5]{MeyVerpoint}. The duality pairing mentioned in the statement below is the natural extension of
\[\lb f, g\rb = \int_{\R^d} \lb f(x),g(x)\rb dx, \ \ \ f\in \Schw(\R^d;X), \ g\in \Schw(\R^d;X^*).\]

\begin{proposition}\label{prop:HWduality}
Let $X$ be a UMD space, $p\in (1, \infty)$ and $w\in A_p$. Then
\[H^{m,p}(\R^d,w;X) = W^{m,p}(\R^d,w;X) \qquad  \text{ for all }\, m \in \N_0. \label{H=W}\]
Moreover, for all $s\in \R$, one has $[H^{s,p}(\R^d,w;X)]^* = H^{-s,p'}(\R^d,w';X^*)$.
\end{proposition}
The UMD condition is also necessary in the above result (see \cite[Theorem 5.6.12]{HNVW1}).

\begin{proposition}[Intersection representation]\label{prop:intersectionrepr}
Let $d,d_1, d_2,n\geq 1$ be integers such that $d_1+d_2=d$. Let $w\in A_p(\R^{d_1})$. Then
\[
W^{n,p}(\R^{d},w;X) =
W^{n,p}(\R^{d_1},w_;L^{p}(\R^{d_2};X)) \cap L^{p}(\R^{d_1},w;W^{n,p}(\R^{d_2};X)).
\]
\end{proposition}
In the above we use the convention that $w$ is extended in a constant way in the remaining $d_2$ coordinates. In this way $w\in A_p(\R^d)$ as well.
\begin{proof}
$\hookrightarrow$ is obvious. To prove the converse
Let $\alpha$ be a multiindex with $k:=|\alpha|\leq n$. It suffices to prove $\|D^\alpha u\|_{L^p(w;X)}\leq C(\|u\|_{L^p(w;X)}+\sum_{j=1}^d \|D_j^k u\|_{L^p(w;X)})$. This follows by using the Fourier multiplier $m$:
\[m(\xi) = \frac{(2\pi\xi)^\alpha}{1+\sum_{j=1}^d (2\pi \rho(\xi_j)\xi_j)^k}.\]
Here $\rho\in C^\infty(\R)$ is an odd function with $\rho = 0$ on $[0,1/2]$ and $\rho = 1$ on $[1, \infty]$.
Now using Proposition \ref{prop:weightedmihlin}  one can argue in a similar way as in \cite[Theorem 5.6.11]{HNVW1}.
\end{proof}

\section{Hardy's inequality, traces, density and interpolation}\label{sec:Hardy}

In this section we will prove some elementary estimates of Hardy and Sobolev type and obtain some density and interpolation results. We will present the results in the $X$-valued setting, and later on apply this in the special case $X = L^p(\R^{d-1})$ to obtain extensions to higher dimensions in Theorem \ref{thm:domain}.

Details on traces in weighted Sobolev spaces can be found in \cite{Kimtrace08} and \cite{lindemulder2017maximal}. We will need some simple existence results in one dimension.

\subsection{Hardy's inequality and related results\label{subsec:Hardy}}

\begin{lemma}\label{lem:traceAp}
Let $p\in [1, \infty)$ and let $w$ be a weight such that $\|w^{-\frac{1}{p-1}}\|_{L^{1}(0,t)}<\infty$ for all $t\in (0,\infty)$.
Then $W^{1,p}(\R_+,w;X)\hookrightarrow C([0,\infty);X)$ and for all $u\in W^{1,p}(\R_+,w;X)$,
\[\sup_{x\in [0,t]}\|u(x)\| \leq C_{t,p,w} \|u\|_{W^{1,p}(\R_+,w;X)}, \ \ \ t\in [0,\infty)\]
Moreover, the following results hold in the special case that $w(x) = w_{\gamma}(x) = |x|^{\gamma}$:
\begin{enumerate}[$(1)$]
\item\label{it:estuniformR+} If $\gamma\in [0,p-1)$, then $u(x)\to 0$ as $x\to \infty$ and for all $u\in W^{1,p}(\R_+,w_{\gamma};X)$,
\begin{equation*}
\sup_{x\geq 0}\|u(x)\| \leq C_{p,\gamma} \|u\|_{W^{1,p}(\R_+,w_{\gamma};X)}.
\end{equation*}
\item\label{it:estuniform0} If $\gamma< -1$, then for all $u\in L^{p}(\R_+,w_{\gamma};X)\cap C([0,\infty);X)$, $u(0)=0$.
\end{enumerate}
\end{lemma}
Note that the local $L^1$-condition on $w$ holds in particular for $w\in A_p$.

\begin{proof}
Let $u\in W^{1,p}(\R_+,w;X)$. By H\"older's inequality and the assumption on $w$ we have $L^p((0,t), w;X) \hookrightarrow L^1(0,t;X)$. In particular $u$ and $u'$ are locally integrable on $[0,\infty)$.
Let \[v(s) = \int_0^s u'(x) \ud x, \ \ s\in (0,t).\]
Then $v$ is continuous on $[0,t]$ and moreover $v' = u'$ on $(0,t)$ (see \cite[Lemma 2.5.8]{HNVW1}). It follows that there is a $z\in X$ such that $u= z + v$ for all $s\in (0,t)$.
In particular, $u$ has a continuous extension $\overline{u}$ to $[0,t]$ given by $\overline{u} = z+v$.

To prove the required estimates we just write $u$ instead of $\overline{u}$. Let $x\in [0,\infty)$.
Define $\zeta$ as $\zeta(x) = 1-x$ for $x\in [0,1]$ and $\zeta = 0$ on $[1, \infty)$. Then for $x\in [0,t]$, we have
\begin{align*}
u(x) &= \int_0^1 \frac{d}{ds}(u(s+x) \zeta(s)) \ud s  = \underbrace{\int_0^1 u'(s+x) \zeta(s) \ud s}_{T_1}+ \underbrace{\int_0^1 u(s+x) \zeta'(s) \ud s}_{T_2}.
\end{align*}
Then by H\"older's inequality
\begin{align*}
\|T_1\| &\leq \Big(\int_0^1 \|u'(s+x)\|^p w(s+x) \ud s\Big)^{1/p} \|s\mapsto  w(s+x)^{-1/(p-1)}\|_{L^{1}(0,1)}^{1/p'}
\\ &\leq C_{w,t,p} \|u'\|_{L^p(\R_+,w;X)},
\end{align*}
where $C_{w,t,p}^{p'} = \|w^{-1/(p-1)}\|_{L^{1}(0,t+1)}$.
Similarly, $\|T_2\| \leq C_{w,t,p} \|u\|_{L^p(\R_+,w;X)}$. Therefore, the required estimate for $\sup_{x\in [0,t]}\|u(x)\|$ follows.

The estimate in \eqref{it:estuniformR+} follows from
\[\int_0^1 w_{\gamma}(s+x)^{-1/(p-1)} \ud s \leq \int_0^1 w_{\gamma}(s)^{-1/(p-1)} \ud s=:C_{p,\gamma}.\]
Moreover,$u(x)\to 0$ as $x\to \infty$ because $\int_0^1 w_{\gamma}(s+x)^{-1/(p-1)} \ud s\to 0$ as $x\to \infty$.

To prove \eqref{it:estuniform0} note that
\[\|u(0)\| = \lim_{t\to \infty} \frac1t \int_0^t \|u(s)\| \ud s.\]
Now by H\"older's inequality we have
\[\frac1t \int_0^t \|u(s)\| \ud s\leq \frac1t \|u\|_{L^p(\R_+,w_{\gamma})} \Big(\int_0^t s^{-\gamma p'} \ud s\Big)^{1/p'} \leq C \|u\|_{L^p(\R_+,w_{\gamma})} t^{-\frac{\gamma+1}{p}}\]
and the latter tends to zero as $t\to 0$.
\end{proof}

Next we state two well-known consequences of Hardy's inequality (see \cite[Theorem 10.3.1]{Garling} and
\cite[Section~5]{Kufner80}).
\begin{lemma}\label{lem:Hardy}
Assume $p\in [1, \infty)$. Let $u\in W^{1,p}(\R_+,w_{\gamma};X)$.
Then
\[\|u\|_{L^p(\R_+,w_{\gamma-p};X)}\leq C_{p,\gamma} \|u'\|_{L^p(\R_+,w_{\gamma};X)}.\]
if ($\gamma<p-1$ and $u(0) = 0$) or $\gamma>p-1$.
\end{lemma}
In the above result, by Lemma \ref{lem:traceAp}, $u\in C([0,\infty);X)$ if $\gamma<p-1$.
\begin{proof}
First consider $\gamma<p-1$. Writing $u(t) = \int_0^t u'(s) ds$, it follows that
\[\|u(t)\|_X \leq \int_0^t \|u'(s)\|_X ds.\]
Now the result follows from Hardy's inequality (see \cite[Theorem 10.3.1]{Garling}).
The case $\gamma>p-1$ follows similarly by writing $u(t) = \int_t^\infty u'(s) ds$. Here we use the fact that, by approximation, it suffices to consider the case where $u=0$ on $[n,\infty)$.
\end{proof}

For other exponents $\gamma$ than the ones considered in Lemma \ref{lem:traceAp} another embedding result follows. Note that this falls outside the class of $A_p$-weights.
\begin{lemma}\label{lem:lem:traceA2p}
Let $p\in [1, \infty)$ and $\gamma\in (p-1,2p-1)$.
Then $W^{2,p}(\R_+,w_{\gamma};X)\hookrightarrow C_b([0,\infty);X)$ and for all $u\in W^{2,p}(\R_+,w_{\gamma};X)$, $u(x)\to 0$ as $x\to \infty$ and
\[\sup_{x\geq 0}\|u(x)\| \leq C_{t,p,\gamma} \|u\|_{W^{2,p}(\R_+,w_{\gamma};X)}.\]
\end{lemma}
\begin{proof}
By Lemma \ref{lem:Hardy} $\|u^{(k)}\|_{L^p(\R_+,w_{\gamma-p};X)}\leq C_{p,\gamma} \|u^{(k+1)}\|_{W^{2,p}(\R_+,w_{\gamma};X)}$ for $k\in \{0,1\}$.
Therefore, $u\in W^{1,p}(\R_+,w_{\gamma-p};X)$. Now the required continuity and estimate of $\|u(x)\|$ for $x\in [0,1]$ follows from Lemma \ref{lem:traceAp}. To prove the estimate for $x\in [1, \infty)$, we can repeat the argument used in Lemma \ref{lem:traceAp} \eqref{it:estuniformR+}. Indeed, for $x\geq 1$,
\[\int_0^1 w_{\gamma}(s+x)^{-1/(p-1)} \ud s \leq \int_0^1 w_{\gamma}(s+1)^{-1/(p-1)} \ud s=:C_{p,\gamma}. \qedhere\]
\end{proof}

\subsection{Traces and Sobolev embedding\label{subsec:tracesSob}}

For $u\in W^{1,1}_{\rm loc}(\overline{\R^d_+};X)$ we say that $\tr(u) = 0$ if $\tr(\varphi u) = 0$ for every $\varphi\in C^\infty$ with bounded support in $\overline{\R^d_+}$. Note that $\varphi u\in W^{1,1}([0,\infty);L^1(\R^{d-1};X))$ whenever $u\in W^{1,p}(\R^d_+,w;X)$ and $w^{-\frac{1}{p-1}}\in L^1_{\rm loc}(\overline{\R^d_+})$. Thus the existence of the trace of $\varphi u$ follows from Lemma \ref{lem:traceAp}.

For integers $k\in \N_0$, $p\in (1, \infty)$ and $w\in A_p$, we let
\begin{align}
\label{eq:Wkpdir}
W^{k,p}_{\Dir}(\R^d_+,w;X) & = \{u\in W^{k,p}(\R^d_+,w;X):\tr(u) = 0\},
\\ \notag
W^{k,p}_{0}(\R^d_+,w;X) &= \{u\in W^{k,p}(\R^d_+,w;X):\tr(D^{\alpha}u) = 0 \ \text{for all $|\alpha|<k$}\}.
\end{align}
The traces in the above formulas exists since $W^{k,p}(\R^d_+,w;X)\hookrightarrow W^{k,1}_{\rm loc}(\overline{\R_+^d};X)$.

We extend the definitions of the above spaces to the non-$A_p$-setting.
For $p\in [1, \infty)$, $\gamma\in (p-1,2p-1)$ and $k\in \N_0$ let
\begin{align*}
W^{k,p}_{\Dir}(\R^d_+,w_{\gamma};X) &= \left\{u\in W^{k,p}(\R^d_+,w_{\gamma};X): \tr(u) = 0 \ \text{if} \ k>\frac{\gamma+1}{p}\right\}.\\
W^{k,p}_{0}(\R^d_+,w_{\gamma};X) &= \left\{u\in W^{k,p}(\R^d_+,w;X):\tr(D^{\alpha}u) = 0 \ \text{if $k-|\alpha|>\frac{\gamma+1}{p}$}\right\}.
\end{align*}
Here the trace exists if $j:=k-|\alpha|>\frac{\gamma+1}{p}$ since then $j\geq 2$ and,
by Lemmas \ref{lem:traceAp} and \ref{lem:lem:traceA2p},
\[W^{j,p}(\R^d_+,w_{\gamma};X)\hookrightarrow W^{j,p}(\R_+,w_{\gamma};L^p(\R^{d-1};X))\hookrightarrow C([0,\infty);L^p(\R^{d-1};X)).\]
For $\gamma\in (-\infty, -1)$ and $k\in \N_0$ we further let
\[W^{k,p}_{\Dir}(\R^d_+,w_{\gamma};X) = W^{k,p}_{0}(\R^d_+,w_{\gamma};X) = W^{k,p}(\R^d_+,w_{\gamma};X).\]
This notation is suitable since for $k\in \N_1$, by Lemma \ref{lem:traceAp},
\begin{align*}
W^{k,p}(\R^d_+,w_{\gamma};X)& \hookrightarrow W^{k,p}(\R_+,w_{\gamma};L^p(\R^{d-1};X))
\\ & \subseteq \{u\in C([0,\infty;X);L^p(\R^{d-1})): u(0) = 0\}.
\end{align*}

Using the $C^k$-diffeomorphisms $\Phi$ of Subsection \ref{subsec:locCk} and localization one can extend the definitions of the traces and function spaces $W^{k,p}_{\Dir}(\dom,w_{\Phi};X)$ and $W^{k,p}_{0}(\dom,w_{\Phi};X)$ to special $C^k_c$-domains $\dom$ and bounded $C^k$-domains.

The following Sobolev embeddings are a direct consequence of Lemma~\ref{lem:Hardy} and a localization argument (also see \cite[Theorem~8.2 $\&$ 8.4]{Kufner80}).
\begin{corollary}\label{cor:lem:Hardy;Sobolev_emb}
Let $p\in [1, \infty)$, $k \in \N_{1}$ and $\gamma\in\R$. Let $\dom$ be a bounded $C^k$-domain or a special $C^k_c$-domain. Then
\begin{align*}
W^{k,p}_{0}(\dom,w_{\gamma}^{\dom};X) & \hookrightarrow W^{k-1,p}(\dom,w_{\gamma-p}^{\dom};X), \qquad  \text{if} \ \gamma<p-1,
\\ W^{k,p}(\dom,w_{\gamma}^{\dom};X) & \hookrightarrow W^{k-1,p}(\dom,w_{\gamma-p}^{\dom};X), \qquad \text{if} \ \gamma>p-1,
\\ W^{k,p}_0(\dom,w_{\gamma}^{\dom};X) & \hookrightarrow W^{k-1,p}_0(\dom,w_{\gamma-p}^{\dom};X), \qquad \text{if} \ \gamma\notin \{j p - 1: j\in \N_1\}.
\end{align*}
\end{corollary}

\subsection{Density results\label{subsec:Density}}

\begin{lemma}\label{lem:densitynonAp}
Let $w\in A_{\infty}$ and $p\in [1, \infty)$. Let $\dom$ be an open subset of $\R^{d}$. Then $C^\infty_c(\dom)\otimes X$ is dense in $L^p(\dom,w;X)$.
\end{lemma}
\begin{proof}
Since $L^p(\dom,w)\otimes X$ is dense in $L^p(\dom,w;X)$ it suffices to consider the scalar setting.
We claim that it furthermore suffices to approximate functions which are compactly supported in $\dom$.

To prove the claim, let $f\in L^p(\dom,w)$ and let $(K_{n})_{n \in \N}$ be an exhaustion by compact sets of $\dom$. Observe that $f\one_{K_n}\to f$ by the dominated convergence theorem. Therefore, it suffices to consider functions $f$ with compact support in $\dom$. Extending such functions $f$ by zero to $\R^{d}$, the claim follows.

Let $q\in (p, \infty)$ be such that $w\in A_q$. Then for all functions $f\in L^p(\R^d,w)$ with compact support $K\subseteq \dom$, by H\"older's inequality one has
\[\|f\|_{L^p(\R^d,w)} \leq  \|f\|_{L^q(\R^d,w)} w(K)^{\frac{q-p}{q}}.\]
Therefore, it suffices to approximate such functions $f$ in the $L^q(\R^d,w)$ norm. To do so one can
use a standard argument (see \cite[Lemma 2.2]{LMV}) by using a mollifier with compact support.
\end{proof}

\begin{lemma}\label{lem:densityCcRd}
Let $p\in (1, \infty)$, $w\in A_p$ and $k\in \N_0$. Let $\dom = \R^d$ or a bounded $C^k$-domain or a special $C^k_c$-domain with $k\in \N_0\cup\{\infty\}$.
Then $C^k_c(\overline{\dom})\otimes X$ is dense in $W^{k,p}(\dom,w;X)$.
\end{lemma}
\begin{proof}
The case $\dom = \R^d$ follows from \cite[Lemma 3.5]{LMV}.
In all other situations, by localization, it suffices to consider $\dom = \R^d_+$. This case can be proved by combining the argument of \cite[Lemma 3.5]{LMV} with \cite[Theorem 1.8.5]{Krylov2008_book}.
\end{proof}

The density result \cite[Theorem 7.2]{Kufner80} can be extended to the vector-valued setting:
\begin{lemma}\label{lem:Kufner}
Let $p\in (1, \infty)$ and $\gamma\geq 0$. Let $\dom$ be a bounded $C^0$-domain or a special $C^0_c$-domain. Then $C^\infty_c(\overline{\dom};X)$ is dense in $W^{k,p}(\dom,w_{\gamma}^{\dom};X)$.
\end{lemma}

Next we will prove a density result for power weights of arbitrary order using functions with compact support in $\Omega$.
\begin{proposition}\label{prop:densenonAp}
Let $\gamma\in \R\setminus\{jp-1:j\in \N_1\}$. Let $\dom$ be a bounded $C^k$-domain or a special $C^k_c$-domain with $k\in \N_0\cup\{\infty\}$. Then $C^k_c(\dom;X)$ is dense in $W^{k,p}_{0}(\dom,w_{\gamma};X)$.
\end{proposition}
\begin{proof}
By a standard localization argument it suffices to consider $\dom = \R^d_+$.
Let $u\in W^{k,p}_{0}(\R^d_+,w_{\gamma};X)$.
By a simple truncation argument we may assume that $u$ is compactly supported on $\overline{\R^d_+}$.
To prove the required result we will truncate $u$ near the plane $x_1 = 0$. For this let $\phi\in C^\infty([0,\infty))$ be such that $\phi = 0$ on $[0,1/2]$ and $\phi = 1$ on $[1, \infty)$.
Let $\phi_n(x_1) = \phi(nx_1)$ and define $u_n(x) = \phi_n(x_1) u(x)$. We claim that $u_n\to u$ in $W^{k,p}(\R^d_+, w_{\gamma};X)$.
This will be proved below. Using the claim the proof can be finished as follows. It remains to show that each $u \in W^{k,p}(\R^{d}_{+},w_{\gamma};X)$ with compact support can be approximated by functions in $C^\infty_c(\R^d_+;X)$. For each $v \in W^{k,p}(\R^{d}_{+},w_{\gamma};X)$ with compact support $K$ it holds that
\begin{align}\label{eq:vcompactsuppequiv}
\|v\|_{W^{k,p}(\R^d_+,w_{\gamma};X)}\eqsim_{K,\gamma} \|v\|_{W^{k,p}(\R^d_{+};X)}.
\end{align}
Therefore, it suffices to approximate $u$ in the $W^{k,p}(\R^d_{+};X)$-norm. This can be done by extension by zero on $\overline{\R^{d}_{-}}$ followed by a standard mollifier argument (see \cite[Lemma 2.2]{LMV}).

To prove the claim for convenience we will only consider $d=1$. Since $\phi_n$ does not depend on $\tx$ the general case is similar.
Fix $m\in \{0,\ldots, k\}$. By Leibniz formula one has $(\phi_n u)^{(m)} = \sum_{i=0}^m c_{i,m} \phi_n^{(m-i)} u^{(i)}$.
By the dominated convergence theorem $\phi_n u^{(m)}\to u^{(m)}$ in $L^p(\R^d_+,w_{\gamma};X)$.
It remains to prove that $\phi_n^{(m-i)} u^{(i)}\to 0$ for $i\in \{0, \ldots, m-1\}$. By Corollary \ref{cor:lem:Hardy;Sobolev_emb}
\[u^{(i)}\in W^{m-i,p}_0(\R_+,w_{\gamma};X)\hookrightarrow L^{p}(\R_+,w_{\gamma - (m-i)p};X).\]
Now we find
\begin{align*}
\|\phi_n^{(m-i)} u^{(i)}\|_{L^p(\R_+,w_{\gamma};X)} &= \int_0^{1/n} n^{p(m-i)} |\phi^{(m-i)}(nx)|^p \|u^{(i)}(x)\|^p |x|^{\gamma} \ud x
\\ & \leq \|\phi^{(m-i)}\|_{\infty}^p \int_0^{1/n}  \|u^{(i)}(x)\|^p |x|^{\gamma-(m-i)p} \ud x.
\end{align*}
The latter tends to zero as $n\to \infty$ by the dominated convergence theorem.
\end{proof}

In the next result we prove a density result in real and complex interpolation spaces. It will be used as a technical ingredient in the proofs of Lemma \ref{lem:interp2} and Proposition \ref{prop:interp_W}.
\begin{lemma}\label{lem:CinftydenserealLW}
Let $p\in (1, \infty)$, $\gamma\in \R\setminus\{jp-1:j\in \N_0\}$, $q\in [1, \infty)$ and $k\in \N\setminus\{0\}$. Let $\dom$ be a bounded $C^k$-domain or a special $C^k_c$-domain with integer $k\geq 2$ or $k=\infty$ and let $\ell\in \{0,\ldots, k\}$. If $\theta\in (0,1)$ satisfies $k\theta<\frac{\gamma+1}{p}$ then the space $C^{k}_{c}(\dom;X)$ is dense in $(L^{p}(\dom,w_{\gamma};X),W^{\ell,p}(\dom,w_{\gamma};X))_{\theta,q}$ and $[L^{p}(\dom,w_{\gamma};X),W^{\ell,p}(\dom,w_{\gamma};X)]_{\theta}$.
\end{lemma}
\begin{proof}
First consider the real interpolation space. In the case $\gamma<-1$ the result follows from $W^{\ell,p}(\dom,w_{\gamma};X) = W^{\ell,p}_0(\dom,w_{\gamma};X)$, Proposition \ref{prop:densenonAp} and \cite[Theorem 1.6.2]{Tr1}.

In the case $\gamma\in \R\setminus\{jp-1:j\in \N_0\}$, it suffices to consider $\dom = \R^d_+$ by a localization argument. Write $Y_j = W^{j,p}(\R^d_+,w_{\gamma};X)$ for $j\in \N_0$. Since $Y_\ell \stackrel{d}{\hookrightarrow} (Y_0,Y_\ell)_{\theta,q}$ (see \cite[Theorem 1.6.2]{Tr1}), by Lemma's \ref{lem:densityCcRd} and \ref{lem:Kufner} it suffices to consider $u\in C^\infty_c(\overline{\R^d_+};X)$ and to approximate it by functions in $C^{\infty}_{c}(\R^d_+;X)$ in the $(Y_0,Y_\ell)_{\theta,q}$-norm. Moreover, note that
\[\|v\|_{(Y_0,Y_\ell)_{\theta,q}}\leq C \|v\|_{Y_0}^{1-\theta} \|v\|^{\theta}_{Y_\ell}\]
for all $v\in Y_\ell$ (see \cite[Theorem 1.3.3]{Tr1}). Therefore, it suffices to construct $v_n\in C^\infty_c(\R^d_+;X)$ such that $\|v_n - u\|_{Y_0}^{1-\theta} \|v_n-u\|_{Y_\ell}^{\theta}\to 0$ as $n\to \infty$. As in Proposition~\ref{prop:densenonAp}, letting $u_n = \phi_n u$, it suffices to show that $\|u_n - u\|_{Y_0}^{1-\theta} \|u_n-u\|_{Y_\ell}^{\theta}\to 0$ as $n\to \infty$.
Note that, for example in the case $d=1$, for one of the terms
\[\|(\phi_n-1) u\|_{L^p(\R_+,w_{\gamma};X)} \leq \|\phi\|_{\infty} \|u\|_{\infty} \Big(\int_0^{1/n} |x|^{\gamma} \ud x\Big)^{1/p}\leq \|\phi\|_{\infty} C_{\gamma,p} n^{-\frac{\gamma+1}{p}} \]
and similarly,
\[\|\phi_n^{(\ell)} u\|_{L^p(\R^d_+,w_{\gamma};X)} \leq n^{\ell-\frac{\gamma+1}{p}} \|\phi''\|_{\infty} \|u\|_{\infty}.\]
Now we obtain that there is a constant $C$ independent of $n$ such that
\[\|(\phi_n-1) u\|_{L^p(\R_+,w_{\gamma};X)}^{1-\theta} \|\phi_n^{(\ell)} u\|_{L^p(\R^d_+,w_{\gamma};X)}^{\theta}\leq C n^{-\frac{\gamma+1}{p}+\ell\theta}.\]
The latter tends to zero by the assumptions. The other terms can be treated with similar arguments.
Finally one can approximate each $u_n$ by using \eqref{eq:vcompactsuppequiv} and the arguments given there.

The density in the complex case follows from
\[(L^{p}(\dom,w_{\gamma};X),W^{\ell,p}(\dom,w_{\gamma};X))_{\theta,1} \stackrel{d}{\hookrightarrow}[L^{p}(\dom,w_{\gamma};X),W^{\ell,p}(\dom,w_{\gamma};X)]_{\theta}\]
(see \cite[Theorems 1.9.3 (c) $\&$ 1.10.3]{Tr1}).
\end{proof}

The next standard lemma gives a sufficient condition for a function to be in $W^{1,1}_{\rm loc}(\R^d;X)$ when it consists of two $W^{1,1}_{\rm loc}$-functions which are glued together. To prove the result one can reduce to the one-dimensional setting and use the formula $u(t)  - u(0) = \int_0^t u(s) \ud s$. We leave the details to the reader.
\begin{lemma}\label{lem:u1dcontW}
Let $u\in L^1_{\rm loc}(\R^d;X)$ be such that $u_+:=u|_{\R_+^d}\in W^{1,1}_{\rm loc}(\overline{\R_+^d};X)$ and $u_- := u|_{\R_-^d}\in W^{1,1}_{\rm loc}(\overline{\R^d_-};X)$. If $\tr(u_{+}) = \tr(u_{-})$, then $u\in W^{1,1}_{\rm loc}(\R^d;X)$ and
\[D_j u = \left\{
          \begin{array}{ll}
            D_j (u_+), & \hbox{on $\R_+^d$;} \\
            D_j (u_-), & \hbox{on $\R_-^d$;} \\
          \end{array}
        \right.
\]
\end{lemma}

Finally we will need the following simple density result in the $A_p$-case.
\begin{lemma}\label{lem:denseApbdr0}
Let $\dom$ be a bounded $C^k$-domain or a special $C^k_c$-domain with $k\in \N_0\cup\{\infty\}$. If $p\in (1, \infty)$, $w\in A_p$ and $k\in \N_0$, then
$E_0:W^{k,p}_0(\dom,w;X)\to W^{k,p}(\R^d,w;X)$ given by the extension by zero defines a bounded linear operator. Moreover, $W^{k,p}_0(\dom,w;X) = \overline{C^k_c(\dom;X)}$.
\end{lemma}
\begin{proof}
By localization it suffices to consider $\dom = \R^d_+$.
If $u\in W^{k,p}_0(\R_+^d,w;X)$, then, by Lemma \ref{lem:u1dcontW}, $E_0 u\in W^{k,1}_{\rm loc}(\R^d_+;X)$ and
\[D^{\alpha} E_0 u = E_0 D^{\alpha} u, \qquad |\alpha|\leq k.\]
In particular, this shows that $E_0$ is bounded.

For the final assertion let $u\in W^{k,p}_0(\R_+^d,w)$. By a truncation we may assume $u$ has bounded support. Take $\zeta\in C^\infty_c(\R_-^d)$ such that $\int \zeta \ud x = 1$ and set $\zeta_n(x) = n^d\zeta(nx)$. Then $\zeta_n*E_0u\to E_0 u$ in $W^{k,p}(\R^d,w;X)$ (see \cite[Lemma 2.2]{LMV}). Since $\zeta_n*E_0u\in C^\infty_c(\R_+^d;X)$, the result follows.
\end{proof}

\subsection{Interpolation\label{subsec:Interp}}

We continue with two interpolation inequalities. The first one is \cite[Lemma 5.8]{LMV}.
\begin{lemma}\label{lem:interp1}
Let $p\in (1, \infty)$ and let $w\in A_p$ be even. Let $\dom = \R^d$ or $\dom = \R_+^d$. Then for every $k\in \N\setminus\{0,1\}$, $j\in\{1,\ldots, k-1\}$ and $u\in W^{k,p}(\dom,w;X)$ we have
\[[u]_{W^{j,p}(\dom,w;X)} \leq C_{p,[w]_{A_p}} \|u\|_{L^p(\dom,w;X)}^{1-\frac{j}{k}} [u]_{W^{k,p}(\dom,w;X)}^{\frac{j}{k}}.\]
\end{lemma}
The above result holds on smooth domains as well provided we replace the homogeneous norms $[\cdot]_{W^{k,p}}$ by $\|\cdot\|_{W^{k,p}}$.
In order to extend this interpolation inequality to a class of non-$A_p$-weights, we will use the following pointwise multiplication mappings $M$ and $M^{-1}$.

Let $M:C^\infty_c(\R^{d}_+;X)\to C^\infty_c(\R^{d}_+;X)$ be given by $M u(x) = x_{1} u(x)$. By duality we obtain a mapping $M:\Distr'(\R^{d}_+;X)\to \Distr'(\R^{d}_+;X)$ as well. Similarly, we define $M^{-1}$ on $C^\infty_c(\R^{d}_+;X)$ and $\Distr'(\R^{d}_+;X)$.

\begin{lemma}\label{lem:isomM}
Let $p\in (1, \infty)$, $\gamma\in (-1, 2p-1)$ and $k\in \{0,1,2\}$. Then
$M:W^{k,p}(\R_+^d,w_{\gamma};X)\to W^{k,p}(\R_+^d,w_{\gamma-p};X)$ is bounded. Moreover,
$M:W^{k,p}_0(\R_+^d,w_{\gamma};X)\to W^{k,p}_0(\R_+^d,w_{\gamma-p};X)$ is an isomorphism.
\end{lemma}
\begin{proof}
Since the derivatives with respect to $x_i$ with $i\neq1$ commute with $M$, we only prove the result in the case $d=1$. Observe that $\|Mu\|_{L^{p}(\R_+,w_{\gamma-p};X)} = \|u\|_{L^{p}(\R_+^d,w_{\gamma};X)}$.
Moreover, by the product rule, we have $(Mu)^{(j)} = j u^{(j-1)} + M u^{(j)}$ for $j\in \{0, 1, 2\}$. Therefore,
\begin{align*}
\|Mu&\|_{W^{k,p}(\R_+,w_{\gamma-p};X)} = \|Mu\|_{L^{p}(\R_+,w_{\gamma-p};X)}  + \sum_{j=1}^k \|(M u)^{(j)}\|_{L^{p}(\R_+,w_{\gamma-p};X)}
\\ & \leq \|u\|_{L^{p}(\R_+,w_{\gamma};X)} + \sum_{j=1}^k \|j u^{(j-1)}\|_{L^{p}(\R_+,w_{\gamma-p};X)} + \|M u^{(j)}\|_{L^{p}(\R_+,w_{\gamma-p};X)}
\\ & \leq \|u\|_{L^{p}(\R_+,w_{\gamma};X)} + C \sum_{j=1}^k \|u^{(j)}\|_{L^{p}(\R_+,w_{\gamma};X)}
\\ & \leq (C+1) \|u\|_{W^{k,p}(\R_+,w_{\gamma};X)},
\end{align*}
where we applied Lemma \ref{lem:Hardy}. This proves the required boundedness of $M$.

By density of $C^\infty_c(\R_+;X)$ in $W^{k,p}_0(\R_+,w_{\gamma};X)$ (see Proposition \ref{prop:densenonAp} and Lemma \ref{lem:denseApbdr0}) it follows that $M:W^{k,p}_0(\R_+,w_{\gamma};X)\to W^{k,p}_0(\R_+,w_{\gamma-p};X)$ is bounded. It remains to prove boundedness of $M^{-1}:W^{k,p}_0(\R_+,w_{\gamma-p};X)\to W^{k,p}_0(\R_+,w_{\gamma};X)$. By Proposition \ref{prop:densenonAp} and Lemma \ref{lem:denseApbdr0} it suffices to prove the required estimate for $u\in C^\infty_c(\R_+;X)$. By the product rule, we have
$(M^{-1}u)^{(j)} = \sum_{i=0}^j c_{i,j} M^{-1+i-j} u^{(i)}$.
Therefore,
\begin{align*}
\|M^{-1}u\|_{W^{k,p}(\R_+,w_{\gamma};X)} & = \sum_{j=0}^k \|(M^{-1} u)^{(j)}\|_{L^{p}(\R_+,w_{\gamma};X)}
\\ & \leq C \sum_{j=0}^k \sum_{i=0}^j \|M^{-1-i} u^{(j-i)}\|_{L^{p}(\R_+,w_{\gamma};X)}
\\ & \leq C \|u\|_{W^{k,p}(\R_+,w_{\gamma-p};X)} + \sum_{j=0}^k \sum_{i=1}^j \|u^{(j-i)}\|_{L^p(\R_+,w_{\gamma-(i+1)p};X)}.
\end{align*}
Now it remains to observe that by Lemma \ref{lem:Hardy} (applied $i$ times)
\[\|u^{(j-i)}\|_{L^p(\R_+,w_{\gamma-(i+1)p};X)} \leq C\|u^{(j)}\|_{L^p(\R_+,w_{\gamma-p};X)}\leq C\|u\|_{W^{k,p}(\R_+,w_{\gamma-p};X)}.\]
\end{proof}

\begin{lemma}\label{lem:interp2}
Let $p\in (1, \infty)$ and $\gamma\in (-p-1,2p-1) \setminus \{ -1,p-1\}$.
Then for every $k\in \N\setminus\{0,1\}$, $j\in\{1,\ldots, k-1\}$ and $u\in W^{k,p}(\R^{d}_{+},w_{\gamma};X)$ we have
\[[u]_{W^{j,p}(\R_+^d,w_{\gamma};X)} \leq C_{\gamma,p,k} \|u\|_{L^p(\R_+^d,w_{\gamma};X)}^{1-\frac{j}{k}} [u]_{W^{k,p}(\R_+^d,w_{\gamma};X)}^{\frac{j}{k}}.\]
\end{lemma}
\begin{proof}
By an iteration argument as in \cite[Exercise 1.5.6]{Krylov2008_book}, it suffices to consider $k=2$ and $j=1$. Moreover, by a scaling involving $u(\lambda\cdot)$ it suffices to show that
\begin{align}\label{eq:lem:interp2}
\|u\|_{W^{1,p}(\R^d_+,w_{\gamma};X)} \leq C_{\gamma,p} \|u\|_{L^p(\R^d_+,w_{\gamma};X)}^{1/2} \|u\|_{W^{2,p}(\R^d_+,w_{\gamma};X)}^{1/2}.
\end{align}

The case $\gamma \in (-1,p-1)$ is contained in Lemma~\ref{lem:interp1}, where we actually do not need to proceed through \eqref{eq:lem:interp2}. So it remains to treat the case $\gamma \in (-p-1,-1) \cup (p-1,2p-1)$. By standard arguments (see e.g. \ \cite[Lemma 1.10.1]{Tr1}), it suffices to show that
\[
(L^{p}(\R^d_+,w_{\gamma};X),W^{2,p}(\R^d_+,w_{\gamma};X))_{\frac{1}{2},1} \hookrightarrow W^{1,p}(\R^d_+,w_{\gamma};X).
\]

We first assume that $\gamma \in (p-1,2p-1)$.
Using Lemma~\ref{lem:isomM} and real interpolation of operators, we see that $M$ is bounded as an operator
\[
(L^{p}(\R^d_+,w_{\gamma};X),W^{2,p}(\R^d_+,w_{\gamma};X))_{\frac{1}{2},1} \longra (L^{p}(\R^d_+,w_{\gamma-p};X),W^{2,p}(\R^d_+,w_{\gamma-p};X))_{\frac{1}{2},1}.
\]
By a combination of \cite[Lemma 1.10.1]{Tr1} and \eqref{eq:lem:interp2} for the case $\gamma \in (-1,p-1)$, the space on the right hand side is continuously embedded into $W^{1,p}(\R^{d}_{+},w_{\gamma-p};X)$. Therefore, $M$ is a bounded operator
\begin{align}\label{eq:lem:interp2;1}
M:(L^{p}(\R^d_+,w_{\gamma};X),W^{2,p}(\R^d_+,w_{\gamma};X))_{\frac{1}{2},1} \longra W^{1,p}(\R^{d}_{+},w_{\gamma-p};X).
\end{align}
From Lemma \ref{lem:CinftydenserealLW} and the fact that $MC^{\infty}_{c}(\R^{d}_{+};X) \subset C^{\infty}_{c}(\R^{d}_{+};X) \subset W^{1,p}_{0}(\R^{d}_{+},w_{\gamma-p};X)$, it follows that $M$ is a bounded operator
\[
M:(L^{p}(\R^d_+,w_{\gamma};X),W^{2,p}(\R^d_+,w_{\gamma};X))_{\frac{1}{2},1} \longra W^{1,p}_{0}(\R^{d}_{+},w_{\gamma-p};X).
\]
Combining this with Lemma~\ref{lem:isomM} we obtain \eqref{eq:lem:interp2}.

Next we assume $\gamma \in (-p-1,-1)$.
As $M$ in \eqref{eq:lem:interp2;1} in the previous case, $M^{-1}$ is a bounded operator
\[
M^{-1}:(L^{p}(\R^d_+,w_{\gamma};X),W^{2,p}_{0}(\R^d_+,w_{\gamma};X))_{\frac{1}{2},1} \longra W^{1,p}_{0}(\R^{d}_{+},w_{\gamma+p};X).
\]
Combining this with $W^{n,p}(\R^d_+,w_{\gamma};X) = W^{n,p}_{0}(\R^d_+,w_{\gamma};X)$ ($n \in \N$) and Lemma~\ref{lem:isomM} we obtain \eqref{eq:lem:interp2}.
\end{proof}

\begin{proposition}\label{prop:interp00_Ap}
Let $X$ be a UMD space, $p \in (1,\infty)$ and $\gamma \in (-1,p-1)$.
Let $\dom$ be a bounded $C^k$-domain or a special $C^k_c$-domain with $k\in \N_0\cup\{\infty\}$.
Then for every $j \in \{0,\ldots,k\}$ the following holds:
\[[L^p(\dom,w_{\gamma}^{\dom};X), W^{k,p}_0(\dom,w_{\gamma}^{\dom};X)]_{\frac{j}{k}} = W^{j,p}_0(\dom,w_{\gamma}^{\dom};X).\]
\end{proposition}
\begin{proof}
By a localization argument it suffices to consider the case $\dom = \R^d_+$.
The operator $\partial_{1}$ on $L^{p}(\R^{d}_{+},w_{\gamma};X)$ with domain $D(\partial_{1}) = W^{1,p}_{0}(\R_{+},w_{\gamma};L^{p}(\R^{d-1};X))$ has a bounded $H^{\infty}$-calculus with $\omega_{H^{\infty}}(\partial_{1})=\frac{\pi}{2}$ by \cite[Theorem 6.8]{LMV}. Moreover, $D((\partial_{1})^{n}) = W^{n,p}_{0}(\R_{+},w_{\gamma};L^{p}(\R^{d-1};X))$ for all $n \in \N$.
For the operator $\Delta_{d-1}$ on $L^p(\R_+^d,w_{\gamma};L^{p}(\R^{d-1};X))$, defined by
\[
D(\Delta_{d-1}) := L^p(\R_+,w_{\gamma},W^{2,p}(\R^{d-1};X)),
\quad \Delta_{d-1}u := \sum_{k=2}^d \partial_k^2u,
\]
it holds that $-\Delta_{d-1}$ a bounded $H^\infty$-calculus with $\omega_{H^\infty}(-\Delta_{d-1}) = 0$.
Moreover, $D((-\Delta_{d-1})^{n/2}) = L^{p}(\R_{+},w_{\gamma};W^{n,p}(\R^{d-1};X))$ for all $n \in \N$.
It follows that $(1+\partial_{t})^{k}$ with $D((1+\partial_{1})^{k}) = W^{k,p}_{0}(\R_{+},w_{\gamma};X)$ is sectorial having bounded imaginary powers with angle $\leq \pi/2$ and that $(1-\Delta_{d-1})^{k/2}$ with $D((1-\Delta_{d-1})^{k/2}) = L^{p}(\R_{+},w_{\gamma};W^{k,p}(\R^{d-1};X))$ is sectorial having bounded imaginary powers with angle $0$. By a combination of Proposition~\ref{prop:BIP} and \cite[Lemma 9.5]{EPS03},
\begin{align*}
& [L^{p}(\R^{d}_{+},w_{\gamma};X),W^{k,p}_{0}(\R_{+},w_{\gamma};L^{p}(\R^{d-1};X)) \cap L^{p}(\R_{+},w_{\gamma};W^{k,p}(\R^{d-1};X))]_{\frac{j}{k}} \\
&\qquad\qquad = [L^{p}(\R^{d}_{+},w_{\gamma};X),D((1+\partial_{1})^{k})  \cap D((1-\Delta_{d-1})^{k/2})]_{\frac{j}{k}} \\
&\qquad\qquad = D((1+\partial_{1})^{j}) \cap D((1-\Delta_{d-1})^{j/2}) \\
&\qquad\qquad = W^{j,p}_{0}(\R_{+},w_{\gamma};L^{p}(\R^{d-1};X)) \cap L^{p}(\R_{+},w_{\gamma};W^{j,p}(\R^{d-1};X)).
\end{align*}
Now the result follows from the following intersection representation for $n\in \N$:
\[
W^{n,p}_{0}(\R_{+}^{d},w_{\gamma};X) =
W^{n,p}_{0}(\R_{+},w_{\gamma};L^{p}(\R^{d-1};X)) \cap L^{p}(\R_{+},w_{\gamma};W^{n,p}(\R^{d-1};X)).
\]
Here $\hookrightarrow$ is clear. To prove the converse let $u$ be in the intersection space. We first claim that $u\in W^{n,p}(\R_{+}^{d},w_{\gamma};X)$. Using a suitable extension operator it suffices to show the result with $\R_+$ and $\R^d_+$ replaced by $\R$ and $\R^d$ respectively. Now the claim follows from Proposition \ref{prop:intersectionrepr}.
To prove $u\in W^{n,p}_0(\R_{+}^{d},w_{\gamma};X)$ let $|\alpha|\leq n-1$ and write $\alpha = (\alpha_1, \tilde{\alpha})$. It remains to show $\tr (D^\alpha u) = 0$. By assumption and the claim
$D^{\alpha_1} u \in W^{1,p}_{0}(\R_{+},w_{\gamma};L^{p}(\R^{d-1};X))$ and $D^{\alpha_1} u \in W^{1,p}(\R_+,w_{\gamma};W^{n-1-\alpha_1}(\R^{d-1};X))$. It follows that $D^{\alpha_1} u \in W^{1,p}_0(\R_+,w_{\gamma};W^{n-1-\alpha_1}(\R^{d-1};X))$ and therefore, we obtain $D^{\alpha} u \in W^{1,p}_{0}(\R_{+};L^p(\R^{d-1};X))$ as required.
\end{proof}

Now we extend the last identity to the non-$A_p$ setting for $j=1$ and $k=2$.
\begin{proposition}\label{prop:interp00}
Let $X$ be a UMD space, $p \in (1,\infty)$ $\gamma\in (-p-1,2p-1)\setminus\{-1,p-1\}$. Let $\dom$ be a bounded $C^2$-domain or a special $C^2_c$-domain. Then the complex interpolation space satisfies
\[[L^p(\dom,w_{\gamma};X), W^{2,p}_0(\dom,w_{\gamma};X)]_{\frac12} = W^{1,p}_0(\dom,w_{\gamma};X).\]
\end{proposition}
\begin{proof}
The case $\gamma\in (-1,p-1)$ is contained in Proposition~\ref{prop:interp00_Ap}.
For  the case $\gamma\in (jp-1, (j+1)p-1)$ with $j=1$ or $j=-1$ we reduce to the previous case. By a localization argument it suffices to consider $\dom = \R^d_+$.  By Lemma \ref{lem:isomM} and since the complex interpolation method is exact we deduce
\begin{align*}
[L^p(\R_+^d,w_{\gamma};X),&W^{2,p}_0(\R_+^d,w_{\gamma};X)]_{\frac12} \\ & = [M^{-j} L^p(\R_+^d,w_{\gamma-jp};X),M^{-j} W^{2,p}_0(\R_+^d,w_{\gamma-jp};X)]_{\frac12}
\\ & = M^{-j} [L^p(\R_+^d,w_{\gamma-jp};X), W^{2,p}_0(\R_+^d,w_{\gamma-jp};X)]_{\frac12}
\\ & = M^{-j} W^{1,p}_0(\R_+^d,w_{\gamma-jp};X) = W^{1,p}_0(\R_+^d,w_{\gamma};X). \qedhere
\end{align*}
\end{proof}

Next we prove a version of Proposition \ref{prop:interp00} without boundary conditions by reducing to the case with boundary conditions.
\begin{proposition}\label{prop:interp_W}
Let $X$ be a UMD space, $p \in (1,\infty)$, $\gamma\in (-p-1,2p-1)\setminus\{-1,p-1\}$. Let $\dom$ be bounded $C^2$-domain or a special $C^2_c$-domain. Then the complex interpolation space satisfies
\[[L^p(\dom,w_{\gamma};X), W^{2,p}(\dom,w_{\gamma};X)]_{\frac12} = W^{1,p}(\dom,w_{\gamma};X).\]
\end{proposition}
\begin{proof}
By a localization argument it suffices to consider $\dom = \R^d_+$.  The case $\gamma\in (-1,p-1)$ follows from \cite[Propositions 5.5 $\&$ 5.6]{LMV} and the case $\gamma \in (-p-1,-1)$ follows from Proposition~\ref{prop:interp00}.

It remains to establish the case $\gamma \in (p-1,2p-1)$. The inclusion $\hookleftarrow$ follows from Proposition \ref{prop:interp00} and $W^{1,p}_0(\R^{d}_+,w_{\gamma};X) = W^{1,p}(\R^{d}_+,w_{\gamma};X)$. To prove $\hookrightarrow$, by Lemma \ref{lem:CinftydenserealLW} it suffices to show that
\[\|u\|_{W^{1,p}(\R^{d}_+,w_{\gamma};X)}\leq C \|u\|_{[L^p(\R^{d}_+,w_{\gamma};X), W^{2,p}(\R^{d}_+,w_{\gamma};X)]_{\frac12}}, \ \ \  u\in C^\infty_c(\R^d_+;X).\]
Since $W^{1,p}_0(\R^{d}_+,w_{\gamma};X) = W^{1,p}(\R^{d}_+,w_{\gamma};X)$, using Lemma \ref{lem:isomM} twice and the result for the $A_{p}$-case already proved, we obtain
\begin{align*}
\|u\|_{W^{1,p}(\R^{d}_+,w_{\gamma};X)}\lesssim \|M u\|_{W^{1,p}(\R^{d}_+,w_{\gamma-p};X)} & \lesssim  \|Mu\|_{[L^p(\R^{d}_+,w_{\gamma-p};X), W^{2,p}(\R^{d}_+,w_{\gamma-p};X)]_{\frac12}}\\ & \lesssim \|u\|_{[L^p(\R^{d}_+,w_{\gamma};X), W^{2,p}(\R^{d}_+,w_{\gamma};X)]_{\frac12}}.
\end{align*}
\end{proof}

Next we turn to a different type of interpolation result where we interpolate all the possible parameters including the target spaces. There are many existing results in this direction (see \cite[Proposition 3.7]{MeyVerpoint} and \cite[Proposition 7.1]{PSW} and references therein). In the unweighted case it has recently also appeared in \cite[Theorem 4.5.5]{Amann19} using a different argument. Since our (independent) proof is of interest we present the details.
\begin{theorem}\label{thm:Hspwinterp}
Let $X_j$ be a UMD space, $p_j\in (1, \infty)$, $w_j\in A_{p_j}$ and $s_j\in \R$ for $j\in \{0,1\}$. Let $\theta\in (0,1)$ and set $X_{\theta} = [X_0, X_1]_{\theta}$, $\frac{1}{p} = \frac{1-\theta}{p_0} + \frac{\theta}{p_1}$, $w = w_0^{(1-\theta)p/p_0} w_1^{\theta p/p_1}$ and $s = (1-\theta)s_0 + \theta s_1$. Then
\[[H^{s_0,p_0}(\R^d,w_0;X_0), H^{s_1,p_1}(\R^d,w_1;X_1)]_{\theta} = H^{s,p}(\R^d,w;X_{\theta}).\]
\end{theorem}
Observe that $w\in A_p$ by \cite[Exercise 9.1.5]{GrafakosM}. The proof of the theorem will be given below.

As a corollary of Proposition~\ref{prop:interp_W} and Theorem~\ref{thm:Hspwinterp} we obtain (using the identification from Proposition~\ref{prop:HWduality}) the following mixed-derivative theorem:
\begin{corollary}\label{cor:mixed-derivative}
Let $X$ be a UMD space, $p \in (1,\infty)$, $\gamma\in (-p-1,2p-1)\setminus\{-1,p-1\}$ and $d\geq 2$. Then
\begin{align*}
& W^{2,p}(\R^{d-1};L^{p}(\R_{+},w_{\gamma};X)) \cap L^{p}(\R^{d-1};W^{2,p}(\R_{+},w_{\gamma};X)) \\
& \qquad\qquad\qquad\qquad \hookrightarrow W^{1,p}(\R^{d-1},W^{1,p}(\R_{+},w_{\gamma};X)).
\end{align*}
\end{corollary}
\begin{proof}
By Proposition \ref{prop:HWduality}, Theorem~\ref{thm:Hspwinterp}, and Proposition~\ref{prop:interp_W},
\begin{align*}
&L^{p}(\R^{d-1};W^{2,p}(\R_{+},w_{\gamma};X)) \cap W^{2,p}(\R^{d-1};L^{p}(\R_{+},w_{\gamma};X)) \\
&\qquad\qquad\qquad =  H^{0,p}(\R^{d-1};W^{2,p}(\R_{+},w_{\gamma};X)) \cap H^{2,p}(\R^{d-1};L^{p}(\R_{+},w_{\gamma};X))  \\
&\qquad\qquad\qquad \hookrightarrow [H^{0,p}(\R^{d-1};W^{2,p}(\R_{+},w_{\gamma};X)), H^{2,p}(\R^{d-1};L^{p}(\R_{+},w_{\gamma};X))]_{\frac{1}{2}} \\
&\qquad\qquad\qquad = H^{1,p}(\R^{d-1};[W^{2,p}(\R_{+},w_{\gamma};X),L^{p}(\R_{+},w_{\gamma};X)]_{\frac{1}{2}}) \\
&\qquad\qquad\qquad = W^{1,p}(\R^{d-1};W^{1,p}(\R_{+},w_{\gamma};X)). \qedhere
\end{align*}
\end{proof}

For the proof of Theorem~\ref{thm:Hspwinterp} we need two preliminary results. The first result follows as in \cite[Theorems 1.18.4 $\&$ 1.18.5]{Tr1}.
\begin{proposition}\label{prop:weightedLp}
Let $(A,\A,\mu)$ be a measure space. Let $X_j$ be a Banach space, $p_j\in (1, \infty)$ and $w_j:S\to (0,\infty)$ measurable for $j\in \{0,1\}$. Let $\theta\in (0,1)$ and set $X_{\theta} = [X_0, X_1]_{\theta}$, $\frac{1}{p} = \frac{1-\theta}{p_0} + \frac{\theta}{p_1}$, $w = w_0^{(1-\theta)p/p_0} w_1^{\theta p/p_1}$ and $s = (1-\theta)s_0 + \theta s_1$. Then
\[[L^{p_0}(A,w_0;X_0), L^{p_1}(A,w_1;X_1)]_{\theta} = L^{p}(\R^d,w;X_{\theta}).\]
\end{proposition}

For the next result we need to introduce some notation. Let $(\varepsilon_k)_{k\geq 0}$ be a Rademacher sequence on a probability space $\Omega$. Let $\sigma:\N\to (0,\infty)$ be a weight function, $p\in (1, \infty)$ and let $\Rad^{\sigma,p}(X)$ denote the space of all sequences $(x_k)_{k\geq 0}$ in $X$ for which
\[\big\|(x_k)_{k\geq 0}\big\|_{\Rad^{\sigma,p}(X)}:= \sup_{n\geq 1} \Big\|\sum_{k=0}^n \varepsilon_k \sigma(k) x_k\Big\|_{L^p(\Omega;X)}<\infty.\]
The above space is $p$-independent and the norms for different values of $p$ are equivalent (see \cite[Proposition 6.3.1]{HNVW2}). If $\sigma \equiv 1$, we write $\Rad^p(X):=\Rad^{\sigma,p}(X)$. Clearly $(x_k)_{k\geq 0}\mapsto (\sigma(k) x_k)_{k\geq 0}$ defines an isometric isomorphism from $\Rad^{\sigma,p}(X)$ onto $\Rad^p(X)$. By \cite[Corollary 6.4.12]{HNVW2}, if $X$ does not contain a copy isomorphic to $c_0$ (which is the case for UMD spaces), then $(x_k)_{k\geq 0}$ in $\Rad^{\sigma,p}(X)$ implies that $\sum_{k\geq 0} \varepsilon_k \sigma(k) x_k$ converges in $L^p(\Omega;X)$ and in this case
\[\big\|(x_k)_{k\geq 0}\big\|_{\Rad^{\sigma,p}(X)} =  \Big\|\sum_{k\geq 0} \varepsilon_k \sigma(k) x_k\Big\|_{L^p(\Omega;X)}.\]
Interpolation of the unweighted spaces
\begin{equation}\label{eq:Radunweightedinter}
[\Rad^{p_0}(X_0), \Rad^{p_1}(X_1)]_{\theta} = \Rad^p(X_{\theta})
\end{equation}
holds if $X_0$ and $X_1$ are $K$-convex spaces (see \cite[Theorem 7.4.16]{HNVW2} for details). In particular, UMD spaces are $K$-convex (see \cite[Proposition 4.3.10]{HNVW1}).  We need the following weighted version of complex interpolation of $\Rad$-spaces.
\begin{proposition}\label{lem:weightedRad}
Let $X_j$ be a $K$-convex space, $\sigma_{j}:\N\to (0,\infty)$ and let $p_j\in (1, \infty)$ for $j\in \{0,1\}$. Let $\theta\in (0,1)$ and set $X_{\theta} = [X_0, X_1]_{\theta}$, $\frac1p = \frac{1-\theta}{p_0} + \frac{\theta}{p_1}$ and $\sigma = \sigma_0^{1-\theta} \sigma_1^{\theta}$.
Then
\[[\Rad^{\sigma_0,p_0}(X_0), \Rad^{\sigma_1,p_1}(X_1)]_{\theta} = \Rad^{\sigma,p}(X_\theta).\]
\end{proposition}

\begin{proof}
We use the same method as in \cite[1.18.5]{Tr1}. Let
\[T:\F_-(\Rad^{\sigma_0,p_0}(X_0), \Rad^{\sigma_1,p_1}(X_1),0)\to \F_-(\Rad^{p_0}(X_0), \Rad^{p_1}(X_1),0)\] be defined by
\[T f(k,z) = \sigma_0(k)^{1-z} \sigma_1(k)^{z} f(k,z).\]
Then $f\mapsto Tf(\cdot,\theta)$ is an isomorphism from
$[\Rad^{\sigma_0,p_0}(X_0), \Rad^{\sigma_1,p_1}(X_1)]_{\theta}$ onto $[\Rad^{p_0}(X_0), \Rad^{p_1}(X_1)]_{\theta} = \Rad^{p}(X_{\theta})$, where we used \eqref{eq:Radunweightedinter} in the last step.
\end{proof}

\begin{proof}[Proof of Theorem \ref{thm:Hspwinterp}]
Set $Y_j = L^{p_j}(\R^d,w_j,X_j))$ for $j\in \{0,1\}$ and let $Y_{\theta} = L^{p}(\R^d,w,X_{\theta}))$. Then by Proposition \ref{prop:weightedLp} $Y_{\theta} = [Y_0,Y_1]_{\theta}$.  Let $\sigma_j(n) = 2^{s_j n}$ and let $(\varphi_k)_{k\geq 0}$ be a smooth Littlewood-Paley sequence as in \cite[Section 2.2]{MeyVerpoint} and let $\phi_{-1} = 0$. By \cite[Proposition 3.2]{MeyVerpoint} and \cite[Theorem 6.2.4]{HNVW2} we have $f\in H^{s_j,p_j}(\R^d,w_j;X_j)$ if and only if $(\varphi_k*f)_{k\geq 0}\in \Rad^{\sigma_j,p_j}(Y_j)$ and in this case
\begin{align}\label{eq:LPineqH}
\|(\varphi_k*f)_{k\geq 0}\|_{\Rad^{\sigma_j,p_j}(Y_j)}\eqsim \|f\|_{H^{s_j,p_j}(\R^d,w_j;X_j)}
\end{align}
with implicit constants only depending on $p_j, X_j,s_j, [w_j]_{A_{p_j}}$. Now to reduce the statement to Proposition \ref{lem:weightedRad} we use a retraction-coretraction argument (see \cite[Theorem 1.2.4]{Tr1} and \cite[Lemma 5.3]{LMV}).
Let $\psi_n = \sum_{k=n-1}^{n+1}\phi_k$ for $n\geq 0$, and let $\psi_{-1} = 0$. Then $\wh{\psi}_k = 1$ on $\supp(\wh{\phi}_k)$ for all $k\geq 0$, and $\supp(\wh{\psi}_0) \subseteq \{\xi:|\xi|\leq 2\}$ and $\supp(\wh{\psi}_k) \subseteq \{\xi:2^{k-2}\leq |\xi|\leq 2^{k+1}\}$ for $k\geq 1$.
Let
$R:\Rad^{\sigma_j,p_j}(Y_j)\to H^{s_j,p_j}(\R^d,w_j;X_j)$ be defined by $R (f_\ell)_{\ell\geq 0} = \sum_{\ell\geq 0} \psi_\ell* f_\ell$ and let $S:H^{s_j,p_j}(\R^d,w_j;X_j)\to \Rad^{p_j,\sigma_j}(Y_j)$ be given by $S f = (\varphi_k*f)_{k\geq 0}$. The boundedness of $S$ follows from \eqref{eq:LPineqH}. We claim that $R$ is bounded and this will be explained below. By the special choice of $\psi_k$ we have $RS = I$. Therefore, the retraction-coretraction argument applies and the interpolation result follows.

To prove claim let $E_j = L^{p_j}(\Omega;Y_j))$. Due to \eqref{eq:LPineqH} and by density it suffices to show that, for all finitely-nonzero sequences $(f_\ell)_{\ell\geq 0}$ in $Y_j$ and all $n\geq 0$,
\begin{equation}\label{eq:bddRest2sj}
\Big\| \sum_{k=0}^n \varepsilon_k 2^{s_j k} \varphi_k*\sum_{\ell\geq 0} \psi_\ell*f_\ell\Big\|_{E_j}\leq C \Big\|\sum_{k\geq 0} \varepsilon_k 2^{s_j k} f_k\Big\|_{E_j}.
\end{equation}
Below, for convenience of notation, we view sequences on $\N$ as sequences on $\Z$ through extension by zero. Under this convention, by the Fourier support properties of $(\varphi_k)_{k}$ and the $R$-boundedness of $\{\varphi_k*:k\geq 0\}$ (see \cite[Lemma 4.1]{MeyVerpoint}) and the implied $R$-boundedness of $\{\psi_{k}* : k \geq 0\}$, we have
\begin{align*}
\Big\| \sum_{k=0}^n \varepsilon_k 2^{s_j k} \varphi_k*\sum_{\ell\geq 0} \psi_\ell*f_\ell\Big\|_{E_j}
& \leq \sum_{j=-2}^{2}  \Big\| \sum_{k=0}^n \varepsilon_k 2^{s_j k} \varphi_k*\psi_{k+j}* f_{k+j}\Big\|_{E_j}
\\ & \lesssim \sum_{j=-2}^{2}  \Big\| \sum_{k=0}^n \varepsilon_k 2^{s_j k}f_{k+j}\Big\|_{E_j}
\\ & \lesssim \Big\|\sum_{k\geq 0} \varepsilon_k 2^{s_j k} f_k\Big\|_{E_j},
\end{align*}
where in the last step we used the contraction principle (see \cite[Proposition 3.24]{HNVW1}).
\end{proof}

\section{$\DD$ on $\R^d_+$ in the $A_p$-setting\label{sec:heatApsetting}}

Let $p\in (1, \infty)$ and $w\in A_p(\R^{d})$.
We consider the Dirichlet Laplacian $\DD$ on $L^p(\R^d_+,w;X)$, defined by
\[
D(\DD) := W^{2,p}_{\Dir}(\R^d_+,w;X), \qquad \DD u := \sum_{j=1}^{d}\partial_{j}^{2}u.
\]

Let $G_z:\R^d\to \R$ denote the standard heat kernel on $\R^d$:
\[G_z(x) = \frac{1}{(4\pi z)^{d/2}}e^{-|x|^2/(4z)}, \qquad z\in C_+.\]
It is well-known that $|G_z*f|\leq {\cos^{-d/2}(\arg(z))}Mf$, where $M$ denotes the Hardy-Littlewood maximal function (see \cite[Section 8.2]{HNVW2}). Therefore, $f\mapsto G_z*f$ is bounded on $L^p(\R^d,w;X)$ for any $w\in A_p$.

Define $T(z):L^p(\R^d_+,w;X)\to L^p(\R^d_+,w;X)$ by
\begin{align}\label{eq:heatsemi}
 T(z) f(x) := H_z * f(x):=\int_{\R^d_+} H_z(x,y) f(y)\ud y = \int_{\R^d} G_z(x-y) \overline{f}(y)\ud y, \ \ z\in \C_+,
\end{align}
with $\overline{f}(y) = \sign(y_1) f(|y_1|, \tilde{y})$ and
\begin{align}\label{eq:heatsemikernel}
H_{z}(x,y) = G_z(x_1-y_1,\tilde{x}-\tilde{y}) - G_z(x_1+y_1, \tilde{x}-\tilde{y}), \ \ x,y \in \R^{d}_{+}.
\end{align}
By the properties of $G_z$, the operator $T(z)$ is bounded on $L^p(\R^d_+,w;X)$ for any $w\in A_p$ with $\|T(z)\|\leq \|M\|_{\mathcal{B}(L^p(w))} {\cos^{-d/2}(\arg(z))}$. In Theorem \ref{thm:domainAp case} we will show that $T(z)$ is an analytic $C_0$-semigroup with generator $\DD$. Moreover, in case $X$ is a UMD space we characterize $D(\DD)$ and prove that $\DD$ is a sectorial operator with a bounded $H^\infty$-calculus of angle zero.

Recall that a weight $w$ is called even if $w(-x_1, \tx) = w(x_1, \tx)$ for $x_1>0$ and $\tx\in \R^{d-1}$.

The next result is the main result of this section on the functional calculus of $-\DD$ on $L^p$-spaces with $A_p$-weights. The result on the whole of $\R^d$ is well-known to experts, but seems not to have appeared anywhere. By a standard reflection argument we deduce the result on $\R^d_+$. It can be seen as a warm-up for Theorem \ref{thm:domain} where weights outside the $A_p$-class are considered.
\begin{theorem}\label{thm:domainAp case}
Let $X$ be a UMD space. Let $p\in (1, \infty)$ and let $w\in A_p$ be even.
Then the following assertions hold:
\begin{enumerate}[$(1)$]
\item\label{it:domainAp}  $-\DD$ is a sectorial operator with $\omega(-\DD) = 0$,
$D(\DD) = W^{2,p}_{\Dir}(\R^d_+,w;X)$ with equivalent norms,
the analytic $C_0$-semigroup $(e^{z\DD})_{z\in \C_+}$ is uniformly bounded on any sector $\Sigma_{\omega}$ with $\omega\in (0,\pi/2)$ and
    \[e^{z\DD} f= T(z)f, \ \ \ z\in \C_+.\]
\item\label{it:HinftyAp} For all $\lambda\geq 0$, $\lambda-\DD$ has a bounded $H^\infty$-calculus with $\omega_{H^\infty}(\lambda-\DD) = 0$.
\end{enumerate}
Moreover, all the implicit constants only depend on $X$, $p$, $d$ and $[w]_{A_{p}}$.
\end{theorem}

For the proof we use a simple lemma on odd extensions. For $u\in L^p(\R^d,w;X)$, the functions $\overline{u}$ and $\Eo u$ denote the {\em odd extension of $u$}:
\[ \overline{u}(-x_1, \tx) = \Eo u(-x_1, \tx)= -u(x_1, \tx)\ \text{for $x_1>0$ and $\tx\in \R^{d-1}$.}\]
For $k\in \N_0$ let $W^{k,p}_{\odd}(\R^d,w;X)$ denote the closed subspace of all odd functions in $W^{k,p}(\R^d,w;X)$.

\begin{lemma}\label{lem:tracech}
Let $p\in (1, \infty)$ and let $w\in A_p$ be even. Let $k\in \{0,1,2\}$.
Then $\Eo:W^{k,p}_{\Dir}(\R^d_+,w;X)\to W^{k,p}_{\odd}(\R^d,w;X)$ is an isomorphism and
\[\|u\|_{W^{k,p}(\R^d,w;X)}\leq \|\Eo u\|_{W^{k,p}(\R^d,w;X)}\leq 2^{1/p} \|u\|_{W^{k,p}(\R^d,w;X)}. \]
Moreover, $\{u\in C^\infty_c(\overline{\R^d_+}): u(0,\cdot)=0\}\otimes X$ is dense in $W^{k,p}_{\Dir}(\R^d_+,w;X)$.
\end{lemma}
\begin{proof}
The case $k=0$ is easy, so let us assume $k\in \{1, 2\}$. For $u\in W^{k,p}_{\Dir}(\R^d_+,w;X)$ one has
\begin{equation}\label{eq:Dalphaovu}
 D^{\alpha} \overline{u}(x) = (\sign(x_1))^{|\alpha_1|+1} (D^{\alpha} u)(|x_1|, \tx), \ \ \ |\alpha|\leq k.
\end{equation}
Indeed, this follows from Lemmas \ref{lem:u1dcontW} and \ref{lem:traceAp}.

From \eqref{eq:Dalphaovu} we find that $\overline{u}\in W^{2,p}(\R^d,w;X)$ and that the stated estimates hold.

If $\overline{u}\in W^{k,p}(\R^d,w)$, then
by Lemma \ref{lem:densityCcRd} we can find $u_n\in C^\infty_c(\R^d)\otimes X$ such that $u_n \to \overline{u}$ in $W^{2,p}(\R^d,w;X)$. Then also $u_n(-\cdot, \cdot)\to \overline{u}$ in $W^{2,p}(\R^d,w;X)$. Now $v_n:=(u_n + u_n(-\cdot,\cdot))/2$ satisfies $v_n\in C^\infty_c(\R^d;X)$ and $v_n(0,\cdot) = 0$ and $v_n\to u$ in $W^{2,p}(\R^d_+,w;X)$. Since $\tr(v_n) = 0$ the continuity of the trace implies $\tr(u) = 0$ as well. This part of the proof also implies the desired density result.
\end{proof}

\begin{proof}[Proof of Theorem \ref{thm:domainAp case}]
Let us first consider the result on $\R^d$. Then $-\Delta$ with $D(\Delta) = W^{2,p}(\R^d,w;X)$ is a closed operator which is sectorial of angle zero (see \cite[Theorem 5.1]{GV17} and Proposition \ref{prop:weightedmihlin}).
Moreover, by Proposition \ref{prop:weightedmihlin} and \cite[Theorem 10.2.25]{HNVW2}), one has that $-\Delta$ has a bounded $H^\infty$-calculus with $\omega_{H^\infty}(-\Delta) = 0$. Moreover, by Remark \ref{rem:Hinftytranslate} the same holds for $\lambda-\Delta$.
Now the half space case follows by a well-known reflection argument, which we partly include here for completeness.

Since $\Eo (\DD f) =  (\Delta\Eo f)$, $\Eo:W^{2,p}_{\Dir}(\R^d_+,w;X)\to W^{2,p}_{\odd}(\R^d,w;X)$ is an isomorphism (see Lemma \ref{lem:tracech}), $\Delta:W^{2,p}_{\odd}(\R^d,w;X)\to L^p_{\odd}(\R^d,w;X)$, and $D(\Delta|_{L^p_{\odd}(\R^d,w;X)}) = W^{2,p}_{\odd}(\R^d,w;X)$,
one has
\[\rho(\Delta)  \subseteq \rho(\DD), \ \ R(\lambda,\DD)f =  (R(\lambda,\Delta)\Eo f)|_{\R^d_+}\ \ \text{and} \ \ D(\DD) = W^{2,p}_{\Dir}(\R^d_+,w;X).\]
All the statements now follow.
\end{proof}

\begin{corollary}[Laplace equation]\label{cor:ellipticregAp}
Let $X$ be a UMD space. Let $p\in (1, \infty)$ and let $w\in A_p$ be even.
For all $u \in W^{2,p}_{\Dir}(\R^d_+,w;X)$ there holds the estimates
\begin{align}\label{eq:cor:ellipticregAp:eq_hom_part_Sobolev}
[u]_{W^{2,p}(\R^{d}_{+},w;X)} \eqsim_{X,p,d,w} \| \Delta u \|_{L^p(\R^d_+,w;X)}.
\end{align}
Furthermore, for every $f\in L^p(\R^d_+,w;X)$ and $\lambda>0$ there exists a unique $u\in W^{2,p}_{\Dir}(\R^d_+,w;X)$ such that $\lambda u - \DD u= f$ and
\begin{align}\label{eq:aprioridAp}
\sum_{|\alpha|\leq 2} |\lambda|^{1-\frac12|\alpha|} \|D^{\alpha} u\|_{L^p(\R^d_+,w;X)} \lesssim_{X,p,d,w} \|f\|_{L^p(\R^d_+,w;X)}.
\end{align}
\end{corollary}
\begin{proof}
We first prove \eqref{eq:cor:ellipticregAp:eq_hom_part_Sobolev}.
Let $u \in W^{2,p}_{\Dir}(\R^d_+,w;X)$. For $r > 0$ we put $u_{r}:=u(r\,\cdot\,)$ and $w_{r}:=w(r\,\cdot\,)$. Then $w_{r} \in A_{p}$ with $[w_{r}]_{A_{p}} = [w]_{A_{p}}$.
So we can apply Theorem~\ref{thm:domainAp case} with $w_{r}$ instead of $w$ to obtain
\begin{align*}
\sum_{|\alpha| \le 2}r^{|\alpha|-\frac{d}{p}}\|\partial^{\alpha}u\|_{L^p(\R^d_+,w;X)}
&=  \|u_{r}\|_{W^{2,p}(\R^{d}_{+},w_{r};X)} \\
&\eqsim_{X,p,d,w} \|u_{r}\|_{L^{p}(\R^{d}_{+},w_{r};X)} + \| \Delta u_{r} \|_{L^p(\R^d_+,w_{r};X)} \\
&=  r^{-\frac{d}{p}}\|u\|_{L^{p}(\R^{d}_{+},w;X)} + r^{2-\frac{d}{p}}\|\Delta u\|_{W^{2,p}(\R^{d}_{+},w_{r};X)}.
\end{align*}
Dividing by $r^{2-\frac{d}{p}}$ and taking the limit $r \to \infty$ gives \eqref{eq:cor:ellipticregAp:eq_hom_part_Sobolev}.

The existence and uniqueness in the second claim follow from the sectoriality in Theorem~\ref{thm:domainAp case}. Moreover, together with \eqref{eq:cor:ellipticregAp:eq_hom_part_Sobolev}, the sectoriality yields the estimates for $|\alpha|=0$ and $|\alpha|=2$ in \eqref{eq:aprioridAp}. The case $|\alpha|=1$  subsequently follows from Lemma~\ref{lem:interp1}.
\end{proof}

\begin{corollary}[Heat equation]\label{cor:weightedmaxAp}
Let $X$ be a UMD space. Let $p,q\in (1, \infty)$, $v\in A_q(\R)$, $w\in A_p(\R^d)$ and assume $w$ is even. Let $J \in \{\R,\R_+\}$. Then the following assertions hold:
For all $\lambda>0$ and $f\in L^q(J, v;L^p(\R^d_+,w;X))$ there exists a unique $u\in W^{1,q}(J,v;L^p(\R^d_+,w;X))\cap L^{q}(J,v;W^{2,p}_{\Dir}(\R^d_+,w;X))$ such that
$u' + (\lambda-\DD)u = f$, $u(0) = 0$ in case $J = \R_+$. Moreover, the following estimate holds
\begin{align*}
\|u'\|_{L^q(J,v;L^p(\R^d_+,w;X))} + \sum_{|\alpha|\leq 2}& \lambda^{1-\frac12|\alpha|} \|D^{\alpha} u\|_{L^{q}(J,v;L^p(\R^d_+,w;X))}
\\ & \lesssim_{p,q,v,w,d} \|f\|_{L^q(J, v;L^p(\R^d_+,w;X))}.
\end{align*}
\end{corollary}
\begin{proof}
Since $L^p(\R^d_+,w;X)$ is a UMD space, by Proposition \ref{propddtHinfty}, $d/dt$ has a bounded $H^\infty$-calculus on $L^q(J,v;L^p(\R^d_+,w;X))$.
Therefore, from Theorem \ref{thm:domainAp case}, Remark \ref{rem:Hinftytranslate} \eqref{it:firsttransH}, and Theorem \ref{thm:operatorsum}
\[\|u'\|_{L^q(J,v;L^p(\R^d_+,w;X))} + \|(\lambda-\Delta)u\|_{L^{q}(J,v;L^p(\R^d_+,w;X))} \lesssim_{p,q,v,w,d} \|f\|_{L^q(J, v;L^p(\R^d_+,w;X))}.\]
Now the result follows from Corollary \ref{cor:ellipticregAp} applied pointwise in $t$.
\end{proof}

\begin{remark} \
\begin{enumerate}[$(i)$]

\item The same result as in Corollary \ref{cor:weightedmaxAp} holds for $\Delta$ on the whole of $\R^d$. For results on elliptic and parabolic equations with $A_p$-weights in space we refer to \cite{HHH}.
\item Due to Calder\'on-Zygmund extrapolation theory one can add $A_q$-weights in time after considering the unweighted case (see \cite{ChFio}).

\item It would be interesting to extend Corollary \ref{cor:weightedmaxAp} to spaces of the form $L^{p}(\R\times \R^d_+,w;X)$ where $w$ depends on time and space. For some result in this direction concerning the maximal regularity estimate we refer to \cite{DKnew}.

\item The estimates in Corollaries \ref{cor:ellipticregAp} and \ref{cor:weightedmaxAp} also hold for $\lambda=0$. However, solvability does not hold for general $f$.
\end{enumerate}
\end{remark}

\section{$\DD$ on $\R^d_+$ in the non-$A_{p}$-setting\label{sec:heatnonAp}}

In this section we will extend the results of Section \ref{sec:heatApsetting} to weighted $L^p$-spaces with $w_{\gamma}(x) = |x_1|^{\gamma}$ where $\gamma\in (p-1,2p-1)$. This case is not included in the $A_p$-weight and is therefore not accessible through classical harmonic analysis. The reflection argument cannot be applied since the weight is not locally integrable in $\R^d$.

\subsection{The heat semigroup\label{subs:heatnonAp}}

Let $T(z):L^p(\R^d_+;X)\to L^p(\R^d_+;X)$ be defined by \eqref{eq:heatsemi}. We first show that $T(z)$ is also bounded on $L^p(\R^d_+,w_{\gamma})$ with $w_{\gamma}(x) = |x_1|^{\gamma}$ for $\gamma\in (-p-1, 2p-1)$ and that this range is optimal. Note that $w_{\gamma}\in A_p(\R^d)$ if and only if $\gamma\in (-1,p-1)$.

\begin{proposition}\label{prop:nonapbdd}
Let $p\in [1, \infty)$ and $\gamma\in (-1-p,2p-1)$. For every $|\phi|<\pi/2$, $(T(z))_{z\in \Sigma_{\phi}}$ defines a bounded analytic $C_0$-semigroup on $L^p(\R^d_+,w_{\gamma};X)$.
\end{proposition}
\begin{proof}
First we consider $p\in (1, \infty)$.
The result for $\gamma\in (-1,p-1)$ follows from Theorem \ref{thm:domainAp case}.
In the remaining cases by duality it suffices to consider $\gamma\in [p-1,2p-1)$.

Let $|\delta|<\phi$ and write $z = t e^{i \delta}$ for $t>0$.

\medskip
{\em Step 1: Reduction to an estimate in the case $X=\C$.}
In this step we show that it is enough to prove the estimate
\begin{align}\label{eq:uniformbddT}
\||H_z|*\|f\|\|_{L^p(\R_+^d,w_{\gamma})} \lesssim_{\phi,\gamma,p} \|f\|_{L^p(\R_+^d,w_{\gamma})}
\end{align}
for all $f \in C_{c}(\R^{d}_{+})$.
Having this estimate, we get
\begin{align*}
\|T(z)f\|_{L^p(\R_+^d,w_{\gamma};X)}\leq \||H_z|*\|f\|_X\|_{L^p(\R_+^d,w_{\gamma})} \leq C_{\phi,\gamma,p} \|f\|_{L^p(\R_+^d,w_{\gamma};X)}
\end{align*}
for all $f\in C_c(\R_+^d)\otimes X$, from which the analyticity and strong continuity follow.
Indeed, note that for $g\in C_c(\R_+)\otimes X$, $z\mapsto \lb T(z) f,g\rb$ is analytic on $\Sigma_{\phi}$ and continuous on $\overline{\Sigma_{\phi}}$  by Theorem \ref{thm:domainAp case}  with $w=1$. Therefore, in case $X=\C$, the weak continuity of $T$ on $\overline{\Sigma_{\phi}}$ follows by density in the case $p\in (1, \infty)$ and by weak$^*$-sequential density of $C_c$ in $L^\infty$ in case $p=1$ (see \cite[Corollary 2.24]{Rudreal}). This in turn implies strong continuity by \cite[Theorem I.5.8]{EN}.
For general $X$, the continuity of $T(z) f$ for $f\in C_c(\R^d_+)\otimes X$ is clear from the scalar case, yielding the case of general $f \in L^{p}(\R^{d}_{+};X)$ by density.
The analyticity of $T$ on $\Sigma_{\phi}$ follows from \cite[Theorem A.7]{ABHN}.

\medskip
{\em Step 2: Reduction to the case $d=1$.}
Writing $H^1_z$ for the kernel of $T(z)$ in case $d=1$ and $G_z^{d-1}$ for the standard heat kernel in dimension $d-1$, we have
\[|H_z|*|f|(x_1, \tilde{x}) = \int_{0}^{\infty} |H_z^{1}(x_1,y_1)| \int_{\R^{d-1}}  G_z^{d-1}(\tilde{x}-\tilde{y}) |f(y_1, \tilde{y})| d\tilde{y} dy_1.\]
Taking $L^p(\R^{d-1})$-norms for fixed $x_1\in \R_+$, and using Minkowski's inequality and Young's inequality, we obtain
\begin{align*}
\|&|H_z|*|f|(x_1, \cdot)\|_{L^p(\R^{d-1})} \\& \leq \int_{0}^{\infty} |H_z^{1}(x_1,y_1)| \Big\|\tilde{x}\mapsto \int_{\R^{d-1}} G_z^{d-1}(\tilde{x}-\tilde{y}) |f(y_1, \tilde{y})| d\tilde{y}\Big\|_{L^p(\R^{d-1})} dy_1
\\ & \leq C_{\phi}\int_{0}^{\infty} |H_z^{1}(x_1,y_1)| \|f(y_1,\cdot)\|_{L^p(\R^{d-1})} dy_1,
\end{align*}
where $C_{\phi} = \sup_{z\in \Sigma_{\phi}}\|G_z^{d-1}\|_{L^1(\R^d)}<\infty$.
Therefore, it remains to prove \eqref{eq:uniformbddT} in the case $d=1$.

\medskip
{\em Step 3: The case $d=1$.} Setting $g(x) := x^{\frac{\gamma+1}{p}} |f(x)|$, $k_z(s,y) := y (x/y)^{\frac{\gamma+1}{p}} |H_z(x,y)|$ and
\[
h_{z}(x) := \int_0^\infty k_z(x,y)  g(y) \frac{dy}{y},
\]
we see that \eqref{eq:uniformbddT} holds if and only if
$\|h_{z}\|_{L^p(\R_+,\frac{dx}{x})} \lesssim_{\phi} \|g\|_{L^p(\R_+,\frac{dx}{x})}$.
To prove this, by Schur's test (see \cite[Theorem 5.9.2]{Garling}) it is enough to show
\begin{align}\label{eq:Schur1}
\sup_{x>0}\int_0^\infty k_z(x,y) \frac{dy}{y}& \leq A,
\\ \label{eq:Schur2} \sup_{y>0}\int_0^\infty k_z(x,y) \frac{dx}{x} & \leq B.
\end{align}
In order to prove these estimates, observe that with $z = t e^{i\delta}$,
\begin{equation}\label{eq:Hzpoint}
\begin{aligned}
(4\pi t)^{1/2} |H_z(x,y)| & = \big|e^{\frac{-|x-y|^2e^{-i\delta}}{4t}} - e^{\frac{-|x+y|^2e^{-i\delta}}{4t}}\big|
\\ & = e^{\frac{-|x-y|^2\cos(\delta)}{4t}} \big| 1- e^{-\frac{xy e^{-i \delta }}{t}}\big|
\\ & \leq e^{\frac{-|x-y|^2\cos(\delta)}{4t}}  \int_0^{xy/t} e^{-s \cos(\delta)} \ud s
\\ & = (4\pi t)^{1/2} \cos(\delta)^{-1} H_{t\cos(\delta)}(x,y).
\end{aligned}
\end{equation}
Therefore, by replacing $x$ and $y$ by $(4t/\cos(\delta))^{1/2} x$ and $(4t/\cos(\delta))^{1/2} y$, respectively, in \eqref{eq:Schur1} and \eqref{eq:Schur2} it suffices to consider $t=1/4$ and $\delta=0$.

From now on we write
\[k(x,y) := y (x/y)^{\frac{\gamma+1}{p}}(e^{-|x-y|^2} - e^{-|x+y|^2}) = y (x/y)^{\frac{\gamma+1}{p}} e^{|x-y|^2} \big| 1- e^{-4xy}\big|.\]
One can check that $| 1- e^{-4xy}|\leq \min\{1, 4xy\}$.
Therefore,  $k$ satisfies
\[k(x,y) \leq y (x/y)^{\frac{\gamma+1}{p}} e^{-|x-y|^2}  \min\{1, 4xy\} \]
It follows that
\begin{align*}
\int_0^\infty k(x,y) \frac{dy}{y} & \leq \int_0^\infty (x/y)^{\frac{\gamma+1}{p}} e^{-|x-y|^2} \min\{1, 4xy\} dy
\\ & \leq \int_0^{x/2} (x/y)^{\frac{\gamma+1}{p}} e^{-|x-y|^2} 4xy dy  +  \int_{x/2}^\infty (x/y)^{\frac{\gamma+1}{p}} e^{-|x-y|^2}  dy
\\ & = T_1 + T_2.
\end{align*}
The first term satisfies
\begin{align*}
 T_1 & \leq 4 \int_0^{x/2} x^{\frac{\gamma+1}{p}+1} y^{1-\frac{\gamma+1}{p}} e^{-x^2/4}  dy = C_{1} x^{3} e^{-x^2/4} \leq A_1,
\end{align*}
where we used $1-\frac{\gamma+1}{p}>-1$. Since $\gamma>-1$, the second term satisfies
\begin{align*}
T_2 \leq C_{2} \int_{-\infty}^\infty  e^{-|x-y|^2}  dy = C_{2} \int_{-\infty}^\infty  e^{-|y|^2}  dy = A_2
\end{align*}

Next we estimate the integral over the $x$-variable.
For $y\in (0,1)$, we can write
\begin{align*}
  \int_0^\infty k(x,y) \frac{dx}{x} & \leq \int_0^\infty (x/y)^{\frac{\gamma+1}{p}-1} e^{-|x-y|^2} 4xy dx
\\ & \leq 4 \int_0^\infty x^{\frac{\gamma+1}{p}} y^{2-\frac{\gamma+1}{p}}   e^{-|x-y|^2} dx
\\ & = \int_{-y}^\infty (x+y)^{\frac{\gamma+1}{p}} y^{2-\frac{\gamma+1}{p}} e^{-|x|^2} dx
\\ & = \int_{-\infty}^\infty (x+1)^{\frac{\gamma+1}{p}}  e^{-|x|^2} dx\leq B_1,
\end{align*}
where we used $2-\frac{\gamma+1}{p}\geq 0$ and $\gamma+1\geq 0$. For $y\geq 1$, since $\frac{\gamma+1}{p}\geq 1$ we have
\begin{align*}
  \int_0^\infty k(x,y) \frac{dx}{x} & \leq \int_0^\infty (x/y)^{\frac{\gamma+1}{p}-1} e^{-|x-y|^2} dx
\\ & = \int_{-y}^\infty \big(\tfrac{x}{y} + 1\big)^{\frac{\gamma+1}{p}-1} e^{-x^2} dx
\\ & \leq \int_{-\infty}^\infty \big(|x| + 1\big)^{\frac{\gamma+1}{p}-1} e^{-x^2} dx\leq B_2.
\end{align*}

{\em Step 4: The case $p=1$:} One can still reduce to the case $d=1$ by Fubini's theorem. Moreover, instead of using Schur's lemma, by Fubini's theorem it suffices to show that
\[\sup_{y>0}\int_0^\infty k(x,y) \frac{dx}{x}<\infty.\]
The case $\gamma\in [0,1)$ can be treated in the same way as in the above proof. In case $\gamma\in (-2,0)$  we argue as follows:
\begin{align*}
\int_{y/2}^\infty k(x,y) \frac{dx}{x} &\leq \int_{y/2}^\infty (x/y)^{\gamma} e^{-|x-y|^2}  dx \leq 2^\gamma \int_{-\infty}^\infty e^{-|x-y|^2}  dx =  C.
\end{align*}
On the other hand, since $\gamma+2>0$, we have
\begin{align*}
\int_0^{y/2}  k(x,y) \frac{dx}{x} & \leq \int_0^{y/2}   (x/y)^{\gamma} e^{-|x-y|^2}  4xy
\\ & \leq 4e^{-y^2/4} y^{-\gamma+1} \int_0^{y/2} x^{\gamma+1} dx = 4 ye^{-y^2/4}. \qedhere
\end{align*}
\end{proof}

In the next example we show that the range for $\gamma$ in Proposition \ref{prop:nonapbdd} is optimal.
\begin{example}
Let $p\in (1, \infty)$ and $\gamma\notin (-p-1, 2p-1)$. We give an example of a function $f\in L^p(\R^d_+,w_{\gamma})$ such that for all $t>0$, $T(t)f\notin  L^p(\R^d_+,w_{\gamma})$. Here $T(t) f$ is defined by \eqref{eq:heatsemi}. By duality we only need to consider $\gamma\geq 2p-1$. Let $\beta\in (1/p,1)$ and set $f(x) = x_1^{-2} |\log(x_1)|^{-\beta} \one_Q(x)$, where $Q = [0,1]^d$. Then, on the one hand, $f\in L^p(\R^d_+;w_{\gamma})$. On the other hand, for $x\in Q^d$,
\begin{align*}
T(t) f(x) &=c_{t,d} \int_{Q} e^{-\frac{|x-y|^2}{4t}} [1 - e^{-\frac{-x_1 y_1}{t}}] y_1^{-2} |\log(y_1)|^{-\beta} \ud y
\\ & = \tilde{c}_{t,d} \int_{0}^1 y_1^{-1} |\log(y_1)|^{-\beta} \ud y_1 = \infty;
\end{align*}
in particular, $T(t)f\notin  L^p(\R^d_+,w_{\gamma})$.
\end{example}

Let $-A$ denote the generator of the semigroup $(T(z))_{z}$ of Proposition \ref{prop:nonapbdd}.
Then by standard results of analytic semigroups we see that $A$ is sectorial with $\omega(A) = 0$.

In the case of a $X$ is a UMD space, $-A$ even has a bounded $H^{\infty}$-calculus:

\begin{proposition}\label{prop:bddHinfty}
Let $X$ be a UMD space. Let $-A$ be the generator of the heat semigroup on $L^p(\R^d_+,w_{\gamma};X)$ given in Proposition \ref{prop:nonapbdd} with $p\in (1, \infty)$ and $\gamma\in (-1-p,2p-1)$. Then $A$ has a bounded $H^\infty$-calculus with $\omega_{H^\infty}(A) = 0$.
\end{proposition}
\begin{proof}
The case $\gamma\in (-1,p-1)$ follows from Theorem \ref{thm:domainAp case}. For the other values of $\gamma$ we use a classical perturbation argument (see \cite{Kree}).

\bigskip

{\em  Step 1:}
Let $0<\sigma<\omega<\pi/2$
Let $\phi\in H^\infty_0(\Sigma_{\omega})$ with $\omega\in (0,\pi/2)$ satisfy $\|\phi\|_{\infty}\leq 1$ and let $\Gamma = \partial \Sigma_{\sigma}$. By definition we have
\begin{align}\label{eq:defHinfty}
\phi(A) & = \frac{1}{2\pi i}\int_{\Gamma} \phi(\lambda) R(\lambda, A) d\lambda = \frac{1}{2\pi i}\int_{\Gamma_+\cup \Gamma_-} \phi(\lambda) R(\lambda, A) d\lambda,
\end{align}
where $\Gamma_{\pm} = \{t e^{\pm \sigma i}: t\in (0,\infty)\}$.

Fix $f\in C_c(\R^d_+;X))$ and let $g = w_{\gamma}^{\frac1p} f$ and $\psi(x,y) = \Big(\frac{w^{\frac1p}_{\gamma}(x_1)}{w^{\frac1p}_{\gamma}(y_1)} - 1\Big)$.
Then for $x\in \R^d_+$
\begin{align*}
\phi(A) f(x) & = w^{-\frac1p}_{\gamma}(x_1) \phi(A) (w^{\frac1p}_{\gamma} f)(x) + w^{-\frac1p}_{\gamma}(x_1) \phi(A) \big(w^{\frac1p}_{\gamma}(x_1) - w^{\frac1p}_{\gamma})f\big)(x)
\\ & = w^{-\frac1p}_{\gamma}(x_1) \phi(A)(g)(x) + w^{-\frac1p}_{\gamma}(x_1) \phi(A)(\psi(x,\cdot) g)(x).
\end{align*}
Therefore,
\[\|\phi(A) f\|_{L^p(\R_+,w_{\gamma};X)} \leq \|\phi(A) g\|_{L^p(\R_+;X)} + \|x\mapsto \phi(A) (\psi(x,\cdot) g)(x)\|_{L^p(\R_+;X)}.\]
The first term on the right-hand side can be estimated by the boundedness of the $H^\infty$-calculus in the unweighted case (see Theorem \ref{thm:domainAp case}):
\[\|\phi(A) g\|_{L^p(\R_+;X)}\leq C \|g\|_{L^p(\R_+;X)} = C\|f\|_{L^p(\R_+,w_{\gamma};X)}.\]
Therefore, it remains to show
\begin{equation}\label{eq:phiApsig}
 \Big\|x\mapsto \phi(A) (\psi(x,\cdot) g)(x)\Big\|_{L^p(\R_+;X)}\leq C\|g\|_{L^p(\R_+;X)} = C\|f\|_{L^p(\R_+,w_{\gamma};X)}.
\end{equation}

\bigskip

{\em  Step 2:}  To prove \eqref{eq:phiApsig} we estimate the integrals over $\Gamma_{\pm}$ in \eqref{eq:defHinfty} separately. By symmetry it suffices to consider $\Gamma_{+}$. Let $\delta = (\pi-\sigma)/2$.  For $\lambda = re^{i\sigma}$ with $r>0$ and $h \in L^{p}(\R_{+},w_{\gamma};X)$, we have the following Laplace transform representation for the resolvent (see \cite{EN}):
\begin{align*}
R(\lambda,A)h & = (\lambda-A)^{-1}h =-(re^{i(\sigma-\pi)} + A)^{-1}h
\\ & = -e^{i\delta} (re^{-i\delta} + e^{i\delta}A)^{-1}h
= e^{i\delta} \int_0^\infty e^{-t re^{-i\delta}} e^{-te^{i\delta} A}h \,dt.
\end{align*}
Observe that by \eqref{eq:Hzpoint} we can write
\begin{align*}
\Big\|&\int_{\Gamma_+} \phi(\lambda) R(\lambda, A) (\psi(x,\cdot) g)(x) d\lambda\Big\|_X \\ &  \leq \int_{0}^\infty \|R(r e^{i \sigma}, A) (\psi(x,\cdot) g)(x)\|  dr
\\ & \leq \int_{0}^\infty \int_0^\infty  \|e^{-te^{i\delta} A}  e^{- tr e^{-i \delta}} (\psi(x,\cdot) g)(x)\|_X  dt dr
\\ & \leq \frac{1}{\cos(\delta)} \int_{0}^\infty \int_0^\infty  e^{-t\cos(\delta) A} \|\psi(x,\cdot) g\|(x) e^{- t r \cos(\delta)}   dt dr
\\ & = \frac{1}{\cos^2(\delta)}\int_0^\infty  e^{-t\cos(\delta) A} \|\psi(x,\cdot) g\|(x)    \frac{dt}{t}.
\end{align*}
Below we will write $x=(x_{1},\tilde{x})$ and $y=(y_{1},\tilde{y})$.
Using the kernel representation of the semigroup we can write
\begin{align*}
\int_0^\infty  & e^{-t\cos(\delta) A} \|\psi(x,\cdot) g\|(x)    \frac{dt}{t}
\\ &= \int_0^\infty   e^{-t A} \|\psi(x,\cdot) g\|(x)    \frac{dt}{t}
\\   & = \int_{0}^\infty \int_{\R^{d}_{+}} (G_{t}(x_{1}-y_{1},\tilde{x}-\tilde{y}) - G_{t}(x_{1}+y_{1},\tilde{x}-\tilde{y}))|\psi(x,y)| \, \|g(y)\|  dy \frac{dt}{t}
\\   & = \int_{\R^{d}_{+}}\int_{0}^\infty(G_{t}(x_{1}-y_{1},\tilde{y}) - G_{t}(x_{1}+y_{1},\tilde{y})) \frac{dt}{t} |\psi(x_{1},y_{1})| \, \|g(y_{1},\tilde{x}-\tilde{y})\|  dy
\\ & = C_{1} \int_{\R^{d}_{+}} \Big(\frac{1}{|(x_{1}/y_{1}-1,\tilde{y}/y_{1})|^{d}}-\frac{1}{|(x_{1}/y_{1}+1,\tilde{y}/y_{1})|^{d}} \Big)  |(x_{1}/y_{1})^{\frac{\gamma}p} - 1\Big| \|g(y)\| \, \frac{dy}{y_{1}^{d}}
\\ & =: C_{1}\int_{\R^{d}_{+}} \ell(x_{1}/y_{1},\tilde{y}/y_{1}) \, \|g(y_{1},\tilde{x}-\tilde{y})\|  \frac{dy}{y_{1}^{d}}
\\ &= C_{1}\int_{\R^{d}_{+}} \ell(x_{1}/y_{1},\tilde{y}) \, \|g(y_{1},\tilde{x}-y_{1}\tilde{y})\|  \frac{dy}{y_{1}},
\end{align*}
where we used
\begin{align*}
\int_0^\infty  G_t(x) \, \frac{dt}{t} &= \int_0^\infty (4\pi t)^{-d/2} e^{-\frac{|x|^2}{4t}} \, \frac{dt}{t}
= \int_0^\infty (4\pi)^{-d/2} s^{d/2} e^{-\frac{s}{4}} \, \frac{ds}{s} |x|^{-d}=C_{1}|x|^{-d} .
\end{align*}
Now,
\begin{align*}
\Big\|& x\mapsto \int_{\R^{d}_{+}} \ell(x_{1}/y_{1},\tilde{y}) \, \|g(y_{1},\tilde{x}-y_{1}\tilde{y})\|  \frac{dy}{y_{1}}  \Big\|_{L^p(\R^d_+)}
\\ & \leq \Big\| x_{1}\mapsto \int_{\R^{d}_{+}} \ell(x_{1}/y_{1},\tilde{y}) \, \| g(y_{1},\,\cdot\,-y_{1}\tilde{y})\|_{L^{p}(\R^{d-1};X)}  \frac{dy}{y_{1}}  \Big\|_{L^p(\R_+)}
\\ & = \Big\| x_{1}\mapsto \int_{\R^{d-1}}\int_{0}^{\infty} \ell(x_{1}/y_{1},\tilde{y}) \, \| g(y_{1},\,\cdot\,)\|_{L^{p}(\R^{d-1};X)}  \frac{dy_{1}}{y_{1}}d\tilde{y}  \Big\|_{L^p(\R_+)}
\\ & = \Big\| x_{1}\mapsto \int_{\R^{d-1}}\int_{0}^{\infty} \ell(y_{1},\tilde{y}) \, \| g(x_{1}/y_{1},\,\cdot\,)\|_{L^{p}(\R^{d-1};X)}  \frac{dy_{1}}{y_{1}}d\tilde{y}  \Big\|_{L^p(\R_+)}
\\ & = \Big\| x_{1}\mapsto \int_{\R^{d-1}}\int_{0}^{\infty} \ell(y_{1},\tilde{y})y_{1}^{p}(x_{1}/y_{1})^{p} \, \| g(x_{1}/y_{1},\,\cdot\,)\|_{L^{p}(\R^{d-1};X)}  \frac{dy_{1}}{y_{1}}d\tilde{y}  \Big\|_{L^p(\R_+,\frac{dx_{1}}{x_{1}})}
\\ & \leq \int_{\R^{d-1}}\int_{0}^{\infty} \ell(y_{1},\tilde{y})y_{1}^{p} \, \Big\| x_{1}\mapsto (x_{1}/y_{1})^{p} \| g(x_{1}/y_{1},\,\cdot\,)\|_{L^{p}(\R^{d-1};X)}  \Big\|_{L^p(\R_+,\frac{dx_{1}}{x_{1}})} \frac{dy_{1}}{y_{1}}d\tilde{y}
\\ & = \int_{\R^{d-1}}\int_{0}^{\infty} \ell(y_{1},\tilde{y})y_{1}^{p} \, \Big\| x_{1}\mapsto x_{1}^{p} \| g(x_{1},\,\cdot\,)\|_{L^{p}(\R^{d-1};X)}  \Big\|_{L^p(\R_+,\frac{dx_{1}}{x_{1}})} \frac{dy_{1}}{y_{1}}d\tilde{y}
\\ & = C_{2}\|g\|_{L^{p}(\R^{d}_{+};X)}.
\end{align*}
Here we use $-1-p<\gamma<2p-1$ to obtain
\begin{align*}
C_{2}
&:= \int_{\R^{d-1}}\int_{0}^{\infty} \ell(y_{1},\tilde{y})y_{1}^{p} \, \frac{dy_{1}}{y_{1}}d\tilde{y} \\
&= \int_{0}^{\infty}\int_{\R^{d-1}}\Big(\frac{1}{|(y_{1}-1,\tilde{y})|^{d}}-\frac{1}{|(y_{1}+1,\tilde{y})|^{d}} \Big) \, d\tilde{y} |y_{1}^{\frac{\gamma}p} - 1|y_{1}^{p} \frac{dy_{1}}{y_{1}} \\
&= C_{3}\int_{0}^{\infty}\Big(\frac{1}{|y_{1}-1|}-\frac{1}{|y_{1}+1|} \Big)  |y_{1}^{\frac{\gamma}p} - 1|y_{1}^{p} \, \frac{dy_{1}}{y_{1}} < \infty,
\end{align*}
where $C_3 = \int_{\R^{d-1}} (1+|\tilde{y}|)^{-d} \ud \tilde{y}$ if $d\geq 2$ and $C_3 = 1$ otherwise.
Combining the above estimates we obtain the required estimate
\begin{align*}
\Big\|x\mapsto \int_{\Gamma_+} \phi(\lambda) R(\lambda, A) (\psi(x,\cdot) g)(x) d\lambda\Big\|_{L^p(\R_+;X)}
\leq \frac{C \|g\|_{L^p(\R_+;X)}}{\cos^2(\delta)}.
\end{align*}
\end{proof}

\subsection{The Dirichlet Laplacian on $\R_+$\label{subs:funcnonAp}}

\begin{proposition}\label{prop:wellposedd=1}
Let $p\in (1,\infty)$ and $\gamma\in (p-1,2p-1)$. Then $\DD$, defined as
\[
D(\Delta_{\Dir}) := W^{2,p}_{\Dir}(\R_{+},w_{\gamma};X),\quad \Delta_{\Dir}u := u'',
\]
is the generator of the heat semigroup on $L^p(\R_+,w_{\gamma};X)$ given in Proposition \ref{prop:nonapbdd}.
\end{proposition}

For the case $\gamma \in (-p-1,-1)$ we refer the reader to Section~\ref{subsec:Dir_Laplace_(-p-1,-1)}.

\begin{proof}
Let $-A$ denote the generator of the heat semigroup $T$ of Proposition \ref{prop:nonapbdd}.
We first show that $\DD \subseteq -A$, that is, $W^{2,p}_{\Dir}(\R_+,w_{\gamma};X) \subseteq D(A)$ and for $u\in W^{2,p}_{\Dir}(\R_+,w_{\gamma};X)$ one has $-A u = \DD u$.
From Theorem \ref{thm:domainAp case} we see that for $u\in C^\infty_c(\R_+;X)$,
\[T(t) u - u = \int_0^t T(s) \DD u ds.\]
Therefore, $\frac1t(T(t) u - u) \to \DD u$ in $L^p(\R_+,w_{\gamma};X)$ by strong continuity of $(T(s))_{s\geq 0}$. Therefore, $u \in D(A)$ with $-A u = \DD u$. Now for $u\in W^{2,p}_{\Dir}(\R_+,w_{\gamma};X)$, using Proposition \ref{prop:densenonAp}, we can find a sequence $(u_n)_{n\geq 1}$ in $C^\infty_c(\R_+;X)$ such that $u_n\to u$ in $W^{2,p}_{\Dir}(\R_+,w_{\gamma};X)$. Then $-A u_n = \DD u_n \to \DD u$ in $L^p(\R_+,w;X)$. Therefore, the closedness of $A$ yields that $u\in D(A)$ and $-A u = \DD u$.

Next we show $-A \subseteq \DD$. Using $\DD \subseteq -A$, for this it is enough that $1+A$ is injective  and $1-\DD$ is surjective. Being the generator of a bounded analytic semigroup (see Proposition~\ref{prop:nonapbdd}), $A$ is sectorial, implying that $1+A$ is injective.
For the surjectivity of $1-\DD$ we consider the equation $u-\DD u=f$, for $f\in L^p(\R_+^d,w_{\gamma};X)$.

Let us first consider $f\in C_c^\infty(\R_+;X)$.
Let $\overline{f}$ denote the odd extension of $f$.
Clearly, $\overline{f}\in C_c^{\infty}(\R;X) \subset \mathcal{S}(\R;X)$.
So we can define $\overline{u} \in \mathcal{S}(\R;X)$ by $\overline{u} := \F^{-1}[\xi \mapsto \frac{\F\overline{f}(\xi)}{1+\xi^2}]$, yielding a solution of the equation $\overline{u}-\overline{u}''=\overline{f}$. Since $\overline{u}$ is odd, it also satisfies the Dirichlet condition $u(0)=0$.
By restriction to $\R_{+}$ we obtain a solution $u:=\overline{u}_{|\R_{+}} \in W^{2,p}_{\Dir}(\R_+,w_{\gamma};X)$ of the equation $(1-\DD)u=f$.
As $W^{2,p}_{\Dir}(\R^d,w_{\gamma};X)$ is complete and $C_c^\infty(\R_+;X)$ is dense in $L^p(\R_+^d,w_{\gamma};X)$ (see Proposition \ref{prop:densenonAp}), it suffices to prove the estimate $\|u\|_{W^{2,p}(\R^d,w_{\gamma};X)} \lesssim \|f\|_{L^p(\R_+^d,w_{\gamma};X)}$.

To finish, we prove this estimate.
As $\DD \subset A$, we have  $u \in D(A)$ with $(1-A)u=f$, so $u=R(1,A)f$. It follows that $\|u\|_{L^p(\R_+,w_{\gamma};X)} \lesssim \|f\|_{L^p(\R_+,w_{\gamma};X)}$. Since $u'' = u-f$ we find that $\|u''\|_{L^p(\R_+,w_{\gamma};X)}\lesssim \|f\|_{L^p(\R_+,w_{\gamma};X)}$.
By interpolation the same estimate holds for $u'$ (see Lemma \ref{lem:interp2}).
\end{proof}

\begin{corollary}\label{cor:prop:wellposedd=1}
Let $p\in (1,\infty)$ and $\gamma\in (p-1,2p-1)$.
For all $\lambda>0$ and $f\in L^p(\R_+;w_{\gamma};X)$ there exists a unique $u\in W^{2,p}_{\Dir}(\R_+,w_{\gamma};X)$ such that $\lambda u-u'' = f$ and
\begin{equation}\label{eq:apriorid=1}
\sum_{j=0}^2 |\lambda|^{1-\frac{j}{2}} \|D^{j} u\|_{L^p(\R_+,w_{\gamma};X)}  \lesssim_{p,\gamma} \|f\|_{L^p(\R_+;w_{\gamma};X)}.
\end{equation}
\end{corollary}
\begin{proof}
This can be proved in the same way as the second statement in Corollary~\ref{cor:ellipticregAp}.
\end{proof}

Combining Propositions and \ref{prop:wellposedd=1} and \ref{prop:bddHinfty}, we find the following result in the one-dimensional case:
\begin{corollary}\label{cor:prop:wellposedd=1;Hinfty}
Let $X$ be a UMD space, $p\in (1,\infty)$ and $\gamma\in (p-1,2p-1)$.
Then $-\DD$ has a bounded $H^{\infty}$-calculus on $L^p(\R_+,w_{\gamma};X)$ with $\omega_{H^{\infty}}(-\Delta_{\Dir})=0$.
\end{corollary}

\subsection{The Dirichlet Laplacian on $\R^d_+$\label{subsec:dgroot}}

The main result of this section is the following theorem. Note that the case $\gamma\in (-1,p-1)$ was already considered in Theorem \ref{thm:domainAp case}. See Section~\ref{subsec:Dir_Laplace_(-p-1,-1)} for the case $\gamma \in (-p-1,-1)$.

Before we state the theorem, let us first define the Dirichlet Laplacian $\DD$ on  $L^{p}(\R^{d}_{+},w_{\gamma};X)$ with $p \in (1,\infty)$ and $\gamma \in (p-1,2p-1)$:
\[
D(\DD) := W^{2,p}_{\Dir}(\R^{d}_{+},w_{\gamma};X),\qquad \DD u:= \Delta u.
\]

\begin{theorem}\label{thm:domain}
Let $X$ be a UMD space, $p\in (1, \infty)$ and $\gamma\in (p-1,2p-1)$. Then the following assertions hold:
\begin{enumerate}
\item\label{it:generator} $\DD$
    is the generator of the heat semigroup from Proposition \ref{prop:nonapbdd}.
\item\label{it:domain} $\DD$ is a closed and densely defined linear operator on $L^p(\R^d_+,w_{\gamma};X)$ with
    \begin{align*}
    D(\DD)
    &= W^{2,p}_{\Dir}(\R^{d}_{+},w_{\gamma};X) \\
    &= W^{2,p}_{\Dir}(\R_{+},w_{\gamma};L^{p}(\R^{d-1};X)) \cap L^{p}(\R_{+},w_{\gamma};W^{2,p}(\R^{d-1};X))
    \end{align*}
    with an equivalence of norms only depending on $X,p,d,\gamma$.
\item\label{it:Hinfty} For all $\lambda\geq 0$, $\lambda-\DD$ has a bounded $H^\infty$-calculus with $\omega_{H^\infty}(-\DD) = 0$.
\end{enumerate}
\end{theorem}
\begin{proof}
Note that \eqref{it:Hinfty} follows from \eqref{it:generator} by Proposition~\ref{prop:bddHinfty} and Remark \ref{rem:Hinftytranslate}.
So we only need to prove \eqref{it:generator} and \eqref{it:domain}.

Below we will frequently use  and Fubini's theorem in the form of the identification
\[
L^{p}(\R^{d}_{+},w_{\gamma};X) = L^{p}(\R_{+},w_{\gamma};L^{p}(\R^{d-1};X)) = L^{p}(\R^{d-1};L^{p}(\R_{+},w_{\gamma};X)),
\]
and that UMD-valued $L^p$-spaces have UMD again.
By Corollary~\ref{cor:prop:wellposedd=1;Hinfty}, for the operator $\Delta_{1,\Dir}$ on $L^p(\R_+^d,w_{\gamma};X)$, defined by
\[
D(\Delta_{1,\Dir}) := W^{2,p}_{\Dir}(\R_+,w_{\gamma};L^{p}(\R^{d-1};X)), \quad \Delta_{1,\Dir}u :=\partial_1^2u,
\]
it holds that $-\Delta_{1,\Dir}$ a bounded $H^\infty$-calculus with $\omega_{H^\infty}(-\Delta_{1,\Dir}) = 0$.
By \cite[Theorem 10.2.25]{HNVW2}, for the operator $\Delta_{d-1}$ on $L^p(\R_+^d,w_{\gamma};X)$, defined by
\[
D(\Delta_{d-1}) := L^p(\R_+,w_{\gamma},W^{2,p}(\R^{d-1};X)),
\quad \Delta_{d-1}u := \sum_{k=2}^d \partial_k^2u,
\]
it holds that $-\Delta_{d-1}$ a bounded $H^\infty$-calculus  with $\omega_{H^\infty}(-\Delta_{d-1}) = 0$.
The operators $\Delta_{1,\Dir}$ and $D(\Delta_{d-1})$ are clearly resolvent commuting.
Therefore, by Theorem \ref{thm:operatorsum} for the operator sum
$\DD^{\Sigma} := \Delta_{1,\Dir} +\Delta_{d-1}$ with $D(\DD^{\Sigma}) = D(\Delta_{1,\Dir}) \cap D(\Delta_{d-1})$ it holds that $-\DD^{\Sigma}$ is a sectorial operator with angle $\omega(-\DD^{\Sigma}) = 0$.
Moreover,
\begin{equation}\label{eq:thm:domain;semigroup}
e^{t\DD^{\Sigma}} = e^{t\Delta_{1,\Dir}}e^{t\Delta_{d-1}}, \quad t \geq 0.
\end{equation}

Writing $H^{1}_{t}$ for the kernel in \eqref{eq:heatsemikernel} in dimension $1$ and $G^{d-1}_{t}$ for the standard heat kernel in dimension $d-1$, \eqref{eq:thm:domain;semigroup} and Proposition~\ref{prop:wellposedd=1} give
\begin{align*}
[e^{t\DD^{\Sigma}}f](x)
&= \int_{0}^{\infty}H^{1}_{t}(x_{1},y_{1})\int_{\R^{d-1}}
G^{d-1}_{t}(\tilde{x}-\tilde{y})f(y_{1},\tilde{y})d\tilde{y} dy_{1} \\
& = \int_{\R^{d}_{+}}H_{t}(x,y)f(y)dy
\end{align*}
for all $f \in L^p(\R_+^d,w_{\gamma};X)$.
Therefore, $\DD^{\Sigma}$ is the generator of the heat semigroup from Proposition \ref{prop:nonapbdd}.

We now show that
\begin{align*}
D(\DD^{\Sigma}) &= W^{2,p}_{\Dir}(\R_{+},w_{\gamma};L^{p}(\R^{d-1};X)) \cap L^{p}(\R_{+},w_{\gamma};W^{2,p}(\R^{d-1};X)) \\
&= W^{2,p}_{\Dir}(\R^{d}_{+},w_{\gamma};X)
\end{align*}
with an equivalence of norms.
Note that then $\DD = \DD^{\Sigma}$ and the assertions
\eqref{it:generator}, \eqref{it:domain}
follow.
Since $\DD^{\Sigma} = \Delta_{1,\Dir} +\Delta_{d-1}$ with $D(\DD^{\Sigma}) = D(\Delta_{1,\Dir}) \cap D(\Delta_{d-1})$, the first identity follows from the domain descriptions of $\Delta_{1,\Dir}$ and $\Delta_{d-1}$.
The second identity follows from Corollary~\ref{cor:mixed-derivative}.
\end{proof}

\begin{corollary}\label{cor:ellipticregnonAp}
Let $X$ be a UMD space, $p\in (1, \infty)$ and $\gamma\in (p-1,2p-1)$.
For all $u \in W^{2,p}_{\Dir}(\R^d_+,w_{\gamma};X)$ there hold the estimates
\begin{align}\label{eq:cor:ellipticregnonAp:eq_hom_part_Sobolev}
[u]_{W^{2,p}(\R^{d}_{+},w_{\gamma};X)} \eqsim_{X,p,d,\gamma} \| \Delta u \|_{L^p(\R^d_+,w_{\gamma};X)}.
\end{align}
Furthermore, for every $f\in L^p(\R^d_+,w_{\gamma};X)$ and $\lambda>0$ there exists a unique $u\in W^{2,p}_{\Dir}(\R^d_+,w_{\gamma};X)$ such that $\lambda u - \DD u= f$ and
\begin{align}\label{eq:aprioridnoAp}
\sum_{|\alpha|\leq 2} |\lambda|^{1-\frac12|\alpha|} \|D^{\alpha} u\|_{L^p(\R^d_+,w_{\gamma};X)} \lesssim_{X,p,d,\gamma} \|f\|_{L^p(\R^d_+,w_{\gamma};X)}.
\end{align}
\end{corollary}
\begin{proof}
This can be done in the same way as Corollary~\ref{cor:ellipticregAp}, now using the explicit formula $w_{\gamma}(r\,\cdot\,)= r^{\gamma}w_{\gamma}$ in the scaling argument.
\end{proof}

\begin{remark}
The second statement in Corollary~\ref{cor:ellipticregnonAp} also follows from \cite[Theorem 4.1 $\&$ Remark 4.2]{Krylovheat}. In our setting it follows from operator sum methods involving bounded imaginary powers (obtained through the $H^\infty$-calculus).
\end{remark}

Now using Theorem \ref{thm:domain}, as in Corollary \ref{cor:weightedmaxAp} we obtain the following maximal regularity result for the weights $w_{\gamma}$ with $\gamma\in (p-1, 2p-1)$. The case $\gamma\in (-1, p-1)$ was already considered in Corollary \ref{cor:weightedmaxAp}.
\begin{corollary}[Heat equation]\label{cor:weightedmaxnonAp}
Let $X$ be a UMD space. Let $p,q\in (1, \infty)$, $v\in A_q(\R)$, $\gamma\in (p-1, 2p-1)$. Let $J\in \{\R_+, \R\}$. Then the following assertions hold:
\begin{enumerate}[$(1)$]
\item $\frac{d}{dt}-\DD$ is a closed sectorial operator on $L^q(J,v;L^p(\R^d_+,w_{\gamma};X))$ which has a bounded $H^\infty$-calculus with $\omega_{H^\infty}(\frac{d}{dt}-\DD) \leq  \frac{\pi}{2}$.
\item\label{it:weightedmaxnonAp2}
For all $\lambda>0$ and $f\in L^q(J, v;L^p(\R^d_+,w_{\gamma};X))$ there exists a unique $u\in W^{1,q}(J,v;L^p(\R^d_+,w_{\gamma};X))\cap L^{q}(J,v;W^{2,p}_{\Dir}(\R^d_+,w_{\gamma};X))$ such that
$u' + (\lambda-\DD)u = f$, $u(0) = 0$ in case $J=\R_+$. Moreover, the following estimate holds
\begin{align*}
\|u'\|_{L^q(J,v;L^p(\R^d_+,w_{\gamma};X))} + \sum_{|\alpha|\leq 2}\lambda^{1-\frac12|\alpha|} &\|D^{\alpha} u\|_{L^{q}(J,v;L^p(\R^d_+,w_{\gamma};X))} \\ & \lesssim_{p,q,v,\gamma,d} \|f\|_{L^q(J, v;L^p(\R^d_+,w_{\gamma};X))}.
\end{align*}
\end{enumerate}
\end{corollary}

\begin{remark}\label{rem:Krylovpq}
In the case $v=1$, Corollary \ref{cor:weightedmaxnonAp} \eqref{it:weightedmaxnonAp2} reduces to \cite[Theorem 0.1]{Krylovheatpq}, where it was deduced using completely different methods. Let us mention here that in \cite[Theorem 0.1]{Krylovheatpq}, \cite[Theorem 4.4]{KoNa} \cite[Theorem 2.1]{DongKimweighted15} more general elliptic operators with time and space-dependent coefficients have been considered.
\end{remark}

\begin{problem} Let $p\in (1, \infty)$.
\begin{enumerate}
\item Characterize those weights $w$ for which $e^{t\DD}$ extends to a bounded analytic semigroup on $L^p(\R^d_+,w)$.
\item Characterize those weights $w$ for which $\DD$ has a bounded $H^\infty$-calculus on $L^p(\R^d_+,w)$.
\item Characterize those weights $w$ for which $\DD$ on $L^p(\R^d_+,w)$ is a closed operator with $D(\DD) = W^{2,p}_{\Dir}(\R^d_+,w)$.
\end{enumerate}
\end{problem}
Given the results of Sections \ref{sec:heatApsetting} and \ref{sec:heatnonAp} it would be natural to conjecture that all weights of the form $w(x) = v_0(x) + x_1 v_1(x)$ with $v_0,v_1\in A_p$ are included.

\subsection{Extrapolation of functional calculus\label{subs:extra}}
As soon as one knows the boundedness of the functional calculus of a generator on a space $L^2(\R^d_+,d\mu)$ for some doubling measure $\mu$, then, if the heat kernel satisfies Gaussian estimates with respect to $\mu$, one can extrapolate the boundedness of the functional calculus to $L^p(\R^d_+,wd\mu)$ for $p\in (1, \infty)$ and $w\in A_p(\mu)$. Here $A_p(\mu)$ is the weight class associated to the measure $\mu$ on $\R^d_+$. The above is presented in the setting of homogeneous spaces in \cite{DuMc99} in the unweighted setting and in \cite[Theorem 7.3]{Martell04} in the weighted setting. Extension to the setting without kernel bounds can be found in \cite{AuMarIII, BluKun}.

In order to apply \cite[Theorem 7.3]{Martell04} to our setting, we set $d\mu(x) = x_1 \ud x$. The reason to take this measure is that the kernel $H_z(x,y)$ as defined in \eqref{eq:heatsemi} has a zero of order one at $x_1=0$. Then $\mu$ is doubling and one can check that $w_{\alpha}(x):=x_1^\alpha$ is in $A_p(\mu)$ if and only if $\alpha\in (-2, 2p-2)$. From Theorem \ref{thm:domain} we know that on $L^2(\R^d_+,\mu)$, $-\Delta_{\Dir}$ has a bounded $H^\infty$-calculus with $\omega_{H^\infty}(-\Delta_{\Dir}) = 0$. So in order to extrapolate the latter to $L^p(\R^d_+,wd\mu)$ for $p\in (1, \infty)$ and $w\in A_p(\mu)$ it suffices to check the kernel condition of \cite[Theorem 7.3]{Martell04}. For this (due to \eqref{eq:Hzpoint}) it suffices to show that there exist constant $C,c>0$ such that
\begin{equation}\label{eq:conditionkernel}
\frac{H_t(x,y)}{y_1} \leq \frac{C e^{-c|x-y|^2/t}}{\mu(B(x,t^{1/2}))}, \ \ \ x,y\in \R^d_+, t>0.
\end{equation}
Here the nominator $y_1$ is to correct for the choice of the measure $\mu$. After renormalization for the condition \eqref{eq:conditionkernel} it suffices to consider $t=1/4$.
In the case that $x_1<1$, \eqref{eq:conditionkernel} is equivalent to
\[\frac{e^{-|x_1-y_1|^2} - e^{-|x_1+y_1|^2}}{y_1} e^{-|\tx - \tilde{y}|^2} \leq \frac{C' e^{-c'|x-y|^2}}{\max\{x_1, 1\}}.\]
Since $\frac{e^{-|x_1-y_1|^2} - e^{-|x_1+y_1|^2}}{y_1} e^{-|\tx - \tilde{y}|^2} = e^{-|x-y|^2} \frac{1-e^{-4x_1y_1}}{y_1}$, it suffices to check that for some $\theta\in [0,1)$ and $C_{\theta}\geq 0$,
\[\frac{1-e^{-4x_1y_1}}{y_1} \leq \frac{C_{\theta}e^{\theta|x_1-y_1|^2}}{\max\{x_1, 1\}}.\]
If $x_1\leq 1$, then this holds with $\theta=0$ and $C_{\theta}=4$ since $\frac{1-e^{-4x_1y_1}}{y_1}\leq 4x_1\leq 4$.

If $x_1>1$, then we split into the cases (i) $y_1\leq \frac{1}{4x_1}$ and (ii) $y_1> \frac{1}{4x_1}$. In case (i) we can write
\[x_1 \frac{1-e^{-4x_1y_1}}{y_1}\leq 4x_1^2\leq 8 e^{\frac14} e^{\frac12(x_1-\frac{1}{4x_1})^2}\leq  8 e^{\frac14} e^{\frac12(x_1-y_1)^2}.\]
In case (ii), we can write
\[x_1 \frac{1-e^{-4x_1y_1}}{y_1}\leq \frac{x_1}{y_1}.\]
Therefore, it remains to check that $A:=\sup \frac{x_1}{y_1} e^{-\frac12(x_1-y_1)^2}<\infty$, where the supremum is taken over all $x_1>1$ and $y_1> \frac{1}{4x_1}$. Since for $x_1\to 0$ and $x_1\to \infty$ this expression tends to zero, optimizing first over $x_1$, yields that $x_1 = \phi(y_1) = \frac{1}{2} y_1 + \sqrt{\frac{1}{4} y^2_1+1}$. Since $x_1> \frac{1}{4y_1}$ and $x_1<\sqrt{y^2_1+4}$, we find $\frac{1}{16y_1^2}<y^2_1+4$. The latter holds if and only if $y_1^2>\sqrt{\frac{65}{16}}-2$. This implies $y_1>1/10$. Note that
\[B_1:=\sup_{y>1/10} \frac{\phi(y_1)}{y_1} e^{-\frac12(\phi(y_1)-y_1)^2}\leq \sup_{y_1>1/10}\frac{\phi(y_1)}{y_1} <\infty.\]
In the case $y_1\leq 1/10$, the optimal solution $x_1 = \phi(y_1)$ is not feasible and the maximum is attained at $x_1 = \frac{1}{4y_1}$. In this case we obtain
\[B_2 := \sup_{0<y_1\leq 1/10} \frac{1}{4y_1^2} e^{-\frac12(\frac1{4y_1}-y_1)^2} \leq \sup_{t>0} \frac{t}{4} e^{-\frac{t}{32}} e^{\frac14}<\infty.\]
It follows that $A\leq \max\{B_1,B_2\}<\infty$.

%
%
%
%

As a consequence we obtain the following result.
\begin{theorem}\label{thm:anderdemethode}
Let $d\mu = x_1 \ud x$, $p\in (1, \infty)$ and $w\in A_p(\mu)$. Then
the heat semigroup given by \eqref{eq:heatsemi} extends to an analytic semigroup on $L^p(\R^d_+,w)$ and its generator $-A$ has the property that $A$ has a bounded $H^\infty$-calculus with $\omega_{H^\infty}(A) = 0$.
\end{theorem}
Note that this does not directly imply the same for $-\DD$ because it is unclear whether $A = -\DD$ in the above setting, because we do not know whether the domains coincide. Note that the approach presented in Theorem \ref{thm:domain} also works for weights of the form $w(x):=x_1^{\gamma} v(\tx)$ with $v\in A_p$.

Instead of applying Theorem \ref{thm:domain} in the above situation one could also apply the simpler Theorem \ref{thm:domainAp case} with $d\mu(x) = x_1^{\beta} \ud x$ with $\beta\in (0,1)$. Indeed, then $w_{\alpha}\in A_p(\mu)$ if and only if $-1<\alpha+\beta<\beta p  +p-1$. Again one can check condition \eqref{eq:conditionkernel} with left-hand side $\frac{1}{x_1^{\beta}}|H_z(x,y)|$ and for the new measure $\mu$. Therefore, choosing $\beta$ arbitrary close to $1$, we obtain $-\Delta_{\Dir}$ has a bounded $H^\infty$-calculus on $L^p(\R^d_+,w_{\gamma})$ for $\gamma\in (-1,2p-1)$. Finally, let us remark that some work needs to be done in order to obtain Theorem \ref{thm:anderdemethode} in the vector-valued setting using the above approach.

\subsection{Some comments on the case $\gamma \in (-p-1,-1)$}\label{subsec:Dir_Laplace_(-p-1,-1)}

In Theorem~\ref{thm:domainAp case}, Proposition~\ref{prop:wellposedd=1} and Theorem~\ref{thm:domain} we have characterized the generator of the heat semigroup from Proposition~\ref{prop:nonapbdd} for the case $\gamma \in (-1,p-1) \cup (p-1,2p-1)$ as the Dirichlet Laplacian $\Delta_{\Dir}$ with domain $D(\Delta_{\Dir})=W^{2,p}_{\Dir}(\R^{d}_{+},w_{\gamma};X)$.
In this subsection we will discuss the failure of this domain description for the case $\gamma \in (-p-1,-1)$.

Let us start with the one-dimensional case.
The point where the proof of Proposition~\ref{prop:wellposedd=1} does not work for the case $\gamma \in (-p-1,-1)$ is the fact that $\mathcal{S}_{\odd}(\R_{+};X) \nsubseteq W^{2,p}(\R_{+},w_{\gamma};X)$ in that case, which is illustrated by the following example.

\begin{example}\label{ex:subsec:Dir_Laplace_(-p-1,-1)}
Let $p \in [1,\infty)$ and $\gamma \in (-p-1,-1)$. Suppose $u \in \mathcal{S}(\R_{+};X)$ satisfies  $u(0) = u''(0) = 0$. Then $u, u'' \in L^{p}(\R_{+},w_{\gamma};X)$, but
\[
u \in W^{2,p}(\R_{+},w_{\gamma};X) \:\Longlra\: u'(0) = 0.
\]
\end{example}
\begin{proof}
Note that $u, u'' \in W^{1,p}_{0}(\R_{+},w_{\gamma+p};X)$. So $u, u'' \in L^{p}(\R_{+},w_{\gamma};X)$ by Lemma~\ref{lem:Hardy} (or Corollary~\ref{cor:lem:Hardy;Sobolev_emb}). In the same way, $u' \in L^{p}(\R_{+},w_{\gamma};X)$ if $u'(0)=0$. On the other hand, $u' \in L^{p}(\R_{+},w_{\gamma};X)$ only if $u'(0)=0$ by (the proof of) Lemma~\ref{lem:traceAp} \eqref{it:estuniform0}.
\end{proof}

As a consequence of the above example,
\begin{equation}\label{eq:subsec:Dir_Laplace_(-p-1,-1);subsetneq_d=1}
W^{2,p}(\R_{+},w_{\gamma};X) \subsetneqq \left\{u \in L^{p}(\R_{+},w_{\gamma};X): u'' \in L^{p}(\R_{+},w_{\gamma};X) \right\}
\end{equation}
for $p \in [1,\infty)$ and $\gamma \in (-p-1,-1)$, despite of the interpolation inequality from Lemma~\ref{lem:interp2}. Note that here $W^{2,p}(\R_{+},w_{\gamma};X) = W^{2,p}_{\Dir}(\R_{+},w_{\gamma};X)$.

A duality argument yields that the right-hand side space in \eqref{eq:subsec:Dir_Laplace_(-p-1,-1);subsetneq_d=1} actually is the "correct" the domain for the Dirichlet Laplacian $\DD$ on $L^{p}(\R_{+},w_{\gamma};X)$ when $\gamma \in (-p-1,-1)$:
\begin{proposition}\label{prop:subsec:Dir_Laplace_(-p-1,-1);domain_d=1}
Let $p \in (1,\infty)$ and $\gamma \in (-p-1,-1)$.
Then $\DD$, defined as
\[
D(\DD) := \{ u \in L^{p}(\R_{+},w_{\gamma};X) : u'' \in L^{p}(\R_{+},w_{\gamma};X) \},\quad \DD u := u'',
\]
is the generator of the heat semigroup on $L^p(\R_+,w_{\gamma};X)$ given in Proposition \ref{prop:nonapbdd}.
\end{proposition}
\begin{proof}
Let $\gamma'=\frac{-\gamma}{p-1} \in (p'-1,2p'-1)$ be the $p$-dual exponent of $\gamma$ and let $\DD'$ be the Dirichlet Laplacian on $L^{p'}(\R_{+},w_{\gamma'};X^{*})$:
\[
D(\DD') := W^{2,p'}_{\Dir}(\R_{+},w_{\gamma'};X^{*}),\quad \DD'u := u''.
\]
Then, viewing $L^{p}(\R_{+},w_{\gamma};X)$ as closed subspace of $[L^{p'}(\R_{+},w_{\gamma'};X^{*})]^*$, we have that $\DD$ coincides with the realization of $[\DD']^{*}$ in $L^{p}(\R_{+},w_{\gamma};X)$.
To see this, denote the latter operator by $A$.
Given $v \in D(\DD)$, we have, for all $u$ in the dense subspace $C^{\infty}_{c}(\R_{+}) \otimes X^{*}$ of $D(\DD') = W^{2,p'}_{\Dir}(\R_{+},w_{\gamma'};X^{*}) = W^{2,p'}_{0}(\R_{+},w_{\gamma'};X^{*})$ (see Proposition~\ref{prop:densenonAp}),
\begin{align*}
\ip{\DD'u}{v}_{\ip{L^{p'}(\R_{+},w_{\gamma'};X^{*})}{L^{p}(\R_{+},w_{\gamma};X)}}
&= \ip{u''}{v}_{\ip{L^{p'}(\R_{+},w_{\gamma'};X^*)}{L^{p}(\R_{+},w_{\gamma};X)}} \\
&= \ip{u''}{v}_{\ip{\Distr(\R_{+};X^{*})}{\Distr'(\R_{+};X)}} \\
&= \ip{u}{v''}_{\ip{\Distr(\R_{+};X^{*})}{\Distr'(\R_{+};X)}} \\
&= \ip{u}{v''}_{\ip{L^{p'}(\R_{+},w_{\gamma'};X^{*})}{L^{p}(\R_{+},w_{\gamma};X)}},
\end{align*}
showing that $\DD \subset [\DD']^{*}$, and hence $\DD \subset A$.
Given $v \in D(A)$, we have, for all $u \in C^{\infty}_{c}(\R_{+}) \otimes X^{*} \subset D(\DD')$,
\begin{align*}
\ip{u}{Av}_{\ip{\Distr(\R_{+};X^{*})}{\Distr'(\R_{+};X)}}
&= \ip{u}{Av}_{\ip{L^{p'}(\R_{+},w_{\gamma'};X^{*})}{L^{p}(\R_{+},w_{\gamma};X)}} \\
&= \ip{\DD'u}{v}_{\ip{L^{p'}(\R_{+},w_{\gamma'};X^{*})}{L^{p}(\R_{+},w_{\gamma};X)}} \\
&= \ip{u''}{v}_{\ip{L^{p'}(\R_{+},w_{\gamma'};X^{*})}{L^{p}(\R_{+},w_{\gamma};X)}} \\
&= \ip{u''}{v}_{\ip{\Distr(\R_{+};X^{*})}{\Distr'(\R_{+};X)}} \\
&= \ip{u}{v''}_{\ip{\Distr(\R_{+};X^{*})}{\Distr'(\R_{+};X)}},
\end{align*}
and thus $Av=v''$, showing that $A \subset \DD$. Since the heat semigroup on $L^{p}(\R_{+},w_{\gamma};X)$ from Proposition \ref{prop:nonapbdd} is the restriction to $L^{p}(\R_{+},w_{\gamma};X)$ of the strongly continuous adjoint (in the sense of \cite[page~6]{Neerven92_adjoint}) of the heat semigroup on $L^{p'}(\R_{+},w_{\gamma'};X^{*})$ from Proposition~\ref{prop:nonapbdd}, the required result follows Proposition~\ref{prop:wellposedd=1} and \cite[Theorem~1.3.3]{Neerven92_adjoint}.
\end{proof}

Let us next turn to the $d$-dimensional case.
\begin{proposition}\label{prop:subsec:Dir_Laplace_(-p-1,-1);domain_d}
Let $X$ be a UMD space, $p \in (1,\infty)$ and $\gamma \in (-p-1,-1)$.
Then $\DD$, defined as
\[
D(\DD) := \{ u \in L^{p}(\R^{d}_{+},w_{\gamma};X) : \Delta u \in L^{p}(\R^{d}_{+},w_{\gamma};X) \},\quad \DD u := \Delta u,
\]
is the generator of the heat semigroup on $L^p(\R^{d}_+,w_{\gamma};X)$ given in Proposition \ref{prop:nonapbdd}. Moreover,
\[
D(\DD) = \left\{ u \in L^{p}(\R_{+},w_{\gamma};W^{2,p}(\R^{d-1};X)) : \partial_{1}^{2}u \in  L^p(\R^{d}_+,w_{\gamma};X) \right\}.
\]
with an equivalence of norms only depending on $X$, $p$, $d$, $\gamma$.
\end{proposition}
\begin{proof}
The first statement can be proved in the same way as Proposition~\ref{prop:subsec:Dir_Laplace_(-p-1,-1);domain_d=1}, using Theorem~\ref{thm:domain} \eqref{it:generator} instead of Proposition~\ref{prop:wellposedd=1}. The second statement can be proved using the operator sum method as in  Theorem~\ref{thm:domain}, using Proposition~\ref{prop:subsec:Dir_Laplace_(-p-1,-1);domain_d=1} instead of Proposition~\ref{prop:wellposedd=1}.
\end{proof}

\section{$\DD$ on bounded domains}\label{sec:loc}

In this section we will use standard localization arguments to obtain versions of Theorems \ref{thm:domainAp case} and \ref{thm:domain} for bounded $C^2$-domains $\dom\subseteq \R^d$.
In particular it will be shown that the Dirichlet Laplacian $\DD$ on $L^{p}(\dom,w_{\gamma}^{\dom})$ with domain $W^{2,p}_{\Dir}(\dom,w_{\gamma}^{\dom})$ is a closed and densely defined linear operator for which $-\DD$ has a bounded $H^\infty$-calculus of angle zero. Moreover, $(e^{z\DD})_{z\in \C_+}$ is an exponentially stable analytic $C_0$-semigroup.

\subsection{Main results\label{subsec:Mainresbdddom}}

Let the Dirichlet Laplacian $\DD$ on $L^{p}(\dom,w^{\dom}_{\gamma};X)$ be defined by
\[
D(\DD):= W^{2,p}_{\Dir}(\dom,w^{\dom}_{\gamma};X), \qquad \DD u := \Delta u.
\]
Here, $w^{\dom}_{\gamma}(x) = \dist(x,\partial \dom)^{\gamma}$.

The main result of this section is the following version of Theorems \ref{thm:domainAp case} and \ref{thm:domain} for bounded $C^2$-domains.
\begin{theorem}\label{thm:boundeddomainLaplaceX}
Let $\dom$ be a bounded $C^{2}$-domain, $X$ a UMD space, $p\in (1, \infty)$ and $\gamma\in (-1,2p-1)\setminus\{p-1\}$.
Then
\begin{enumerate}[$(1)$]
\item\label{it:generatordomain} $\DD$ is the generator of an analytic $C_0$-semigroup on $L^p(\dom,w_{\gamma}^{\dom};X)$.
\item\label{it:domaindomain} $\DD$ is a closed and densely defined linear operator on $L^p(\dom,w_{\gamma}^{\dom};X)$ with
    \begin{align*}
    D(\DD) = W^{2,p}_{\Dir}(\dom,w_{\gamma}^{\dom};X)
    \end{align*}
    with an equivalence of norms only depending on $X,p,d,\gamma$ and $\dom$.
\item\label{it:Hinftydomain} For every $\varphi>0$ there exists a $\wt{\lambda} \in \R$ such that for all $\lambda\geq \wt{\lambda}$ the operator $\lambda-\DD$ has a bounded $H^\infty$-calculus with $\omega_{H^\infty}(\lambda-\DD) \leq \varphi$.
\end{enumerate}
\end{theorem}

In the scalar case Theorem \ref{thm:boundeddomainLaplaceX} implies the following result where we obtain additional information on the value of $\wt{\lambda}$.
\begin{corollary}\label{cor:boundeddomainLaplaceC}
Let $\dom$ be a bounded $C^{2}$-domain, $p\in (1, \infty)$ and $\gamma\in (-1,2p-1)\setminus\{p-1\}$.
Then the following assertions hold:
\begin{enumerate}[$(1)$]
\item\label{it:spectrumR} $\sigma(-\DD) = \{\lambda_i: i\in \N_0\}\subseteq (0,\infty)$, is not dependent on $p\in (1, \infty)$ and $\gamma\in (-1,2p-1) \setminus \{p-1\}$.
\item\label{it:HinftydomR} For all $\lambda>-\min\{\lambda_i: i\in \N_0\}$, $\lambda-\DD$ has a bounded $H^\infty$-calculus of angle zero.
\item\label{it:domaindomR} $\DD$ is a closed and densely defined operator on $L^p(\R^d_+,w^{\dom}_{\gamma})$ for which there is an equivalence of norms in $D(\DD) = W^{2,p}_{\Dir}(\dom,w_{\gamma}^{\dom})$ and $\DD$ generates an exponentially stable analytic $C_0$-semigroup on $L^p(\dom,w_{\gamma}^{\dom})$.
\item\label{it:soleqdomR} For every $\lambda\geq 0$ and $f\in L^p(\dom,w^{\dom}_{\gamma})$ there exists a unique $u\in W^{2,p}_{\Dir}(\dom,w^{\dom}_{\gamma})$ such that $\lambda u - \DD u= f$, and there exists a constant $C_{p,\gamma,\dom}$ such that
\begin{align*}
\sum_{|\alpha|\leq 2} (\lambda+1)^{1-\frac12|\alpha|} \|D^{\alpha} u\|_{L^p(\dom,w^{\dom}_{\gamma})}  \leq C_{p,\gamma,\dom} \|f\|_{L^p(\dom,w^{\dom}_{\gamma})}.
\end{align*}
\end{enumerate}
\end{corollary}
\begin{proof}
\eqref{it:domaindomR}: All assertions follow from Theorem \ref{thm:boundeddomainLaplaceX} except the exponential stability. The latter will follow from \eqref{it:HinftydomR}.

\eqref{it:HinftydomR}: Fix $\phi>0$. Then, by Theorem \ref{thm:boundeddomainLaplaceX}, for $\lambda>0$ large enough, $\lambda-\Delta$ has a bounded $H^\infty$-calculus with $\omega_{H^\infty}(\lambda-\Delta)\leq \phi$. Next we will show that this holds for small values of $\lambda$ as well.
For this we first prove \eqref{it:spectrumR}.
Note that
\[D(\DD) = W^{2,p}_{\Dir}(\dom,w^{\dom}_{\gamma})\hookrightarrow W^{1,p}(\dom,w^{\dom}_{\gamma})\stackrel{\text{compact}}{\hookrightarrow} L^p(\dom,w^{\dom}_{\gamma}),\]
where the compactness follows from \cite[Theorem 8.8]{OpicGurka89II}.
We obtain that $(\lambda-\DD)^{-1}$ is compact for $\lambda\in \rho(\DD)$. By Riesz' theory of compact operators (see \cite[Chapter 4]{Ru91}), we obtain that $(\lambda-\DD)^{-1}$ has a discrete countable spectrum $\{\mu_i:i\geq 0\}$ and for every $\mu_i\neq 0$, $\mu_i$ is an eigenvalue of $(\lambda-\DD)^{-1}$. Moreover, $0$ is in the spectrum of $(\lambda-\DD)^{-1}$ and is the only accumulation point of the spectrum. We find that $\sigma(-\DD) = \{\mu_i^{-1} - \lambda: i\geq 0 \ \text{with} \ \mu_i\neq 0\}$.
In the case $p=2$ and $\gamma=0$, it is standard that the spectrum has the required form as stated in \eqref{it:spectrumR} (see e.g.\ \cite[Theorem 6.5.1]{Evans10}). Now arguing as in \cite[Corollary 1.6.2]{Davies89} one sees that the spectrum is independent of $\gamma\in (-1,2p-1)\setminus\{p-1\}$ and $p\in (1, \infty)$.

Set $\delta_{\dom}: = \min\{\lambda_i: i\in \N_0\}$. By the analyticity of $z\mapsto (z-\Delta)^{-1}$ for $\C\setminus(-\infty, -\delta_{\dom}]$ and the sectoriality of $\mu-\Delta$ with angle $\leq \phi$, it follows that for any $\lambda>-\delta_{\dom}$ and any $\phi'>\phi$, the operator $\lambda-\Delta$ is sectorial of angle $\leq \phi'$. Therefore, Remark \ref{rem:Hinftytranslate} implies that for any $\lambda>-\delta_{\dom}$, $\lambda-\Delta$ has a bounded $H^\infty$-calculus with $\omega_{H^\infty}(\lambda-\Delta)\leq \phi'$. Finally, since $\phi$ is arbitrary \eqref{it:HinftydomR} follows.

\eqref{it:soleqdomR}: By the sectoriality of $-\tfrac12\delta_{\dom}-\DD$, we have
\[(\lambda+\tfrac12\delta_{\dom}) \|u\|_{L^p(\dom,w_{\gamma}^{\dom})} \leq C\|f\|_{L^p(\dom,w_{\gamma}^{\dom})}\]
for all $\lambda\geq 0$. On the other hand,
\[\|\Delta_{\Dir} u\|_{L^p(\dom,w_{\gamma}^{\dom})}\leq (C+1)\|f\|_{L^p(\dom,w_{\gamma}^{\dom})}.\]
Therefore, since $D(\DD) = W^{2,p}_{\mathrm{Dir}}(\dom,w^{\dom}_{\gamma})$ and $\DD$ is invertible we can deduce
\[\|u\|_{W^{2,p}_{\mathrm{Dir}}(\dom,w^{\dom}_{\gamma})}\lesssim \|f\|_{L^p(\dom,w_{\gamma}^{\dom})}.\]
Finally, the estimates for the first order terms follow from Lemma \ref{lem:int_special_domain} below.
\end{proof}

Corollary \ref{cor:boundeddomainLaplaceC} has the following consequences similar to  Corollaries \ref{cor:weightedmaxAp} and \ref{cor:weightedmaxnonAp}. This time we can allow $\lambda=0$ since the semigroup is exponentially stable. A similar maximal regularity consequence can be deduced from Theorem \ref{thm:boundeddomainLaplaceX} in the $X$-valued case, but this time with additional conditions on $\lambda$.
\begin{corollary}[Heat equation]\label{cor:max-reg_bdd_domain}
Let $p,q\in (1, \infty)$, $v\in A_q(\R)$ and let $\gamma\in (-1, 2p-1)\setminus\{p-1\}$. Let $J\in \{\R_+, \R\}$. Then for all $\lambda\geq 0$ and $f\in L^q(J, v;L^p(\dom,w_{\gamma}^{\dom}))$ there exists a unique $u\in W^{1,q}(J,v;L^p(\dom,w_{\gamma}^\dom))\cap L^{q}(J,v;W^{2,p}_{\Dir}(\dom,w_{\gamma}^{\dom}))$ such that
$u' + (\lambda-\DD) u = f$, $u(0) = 0$ in the case $J = \R_+$. Moreover, the following estimates hold
\[\|u\|_{W^{1,q}(J,v;L^p(\dom,w_{\gamma}^\dom))} + \|u\|_{L^{q}(J,v;W^{2,p}_{\Dir}(\dom,w_{\gamma}^\dom))} \lesssim_{p,q,v,\gamma,d,X} \|f\|_{L^q(J, v;L^p(\dom,w_{\gamma}^\dom))},\]
and
\[\sum_{|\alpha|\leq 1}(\lambda+1)^{1-\frac12|\alpha|} \|D^{\alpha} u\|_{L^{q}(J,v;L^p(\dom,w_{\gamma}^\dom))}\lesssim_{p,q,v,\gamma,d,X} \|f\|_{L^q(J, v;L^p(\dom,w_{\gamma}^\dom))}. \]
\end{corollary}

\begin{remark}
Maximal regularity results have been obtained in \cite[Theorem 2.10]{KK04}, \cite{KimK-H08} and \cite[Theorem 3.13]{KIK14} for the case $\gamma\in (p-1,2p-1)$, for very general elliptic operators $A$ with time-dependent coefficient on bounded $C^1$-domains. The range $\gamma\in (-1, 2p-1)$ has been considered in \cite[Theorem 5.1]{KoNa} under different conditions on the domain. The boundedness of the $H^\infty$-calculus in the weighted case seems to be new for all $\gamma\in (-1,2p-1)$.
\end{remark}

\subsection{The adjoint operator $[\Delta_{\Dir}]^{*}$}

Recall that every UMD space is reflexive. Let $X$ be a reflexive Banach space, $p \in (1,\infty)$ and $\gamma \in \R$. Then $L^{p}(\dom,w_{\gamma}^{\dom};X)$ is a reflexive Banach space with $[L^{p}(\dom,w_{\gamma}^{\dom};X)]^{*} = L^{p'}(\dom,w_{\gamma'}^{\dom};X^{*})$ (see \cite[Corollary 1.3.22]{HNVW1}). Here $\gamma'=\frac{-\gamma}{p-1}$ and we use the unweighted pairing
\[\lb f,g\rb = \int_{\dom}\lb f(x), g(x)\rb \ud x.\]

\begin{proposition}
Let $X$ be a UMD space, $p \in (1,\infty)$ and $\gamma \in (-1,p-1)$. Let $\DD$ be the Dirichlet Laplacian on $L^{p}(\dom,w_{\gamma}^{\dom};X)$ and let $\DD'$ be the Dirichlet Laplacian on $L^{p'}(\dom,w_{\gamma'}^{\dom};X^{*})$. Then $[\DD]^{*} = \DD'$.
\end{proposition}
\begin{proof}
Integration by parts yields that
\[
\ip{\DD'u}{v}_{\ip{L^{p'}(\mathcal{O},w_{\gamma'}^{\dom};X^{*})}{L^{p}(\mathcal{O},w_{\gamma}^{\dom};X)}}
= \ip{u}{\Delta v}_{\ip{L^{p'}(\mathcal{O},w_{\gamma'}^{\dom};X^{*})}{L^{p}(\mathcal{O},w_{\gamma}^{\dom};X)}}
\]
for all $u \in D(\DD')$ and $v \in D(\DD)$, showing that $\DD \subset [\DD']^{*}$ and $\DD' \subset [\DD]^{*}$. The first inclusion gives $\DD' = [\DD']^{**} \subset [\DD]^{*}$. Hence, $[\DD]^{*} = \DD'$.
\end{proof}

\begin{proposition}
Let $X$ be a UMD space, $p \in (1,\infty)$ and $\gamma \in (p-1,2p-1)$. Let $\DD$ be the Dirichlet Laplacian on $L^{p}(\dom,w_{\gamma}^{\dom};X)$. Then
\[
D([\DD]^{*}) = \left\{ u \in L^{p'}(\dom,w_{\gamma'}^{\dom};X^{*}) : \Delta u \in L^{p'}(\dom,w_{\gamma'}^{\dom};X^{*})  \right\}, \qquad [\DD]^{*}u = \Delta u.
\]
\end{proposition}
\begin{proof}
This can be shown in the same way as in the proof of Proposition~\ref{prop:subsec:Dir_Laplace_(-p-1,-1);domain_d=1}.
\end{proof}
Note that in general the above domain is larger than $W^{2,p}(\R_{+},w_{\gamma};X^*)$ (see
\eqref{eq:subsec:Dir_Laplace_(-p-1,-1);subsetneq_d=1}).

\subsection{Intermezzo: identification of $D((-\DD)^{\frac{k}{2}})$\label{sec:compl}}

In order to transfer the results of the previous sections to smooth domains (and in particular to prove Theorem \ref{thm:boundeddomainLaplaceX}) we will use standard arguments. However, in order to use perturbation arguments we need to identify several fractional domain spaces and interpolation spaces. In principle this topic is covered by the literature as well. However, the weighted setting is not available for the class of weights we consider here and requires additional arguments.

We start with a simple interpolation result for general $A_p$-weights.
In the next result we extend the definition of \eqref{eq:Wkpdir} to all $k\in \N_0$ in the following way
\[W^{k,p}_{(\Delta,\Dir)}(\R^d_+,w;X) = \{u\in W^{k,p}(\R^d_+,w;X):\tr(\Delta^j u) = 0 \ \forall j< k/2\}.\]

\begin{proposition}\label{prop:complexinterApR+}
Let $X$ be a UMD space. Let $p\in (1, \infty)$ and let $w\in A_p$ be even. Then for any $k\in \N_1$ and $j\in \{0,\ldots,k\}$ the following holds:
\begin{align*}
[L^p(\R_+^d,w;X), W^{k,p}_{(\Delta,\Dir)}(\R^d_+,w;X)]_{\frac{j}{k}} = W^{j,p}_{(\Delta,\Dir)}(\R^d_+,w;X).
\end{align*}
In particular, for any $k\in \N_0$, $D((-\Delta_{\Dir})^{k/2}) = W^{k,p}_{(\Delta,\Dir)}(\R^d_+,w;X)$.
\end{proposition}
\begin{proof}
To identify the complex interpolation spaces recall from  Lemma \ref{lem:tracech} that $\Eo:W^{k,p}_{(\Delta,\Dir)}(\R^d_+,w;X)\to W^{k,p}_{\odd}(\R^d,w;X)$ is an isomorphism for $k\in \{0, 1, 2\}$. Moreover, from \eqref{eq:Dalphaovu} we see that $\DD$ commutes with $\Eo$. Therefore, the above isomorphism extends to all $k\in \N_0$.

Therefore, by a standard retraction-coretraction argument (see \cite[Theorem 1.2.4]{Tr1} and see \cite[Lemma 5.3]{LMV} for explicit estimates), it is sufficient to prove
\begin{align*}
[L^p_{\odd}(\R^d,w;X), W^{k,p}_{\odd}(\R^d,w;X)]_{\frac{j}{k}} = W^{j,p}_{\odd}(\R^d,w;X).
\end{align*}
Define $R:W^{m,p}(\R^d,w;X)\to W^{m,p}_{\odd}(\R^d,w;X)$ by $R f(x) = (f(x_1,\tx)-f(-x_1,\tx))/2$ and let $S:W^{m,p}_{\odd}(\R^d,w;X)\to W^{m,p}(\R^d,w;X)$ denote the injection. By the symmetry of $w$, $R$ is bounded. Moreover, $RS$ equals the identity operator, and since by Proposition \ref{prop:HWduality} and Theorem \ref{thm:Hspwinterp} we have
$[L^p(\R^d,w;X), W^{k,p}(\R^d,w;X)]_{\frac{j}{k}} = W^{j,p}(\R^d,w;X)$,  the required identity follows from the retraction-coretraction argument again.

The final assertion is clear for even $k$. For odd $k = 2\ell+1$ with $\ell\in \N_0$ by Proposition \ref{prop:BIP}, Theorem \ref{thm:domainAp case} and the result in the even case we can write
\begin{align*}
D((-\Delta_{\Dir})^{k/2}) &=[L^p(\R^d_+,w;X),D((-\Delta_{\Dir})^{\ell})]_{\frac{k}{2\ell}}
\\ & = [L^p(\R^d_+,w;X), W^{2\ell,p}_{(\Delta,\Dir)}(\R^d_+,w;X)]_{\frac{k}{2\ell}}= W^{k,p}_{(\Delta,\Dir)}(\R^d_+,w;X).
\end{align*}
\end{proof}

We can now prove the two main results of this section.
\begin{theorem}\label{thm:DD32}
Let $X$ be a UMD space. Let $p\in (1, \infty)$ and
$\gamma\in (p-1, 2p-1)$. Then
\begin{align*}
D((-\DD)^{1/2})&=[L^p(\R_+^d,w_{\gamma};X), D(\DD)]_{\frac12} = W^{1,p}(\R_+^d,w_{\gamma};X).
\\ D((-\DD)^{3/2}) &= [L^p(\R_+^d,w_{\gamma};X), D(\DD^2)]_{\frac34} = \{u\in W^{3,p}(\R^d_+,w_{\gamma};X): \tr(u) = 0\}.
\end{align*}
\end{theorem}
\begin{proof}
By Theorem \ref{thm:domainAp case} $-\DD$ has bounded imaginary powers. Therefore, by Proposition \ref{prop:BIP} $D((-\DD)^{j/k})=[L^p(\R_+,w_{\gamma};X), D(\DD^k)]_{\frac{j}{k}}$ for all integers $0\leq j\leq k$. It remains to identify the complex interpolation spaces. For $d=1$ we can use Proposition \ref{prop:interp00} and the fact that $W^{2,p}_0(\R_+^d,w_{\gamma};X) = D(\DD)$, and $W^{1,p}_0(\R_+^d,w_{\gamma};X) = W^{1,p}(\R_+^d,w_{\gamma};X)$ for $\gamma>p-1$. For $d\geq 2$ we can use the $d=1$ case and standard results about $\Delta_{d-1}$ combined with \cite[Lemma 9.5]{EPS03} to obtain
\begin{align*}
D((-\DD)^{1/2}) &= D((2-\DD)^{1/2})  = D((1-\Delta_{\Dir,1})^{1/2})\cap D((1-\Delta_{d-1})^{1/2})
\\ & = W^{1,p}(\R_+,w_{\gamma};L^p(\R^{d-1};X))\cap L^p(\R_+,w_{\gamma};W^{1,p}(\R^{d-1};X))
\\ & = W^{1,p}(\R_+^d,w_{\gamma};X).
\end{align*}

To identify $D((-\DD)^{3/2})$ in the case $\gamma>p-1$ we first consider $d=1$. By Theorem \ref{thm:domain} and the previous case one has
\begin{align*}
D((-\DD)^{3/2}) & = \{u\in D(\DD): \DD u\in D((-\DD)^{1/2})\}
\\ & = \{u\in W^{2,p}_{\Dir}(\R_+,w_{\gamma};X): \tr u =0, u''\in W^{1,p}(\R_+,w_{\gamma};X)\}
\\ & = \{u\in W^{3,p}(\R^d_+,w_{\gamma};X): \tr(u) = 0\}.
\end{align*}
If $d\geq 2$, then
\begin{align*}
D((-\DD)^{3/2}) &= D((1-\DD)^{3/2})
\\ & = \{u\in L^p(\R^d_+,w_{\gamma}): (1-\DD)u\in D((2-\DD)^{1/2}), \tr(u) = 0\}
\\ & = \{u\in W^{2,p}_{\Dir}(\R^d_+,w_{\gamma};X): (1-\DD)u\in W^{1,p}(\R_+^d,\gamma;X)\}.
\end{align*}
Observe that
\[W^{1,p}(\R_+^d,w_{\gamma};X) = W^{1,p}(\R_+,w_{\gamma}; L^p(\R^{d-1};X))\cap L^p(\R_+,w_{\gamma}; W^{1,p}(\R^{d-1};X)).\] Thus by the $d=1$ case, the boundedness of $\Delta_{\Dir,1} (1-\DD)^{-1}$ and $\Delta_{d-1} (1-\DD)^{-1}$ (see Corollary \ref{cor:ellipticregnonAp}), we obtain that for $u\in W^{2,p}_{\Dir}(\R^d_+,w_{\gamma};X)$, we have
$(1-\DD)u\in W^{1,p}(\R_+^d,w_\gamma;X)$ if and only if
\begin{align*}
u&\in W^{3,p}(\R_+,w_{\gamma}; L^p(\R^{d-1};X))\cap W^{2,p}(\R_+,w_{\gamma}; W^{1,p}(\R^{d-1};X))\cap \\ & \qquad \cap L^p(\R_+,w_{\gamma}; W^{3,p}(\R^{d-1};X)) \cap W^{1,p}(\R_+,w_{\gamma}; W^{2^p}(\R^{d-1};X)) \\ & = W^{3,p}(\R^d_+,w_{\gamma};X),
\end{align*}
with the required norm estimate. Therefore, the required identity for $D((-\DD)^{3/2})$ follows.
\end{proof}

\subsection{Localization: the proof of Theorem~\ref{thm:boundeddomainLaplaceX}\label{subsec:proofbounddomain}}

As a first step in the localization we prove the following result for $\Delta_{\Dir}$ on small deformations of half-spaces.
\begin{lemma}\label{lem:smallperturb}
Let $X$ be a UMD space.
Let $p \in (1,\infty)$ and $\gamma\in (-1, 2p-1)\setminus\{p-1\}$.
For all $\varphi>0$ there exists
an $\varepsilon>0$ and $\lambda > 0$ such that if $\dom$ is a special $C^{2}_{c}$-domain with $[\dom]_{C^{1}} < \varepsilon$ (see \eqref{DBVP:def:special_domain} and \eqref{DBVP:eq:special_domain;above_graph;C^{k}-number}), then the following assertions hold for $\DD$ on $L^{p}(\dom,w^{\dom}_{\gamma};X)$:
\begin{enumerate}[$(1)$]
\item\label{it:Hinftydom:special} $\lambda-\DD$ has a bounded $H^\infty$-calculus with $\omega_{H^\infty}(\lambda-\Delta_{\Dir}) \leq \varphi$.
\item\label{it:domaindom} $\DD$ is a closed and densely defined operator on $L^p(\R^d_+,w^{\dom}_{\gamma};X)$ for which there is an equivalence of norms in $D(\DD) = W^{2,p}_{\Dir}(\dom,w^{\dom}_{\gamma};X)$.
\end{enumerate}
\end{lemma}

\begin{proof}
Let $\dom$ be a special $C^{2}_{c}$-domain with $[\dom]_{C^{1}} < \varepsilon$. Then we can choose $h\in C^2_{c}(\R^{d-1})$ as in \eqref{eq:Omegaspecial} with $\|h\|_{C_b^1(\R^{d-1})}\leq \varepsilon$.

Let $\Phi$ be as in \eqref{eq:diffeo_special_domain}.
Let $\Delta^{\Phi}:W^{2,1}_{\rm loc}(\R^d_+;X)\to L^2_{\rm loc}(\R^d_+;X)$ be defined by
\[\Delta^{\Phi}  = \Phi_* \Delta (\Phi^{-1})_*,\]
where $\Phi$ is as below \eqref{eq:Omegaspecial}. Let $\Delta^{\Phi}_{\Dir}$ denote the restriction of $\Delta^{\Phi}$ to $D(\Delta^{\Phi}_{\Dir}) = W^{2,p}_{\Dir}(\R^d_+,w^{\dom}_{\gamma};X)$.
By the above transformations, it suffices to prove the result for $\Delta^{\Phi}$ on $L^p(\R^d_+,w^{\dom}_{\gamma};X)$. For this we use the perturbation theorem \cite[Theorem 3.2]{DDHPV}.

Without loss of generality we can take $\varepsilon\in (0,1)$. A simple calculation shows that
\begin{align}\label{eq:formDeltaPhi}
\Delta^{\Phi} = \Delta  + \underbrace{|\nabla h|^2 \partial_1^2 - 2\partial_1 (\nabla h \cdot \nabla_{d-1})}_{=:A} \underbrace{- (\Delta h) \partial_1}_{=:B}
\end{align}
We first apply perturbation theory to obtain a bounded $H^\infty$-calculus for $\Delta_{\Dir} + A$.
By the assumption we have
\[\|A u\|_{L^p(\R^d_+,w^{\dom}_{\gamma};X)} \leq C \varepsilon \|u\|_{W^{2,p}(\R^d_+,w^{\dom}_{\gamma};X)} \leq C'\varepsilon \|(1-\Delta)u \|_{L^p(\R^d_+,w^{\dom}_{\gamma};X)},\]
where in the last step we used Corollary \ref{cor:ellipticregnonAp}. This proves one of the required conditions for the perturbation theorem. In particular, this part is enough to obtain that for any $\varphi>0$ and for $\varepsilon$ small enough $D(\Delta_{\Dir}+A) = D(\Delta_{\Dir})$ and $1-\Delta_{\Dir}-A$ is sectorial of angle $\leq \varphi$ (see \cite[Proposition 2.4.2]{Lun}).

In order to apply \cite[Theorem 3.2]{DDHPV} it remains to show $AD((1-\Delta_{\Dir})^{1+\alpha})\subseteq D((1-\Delta_{\Dir})^{\alpha})$ and
\begin{equation}\label{eq:toprovepertualpha}
\|(1-\Delta_{\Dir})^{\alpha} A u\|\leq C\|(1-\Delta_{\Dir})^{1+\alpha} u\|, \ \ u\in D((1-\Delta_{\Dir})^{1+\alpha}).
\end{equation}
for some $\alpha\in (0,1)$. We will check this for $\alpha = 1/2$. For any $u\in W^{3,p}_{\Dir}(\R^d_+,w^{\dom}_{\gamma};X)$  we have
\[\|Au\|_{W^{1,p}(\R^d_+,w^{\dom}_{\gamma};X)}\leq C \|h\|_{C^2_b}^2 \|u\|_{W^{3,p}(\R^d_+,w^{\dom}_{\gamma};X)}.\]
Therefore, by Proposition \ref{prop:complexinterApR+} and Theorem \ref{thm:DD32}, condition \eqref{eq:toprovepertualpha} follows. Here we used the standard fact $D((-\DD)^{\alpha}) = D((1-\DD)^{\alpha})$, which is true for any sectorial operator and $\alpha>0$. We can conclude that for $\varepsilon\in (0,1)$ small enough, $1-\Delta_{\Dir} -A$ has a bounded $H^\infty$-calculus of angle $\leq \varphi$.

To obtain the same result for $\lambda-\Delta^{\Phi}$ for $\lambda>0$ large enough it remains to apply a lower order perturbation result (see \cite[Proposition 13.1]{KuWe}). For this observe
\[\|Bu\|_{L^p(\R^d_+,w^{\dom}_{\gamma};X)} \leq \|h\|_{C^2_b} \|u\|_{W^{1,p}(\R^d_+,w^{\dom}_{\gamma};X)}, \  \ u\in W^{1,p}_{\Dir}(\R^d_+,w^{\dom}_{\gamma};X).\]
The required estimate follows since by Proposition \ref{prop:complexinterApR+} and Theorem \ref{thm:DD32},
\begin{align*}
W^{1,p}_{\Dir}(\R^d_+,w^{\dom}_{\gamma};X) & = [L^p(\R^d_+,w^{\dom}_{\gamma};X), W^{2,p}_{\Dir}(\R^d_+,w^{\dom}_{\gamma};X)]_{\frac12}
\\ & = [L^p(\R^d_+,w^{\dom}_{\gamma};X), D(1-\Delta_{\Dir} -A)]_{\frac12} = D((1-\Delta_{\Dir} -A)^{1/2}),
\end{align*}
where in the last step we applied Proposition \ref{prop:BIP}.

The two perturbation arguments give $\lambda >0$ such that $\lambda-\DD^{\Phi}$ has a bounded $H^{\infty}$-calculus with $\omega_{H^{\infty}}(\lambda-\DD^{\Phi}) \leq \phi$. Moreover, there is an equivalence of norms in $D(\DD^{\Phi}) = D(\DD) = W^{2,p}(\R^{d}_{+},w_{\gamma};X)$. The desired results follow.
\end{proof}

The following lemma follows from Proposition~\ref{prop:complexinterApR+} and Theorem~\ref{thm:DD32} under a change of coordinates according to the $C^{2}$-diffeomorphism $\Phi$ from \eqref{eq:diffeo_special_domain} and a standard retration-coretraction argument using \eqref{eq:domaindecomoperator}.
\begin{lemma}\label{lem:int_special_domain}
Let $X$ be a UMD space. Let $\dom$ be a bounded $C^2$-domain or a special $C^2_c$-domain, $p \in (1,\infty)$ and $\gamma \in (-1,2p-1)\setminus\{p-1\}$. Then
\[
[L^{p}(\dom,w^{\dom}_{\gamma};X),W^{2,p}_{\Dir}(\dom,w^{\dom}_{\gamma};X)]_{\frac{1}{2}} = W^{1,p}_{\Dir}(\dom,w^{\dom}_{\gamma};X).
\]
\end{lemma}

The next step in the proof of the above theorem is a localization argument. This localization argument is a modification of the one in \cite[Section 8]{DDHPV} combined with the one in \cite[Ch. 8, Sections 4 $\&$ 5]{Krylov2008_book} and results in the next lemma.
On an abstract level the localization argument takes the following form.

\begin{lemma}\label{DBVP:lemma:H_infty-calculus_top_embd_pert}
Let $A$ be a linear operator on a Banach space $X$, $\tilde{A}$ a densely defined closed linear operator on a Banach space $Y$ such that $\tilde{A}$ has a bounded $H^\infty$-calculus. Assume there exists bounded linear mapping $\mathcal{P}:Y\to X$ and $\mathcal{I}:X\to Y$ such that the following conditions hold:
\begin{enumerate}[$(1)$]
\item $\mathcal{P}\mathcal{I} = I$.
\item $\mathcal{I}D(A)\subseteq D(\tilde{A})$ and $\mathcal{P}D(\tilde{A})\subseteq D(A)$.
\item $\tilde{B}:= (\mathcal{I}A - \tilde{A}\mathcal{I})\mathcal{P}:D(\tilde{A}) \longra Y$ and $\tilde{C}:= \mathcal{I}(A\mathcal{P} - \mathcal{P}\tilde{A}):D(\tilde{A}) \longra Y$ both extend to bounded linear operators $[Y,D(\tilde{A})]_{\theta} \longra Y$ for some $\theta \in (0,1)$.
\end{enumerate}
Then $A$ is a closed and densely defined operator and for every $\phi > \omega_{H^{\infty}}(\tilde{A})$ there exists $\mu > 0$ such that $A+\mu$ has a bounded $H^\infty$-calculus with $\omega_{H^{\infty}}(A+\mu) \leq \phi$.
\end{lemma}
\begin{proof}
Let $\phi > \omega_{H^{\infty}}(\tilde{A})$. By a lower order perturbation result (see \cite[Proposition 13.1]{KuWe}), there exist $\tilde{\mu} > 0$ such that $\tilde{A}+\tilde{B}+\tilde{\mu}$ has a bounded $H^\infty$-calculus with $\omega_{H^{\infty}}(\tilde{A}+\tilde{B}+\tilde{\mu}) \leq \phi$. From the definition of $B$ one sees \[\mathcal{I} A = (\tilde{A}+\tilde{B}) \mathcal{I} \ \ \  \text{on $D(A)$}.\]
Since $\tilde{A}+\tilde{B}$ is closed, the injectivity of $\mathcal{I}$ implies that $A$ is closed. Since $\mathcal{P}$ is surjective, we have
\[X=\mathcal{P} Y = \mathcal{P} \overline{D(\tilde{A})}\subseteq \overline{\mathcal{P}D(\tilde{A})}\subseteq \overline{D(A)}\]
Therefore, $A$ is densely defined. Now we will transfer the functional calculus properties of $\tilde{A}+\tilde{B}$ to $A$. For this we claim that for $\mu$ large enough
and $\lambda\in \C \setminus \Sigma_{\phi}$ we have $\lambda\in \rho(A+\mu)$ and
\[R(\lambda,A+\mu) = \mathcal{P}R(\lambda,\tilde{A}+\tilde{B}+\mu)\mathcal{I}.\]
This clearly yields that $A+\mu$ has a bounded $H^\infty$-calculus of angle $\leq \phi$.

In order to prove the claim we first show that given $\lambda \in \rho(\tilde{A}+\tilde{B})$, for $u \in D(A)$ and $f \in X$ it holds that
\begin{align}
(\lambda-A)u = f \ \ \ \Longra \ \ \ u = \mathcal{P}R(\lambda,\tilde{A}+\tilde{B})\mathcal{I}f. \label{DBVP:eq:lemma:H_infty-calculus_top_embd_pert;sol_formula}
\end{align}
Indeed, if $(\lambda-A)u = f$, then since $\mathcal{I}(\lambda-A) = (\lambda-\tilde{A}-\tilde{B})\mathcal{I}$ on $D(A)$, we obtain $(\lambda-\tilde{A}-\tilde{B})\mathcal{I}u = \mathcal{I}f$ and hence the required identity for $u$ follows.
We next prove that if $\mathcal{P}R(\lambda,\tilde{A}+\tilde{B})\mathcal{I}:X \longra D(A)$ is injective, then \eqref{DBVP:eq:lemma:H_infty-calculus_top_embd_pert;sol_formula} becomes an equivalence.
and in this case $\lambda \in \rho(A)$ and
\begin{equation}\label{eq:RlambdaARtilde}
R(\lambda,A) = \mathcal{P}R(\lambda,\tilde{A}+\tilde{B})\mathcal{I}
\end{equation}
To prove the implication $\Longleftarrow$, define $u = \mathcal{P}R(\lambda,\tilde{A}+\tilde{B})\mathcal{I}f$ and $g=(\lambda-A)u$. Then by the implication $\Longrightarrow$ we find $u = \mathcal{P}R(\lambda,\tilde{A}+\tilde{B})\mathcal{I}g$ and thus by injectivity $f=g$ as required and additionally \eqref{eq:RlambdaARtilde} holds.

Next we prove that there exists $\mu \geq \tilde{\mu}>0$ with the property that for all $\lambda \in \C \setminus \Sigma_{\phi}$, $\mathcal{P}R(\lambda,\tilde{A}+\tilde{B}+\mu)\mathcal{I}$ is injective.
Let $f \in X$ be such that $\mathcal{P}\tilde{u} :=\mathcal{P}R(\lambda,\tilde{A}+\tilde{B}+\mu)\mathcal{I}f=0$. Observing that $\tilde{B}=\tilde{B}\mathcal{I}\mathcal{P}$, we get $\tilde{B}\tilde{u}=0$.
So $(\tilde{A}+\mu-\lambda)\tilde{u} = \mathcal{I}f - \tilde{B}\tilde{u} = \mathcal{I}f$, or equivalently, $\tilde{u} = R(\lambda,\tilde{A}+\mu)\mathcal{I}f$.
It follows that
\begin{align*}
0 &= \mathcal{I}(A+\mu-\lambda)\mathcal{P}\tilde{u} = (\mathcal{I}\mathcal{P}(\tilde{A}+\mu-\lambda)+\tilde{C})R(\lambda,\tilde{A}+\mu)\mathcal{I}f \\
 &= \mathcal{I}f + \tilde{C}R(\lambda,\tilde{A}+\mu)\mathcal{I}f.
\end{align*}
Estimating
\begin{align*}
\norm{\tilde{C}R(\lambda,\tilde{A}+\mu)\mathcal{I}f}_{Y}
&\lesssim
\norm{R(\lambda,\tilde{A}+\mu)\mathcal{I}f}_{[Y,D(\tilde{A})]_{\theta}}
\\ & \leq \norm{R(\lambda,\tilde{A}+\mu)\mathcal{I}f}_{Y}^{1-\theta}\norm{R(\lambda,\tilde{A}+\mu)\mathcal{I}f}_{D(\tilde{A})}^{\theta}
\\ &\lesssim |\lambda-\mu|^{\theta-1}\norm{\mathcal{I}f}_{Y} \lesssim_{\phi} |\mu|^{\theta-1}\norm{\mathcal{I}f}_{Y},
\end{align*}
we see that $\tilde{C}R(\lambda,\tilde{A}+\mu)\mathcal{I}$ is a contraction from $X$ to $Y$ when $\mu$ is sufficiently large, in which case $\mathcal{I}f = -\tilde{C}R(\lambda,\tilde{A}+\mu)\mathcal{I}f$ implies that $\mathcal{I}f=0$ and hence $f = 0$. This yields the required injectivity.
\end{proof}

\begin{proof}[Proof of Theorem \ref{thm:boundeddomainLaplaceX}]
In this proof we let $A=\DD$ on $\dom$.
Let $\epsilon > 0$ be as in Lemma \ref{lem:smallperturb}.
Choose a finite open cover $\{V_{n}\}_{n=1}^{N}$ of $\partial \dom$ together with special $C^{2}_{c}$-domains $\{\dom_{n}\}_{n=1}^{N}$ such that
\[
\dom \cap V_{n} = \dom_{n} \cap V_{n} \quad\mbox{and}\quad \partial \dom \cap V_{n} = \partial \dom_{n} \cap V_{n}, \quad\quad n=1,\ldots,N,
\]
and $[\dom_{n}]_{C^{2}} \leq \epsilon$ for $n=1,\ldots,N$.
Let $\{\eta_{n}\}_{n=1}^{N} \subset C^{\infty}_{c}(\R^{d})$, $Y$, $\mathcal{P}$ and  $\mathcal{I}$ be the objects associated to the above sets as in Subsection \ref{subsec:locCk}.
Define the linear operator $\tilde{A}:D(\tilde{A}) \subset Y \longra Y$ as the direct sum $\tilde{A} := \bigoplus_{n=0}^{N}\tilde{A}_{n}$, where
$\tilde{A}_{0}:D(\tilde{A}_{0}) \subset L^{p}(\R^{d};X) \longra L^{p}(\R^{d};X)$ is defined by
\[
D(\tilde{A}_{0}): = W^{2,p}(\R^{d}) \quad\mbox{and}\quad \tilde{A}_{0}u :=\Delta u
\]
and where, for each $n \in \{1,\ldots,N\}$, $\tilde{A}_{n}:D(\tilde{A}_{n}) \subset L^{p}(\dom_{n},w^{\dom_{n}}_{\gamma};X) \longra L^{p}(\dom_{n},w^{\dom_{n}}_{\gamma};X)$
is defined by
\[
D(\tilde{A}_{n}) := W^{2,p}_{\mathrm{Dir}}(\dom_{n},w^{\dom_{n}}_{\gamma};X) \quad\mbox{and}\quad \tilde{A}_{n}u:=\Delta u.
\]
Furthermore, we define $B:D(A) \longra Y$ by $Bu:=([\Delta,\eta_{n}]u)_{n=0}^{N}$ and $C:D(\tilde{A}) \longra X$ by $C\tilde{u} := \sum_{n=0}^{N}[\Delta,\eta_{n}] \tilde{u}$.

By Lemma \ref{lem:smallperturb}, there exists $\mu>0$ such that $\mu - \tilde{A}_{n}$  has a bounded $H^\infty$-calculus with $\omega_{H^{\infty}}(\mu-\tilde{A}_{n}) \leq \phi$ for $n=1,\ldots,N$. Since $-A_{0}$ has a bounded $H^\infty$-calculus with $\omega_{H^{\infty}}(-\tilde{A}_{0}) = 0$, it follows that $\tilde{A}-\mu$ has a bounded $H^\infty$-calculus with $\omega_{H^{\infty}}(\tilde{A}-\mu) \leq \phi$ (see \cite[Example~10.2]{KuWe}).
Moreover, by a combination of Lemmas \ref{lem:smallperturb} and \ref{lem:int_special_domain},
$[L^{p}(\dom_{n},w^{\dom_{n}}_{\gamma};X),D(\tilde{A}_{n})]_{\frac{1}{2}} = W^{1,p}_{\Dir}(\dom_{n},w^{\dom_{n}}_{\gamma};X)$.
Since $[L^{p}(\R^{d;X}),D(\tilde{A}_{0})]_{\frac{1}{2}} = W^{1,p}(\R^{d};X)$ by \cite[Theorems 5.6.9 and 5.6.11]{HNVW1},
it follows that
\begin{align}\label{eq:complex_int_direct_sum}
[Y,D(\tilde{A})]_{\frac{1}{2}}\nonumber
&= [L^{p}(\R^{d};X),D(\tilde{A}_{0})]_{\frac{1}{2}} \oplus \bigoplus_{n=1}^{N} [L^{p}(\dom_{n},w^{\dom_{n}}_{\gamma};X),D(\tilde{A}_{n})]_{\frac{1}{2}} \\
&= W^{1,p}(\R^{d};X)  \oplus \bigoplus_{n=1}^{N} W^{1,p}_{\Dir}(\dom_{n},w^{\dom_{n}}_{\gamma};X).
\end{align}
Note that $\mathcal{I}$ maps $D(A)$ into $D(\tilde{A})$ and that $\mathcal{I}Au=\tilde{A}\mathcal{I}u+Bu$ for every $u \in D(A)$.
Also note that $\mathcal{P}$ maps $D(\tilde{A})$ to $D(A)$ and that $A\mathcal{P}\tilde{u} = \mathcal{P}\tilde{A}\tilde{u} + C\tilde{u}$ for every $\tilde{u} \in D(\tilde{A})$.
 Since each commutator $[\Delta,\eta_{n}]$ is a first order partial differential operator with $C^{\infty}_{c}$-coefficients, it follows that $\mathcal{I}A-\tilde{A}\mathcal{I}$ extends to a bounded linear operator from $W^{1,p}_{\Dir}(\dom,w^{\dom}_{\gamma};X)$ to $Y$.
Since $\mathcal{P}$ is a bounded linear operator from $[Y,D(\tilde{A})]_{\frac{1}{2}}$ to $W^{1,p}_{\Dir}(\dom,w^{\dom}_{\gamma};X)$ in view of \eqref{eq:complex_int_direct_sum}, it follows that $(\mathcal{I}A-\tilde{A}\mathcal{I})\mathcal{P}$ extends to a bounded linear operator from $[Y,D(\tilde{A})]_{\frac{1}{2}}$ to $Y$.
Similarly we see that $\mathcal{I}(A\mathcal{P} - \mathcal{P}\tilde{A})$ extends to a bounded linear operator from $[Y,D(\tilde{A})]_{\frac{1}{2}}$ to $Y$.
An application of Lemma~\ref{DBVP:lemma:H_infty-calculus_top_embd_pert} finishes the proof.
\end{proof}

\section{The heat equation with inhomogeneous boundary conditions\label{sec:inhheat}}

In this section we will consider the heat equation on a smooth domain $\dom\subseteq \R^d$ with inhomogeneous boundary conditions of Dirichlet type. In particular, Theorem \ref{thm:inhmainintro} is a special case of Theorem \ref{thm:max-reg_bdd_domain_bd-data;interval} below. The main novelty is that we consider weights of the form $w_{\gamma}^{\dom}(x) = \dist(x,\partial \dom)^{\gamma}$ with $\gamma\in (p-1, 2p-1)$, which allows us to treat the heat equation with very rough boundary data.

\subsection{Identification of the spatial trace space\label{subsec:spacetrace}}

We begin with an extension of a trace result from \cite{lindemulder2017maximal} to the range $\gamma\in (p-1, 2p-1)$.
\begin{theorem}[Spatial trace space]\label{thm:trace_max-reg_space}
Let $X$ be a UMD space, $k,\ell \in \N\setminus\{0\}$, $p,q \in (1,\infty)$, $v \in A_{q}(\R)$ and $\gamma \in (-1,kp-1)\setminus [\N p-1]$. Let $\dom$ be either $\R^{d}_{+}$ or a bounded $C^{k}$-domain. Then $\tr_{\dom}$ is a retraction from
\[
W^{\ell,q}(\R,v;L^{p}(\dom,w_{\gamma}^{\dom};X)) \cap L^{q}(\R,v;W^{k,p}(\dom,w_{\gamma}^{\dom};X))
\]
to
\[
F^{\ell-\frac{\ell}{k}\frac{1+\gamma}{p}}_{p,q}(\R,v;L^{p}(\partial\dom;X)) \cap
L^{q}(\R,v;B^{k-\frac{1+\gamma}{p}}_{p,p}(\partial\dom;X)).
\]
\end{theorem}

In order to prove this we need a preliminary result.
On the compact $C^{k}$-boundary $\partial\dom$, we define the Besov spaces $B^{s}_{p,q}(\partial\dom;X)$, $p \in (1,\infty)$, $q \in [1,\infty]$ $s \in (0,k) \setminus \N$, by real interpolation:
\[
B^{s}_{p,q}(\partial\dom;X) := (W^{n,p}(\partial\dom;X),W^{n+1,p}(\partial\dom;X))_{\theta,q}
\]
for $s=\theta+n$ with $\theta \in (0,1)$ and $n \in \{0,1,\ldots,k-1\}$.

In the proof of this theorem we use weighted mixed-norm anisotropic Triebel-Lizorkin spaces as considered in \cite[Section~2.4]{lindemulder2017maximal} (see \cite{Lindemulder_master-thesis} for more details); for definitions and notations we simply refer the reader to these references.

As in the standard isotropic case (see \cite{LindF}), we have:
\begin{lemma}\label{lemma:thm:trace_max-reg_space;trivial_embd_F}
Let $\ell,k \in \N \setminus \{0\}$, $p,q \in (1,\infty)$, $v \in A_{q}(\R)$ and $\gamma \in (-1,\infty)$. Then
\begin{align*}
& F^{1,(\frac{1}{k},\frac{1}{\ell})}_{(p,q),1,(d,1)}(\R^{d}_{+} \times \R,(w_{\gamma},v);X)  \\
& \qquad \hookrightarrow W^{l,q}(\R,v;L^{p}(\R^{d}_{+},w_{\gamma};X)) \cap  L^{q}(\R,v;W^{k,p}(\R^{d}_{+},w_{\gamma};X)).
\end{align*}
\end{lemma}
\begin{proof}
We cannot reduce to the $\R^d$-case directly since $L^p(\R^d,w_{\gamma};X)\not\hookrightarrow L^1_{\rm loc}(\R^d;X)$ for $\gamma \geq p-1$ and therefore cannot be seen as a subspace of the distributions on $\R^d$. However, we can proceed as follows. An easy direct argument (see \cite[Remark~3.13]{MeyVersharp} or \cite[Proposition~5.2.31]{Lindemulder_master-thesis}) shows that
\[
\|f\|_{L^{q}(\R,v;L^{p}(\R^{d},w_{\gamma};X))} \lesssim \|f\|_{F^{0,(\frac{1}{k},\frac{1}{\ell})}_{(p,q),1,(d,1)}(\R^{d} \times \R,(w_{\gamma},v);X)}
\]
for all $f \in \mathcal{S}(\R^{d} \times \R;X)$.
Using
\[
L^{q}(\R,v;L^{p}(\R^{d}_{+},w_{\gamma};X)) \hookrightarrow \Distr'(\R^{d}_{+} \times \R;X)
\]
and using density of $\mathcal{S}(\R^{d} \times \R;X)$ in $F^{0,(\frac{1}{k},\frac{1}{\ell})}_{(p,q),1,(d,1)}(\R^{d} \times \R,(w_{\gamma},v);X)$, we find that the restriction operator
\[
R:\mathcal{D}'(\R^{d} \times \R;X) \to \mathcal{D}'(\R^{d}_{+} \times \R;X),\, f \mapsto f_{|\R^{d}_{+} \times \R},
\]
restricts to a bounded linear operator
\[
R: F^{0,(\frac{1}{k},\frac{1}{\ell})}_{(p,q),1,(d,1)}(\R^{d} \times \R,(w_{\gamma},v);X) \longra L^{q}(\R,v;L^{p}(\R^{d}_{+},w_{\gamma};X)).
\]
By \cite[Section~2.4]{lindemulder2017maximal} (see \cite[Proposition 5.2.29]{Lindemulder_master-thesis}), this implies that $R$ is also bounded as an operator
\begin{align*}\label{eq:lemma:thm:trace_max-reg_space;trivial_embd_F}
& R:F^{1,(\frac{1}{k},\frac{1}{\ell})}_{(p,q),1,(d,1)}(\R^{d} \times \R
 ,(w_{\gamma},v);X) \\
& \qquad \longra W^{l,q}(\R,v;L^{p}(\R^{d}_{+},w_{\gamma};X)) \cap  L^{q}(\R,v;W^{k,p}(\R^{d}_{+},w_{\gamma};X)).
\end{align*}
The desired inclusion now follows.
\end{proof}

With a similar argument as in the above proof one can show the following embedding for an arbitrary open set $\dom\subseteq \R^d$:
\begin{equation}\label{eq:TrLWemb}
B^{k}_{p,1}(\dom,w_{\gamma}^{\dom};X) \hookrightarrow F^{k}_{p,1}(\dom,w_{\gamma}^{\dom};X) \hookrightarrow W^{k,p}(\dom,w_{\gamma}^{\dom};X),
\end{equation}
where $k\in \N_0$, $\gamma>-1$ and $p\in [1, \infty)$.

In the proof of Theorem~\ref{thm:trace_max-reg_space} we will furthermore use the following Sobolev embedding, which is a partial extension of Corollary~\ref{cor:lem:Hardy;Sobolev_emb} to the case $k=0$, obtained by dualizing Corollary~\ref{cor:lem:Hardy;Sobolev_emb}.

\begin{proposition}\label{prop:Sobolev_emb_into_neg_smoothness}
Let $X$ be a UMD space, $p \in (1,\infty)$, $m \in \N$ and $\gamma \in (mp-1,(m+1)p-1)$. Let $\dom$ be a bounded $C^m$-domain or a special $C^m_c$-domain. Then
\[
L^{p}(\dom,w_{\gamma}^{\dom};X) \hookrightarrow H^{-m,p}(\dom,w_{\gamma-mp}^{\dom};X).
\]
\end{proposition}
To prove this embedding we need a simple lemma.
\begin{lemma}\label{lem:distrembdualwithoutDir}
Let $X$ be a UMD space, $p \in (1,\infty)$, $m \in \N$ and let $w\in A_p$. Let $\dom\subseteq \R^d$ be a bounded $C^m$-domain or a special $C^m_c$-domain. Then $H^{-m,p}(\dom,w;X)$ is reflexive and
\begin{align}\label{eq:dense1HD}
\mathcal{D}(\dom;X) &\stackrel{d}{\hookrightarrow} H^{-m,p}(\dom,w;X) \stackrel{d}{\hookrightarrow} \mathcal{D}'(\dom;X),
\end{align}
Under the natural pairing, we have
\begin{align}\label{eq:dense2HD}
\mathcal{D}(\dom;X^{*}) \stackrel{d}{\hookrightarrow} [H^{-m,p}(\dom,w;X)]^{*} \stackrel{d}{\hookrightarrow} \mathcal{D}'(\dom;X^{*}),
\end{align}
\begin{align}\label{eq:Hdequality}
[H^{-1,p}(\dom,w;X)]^{*} = W^{m,p'}_{0}(\dom,w';X^{*}).
\end{align}
\end{lemma}
\begin{proof}
The reflexivity of $H^{-m,p}(\dom,w;X)$ follows from Proposition \ref{prop:HWduality}. The continuity of the inclusions in \eqref{eq:dense1HD} are obvious. The density in the first embedding of \eqref{eq:dense1HD} holds because of
Lemma~\ref{lem:densitynonAp} and $L^{p}(\dom,w;X )
\stackrel{d}{\hookrightarrow} H^{-m,p}(\dom,w;X)$.
The density of the second embedding in \eqref{eq:dense1HD} follows from the density of
$\mathcal{D}(\dom;X)$ in $\mathcal{D}'(\dom;X)$.
The dense embeddings \eqref{eq:dense2HD} follow from \eqref{eq:dense1HD}, $\Distr(\dom;X)^* = \Distr'(\dom;X^*)$ and $\Distr'(\dom;X)^* = \Distr(\dom;X^*)$ and the reflexivity of $H^{-m,p}(\dom,w;X)$. To prove \eqref{eq:Hdequality}, by density (see  Lemma \ref{lem:denseApbdr0}) it suffices to prove
\begin{equation}\label{eq:normequivHWdual}
\|f\|_{[H^{-m,p}(\dom,w;X)]^{*}} \eqsim \|f\|_{W^{m,p'}(\R^{d},w';X^{*})}, \ \ \ f\in \mathcal{D}(\dom;X^*).
\end{equation}
Let $f\in \mathcal{D}(\dom;X^*)$. Then, by Proposition \ref{prop:HWduality}, for all $g\in \mathcal{D}(\dom;X)$,
\[|\lb f, g\rb|\leq \|f\|_{W^{m,p'}(\R^{d},w';X^{*})}\|g\|_{H^{-m,p}(\R^{d},w;X)}.\]
Taking the infimum over all such $g$ and using \eqref{eq:dense1HD}, the estimate $\lesssim$ in \eqref{eq:normequivHWdual} follows. For the converse we use Proposition \ref{prop:HWduality} to obtain
\begin{align*}
\|f\|_{W^{m,p'}(\R^{d},w';X^{*})} \eqsim \|f\|_{H^{m,p'}(\R^{d},w';X^{*})}
= \|f\|_{[H^{-m,p}(\R^{d},w;X)]^{*}}.
\end{align*}
For an appropriate $g\in H^{-m,p}(\R^{d},w;X)$ of norm $\leq 1$ we obtain
\begin{align*}
\|f\|_{W^{m,p'}(\R^{d},w';X^{*})} & \lesssim |\lb f,g\rb| = |\lb f,g_{|\dom}\rb|
\leq \|f\|_{[H^{-m,p}(\dom,w;X)]^{*}},
\end{align*}
where we used $\|g_{|\R_+}\|_{H^{-1,p}(\dom,w;X)}\leq 1$.
\end{proof}

\begin{proof}[Proof of Proposition \ref{prop:Sobolev_emb_into_neg_smoothness}]
Let us first note that $X$ is reflexive as a UMD space.
Put $\gamma':= -\frac{\gamma}{p-1} \in (-mp'-1,(1-m)p'-1)$. Then $[w_{\gamma}^{\dom}]'=w_{\gamma'}^{\dom}$ and $[w_{\gamma-mp}^{\dom}]'=w_{\gamma'+mp'}^{\dom}$, the $p$-duals weights of $w_{\gamma}^{\dom}$ and $w_{\gamma-mp}^{\dom}$, respectively.
Note that $\gamma-mp \in (-1,p-1)$ and $\gamma'+mp' \in (-1,p'-1)$, so $w_{\gamma-mp}^{\dom} \in A_{p}$ and $w_{\gamma'+mp'}^{\dom} \in A_{p'}$.
By Corollary~\ref{cor:lem:Hardy;Sobolev_emb} and Lemma~\ref{lem:densitynonAp},
\[
W^{m,p'}_{0}(\dom,w_{\gamma'+mp'}^{\dom};X^{*}) \stackrel{d}{\hookrightarrow} L^{p'}(\dom,w_{\gamma'}^{\dom};X^{*}).
\]
Therefore, Proposition \ref{prop:HWduality} and Lemma \ref{lem:distrembdualwithoutDir} give
that
\begin{align*}
L^{p}(\dom,w_{\gamma}^{\dom};X) &= [L^{p'}(\dom,w_{\gamma'}^{\dom};X^{*})]^{*}
\hookrightarrow [W^{m,p'}_{0}(\dom,w_{\gamma'+mp'}^{\dom};X^{*})]^*
\\ & = [H^{-1,p}(\dom,w_{\gamma-mp}^{\dom};X)]^{**}
 =H^{-m,p}(\dom,w_{\gamma-mp}^{\dom};X),
\end{align*}
where we again used reflexivity of $X$.
\end{proof}

\begin{proof}[Proof of Theorem~\ref{thm:trace_max-reg_space}.]
By a standard localization argument it suffices to consider the case $\dom = \R^{d}_{+}$. The case $\gamma\in (-1,p-1)$ is already considered in \cite[Theorem~2.1 $\&$ Corollary~4.9]{lindemulder2017maximal} (also see \cite[Theorem~4.4]{lindemulder2017maximal}), so from now on we will assume $\gamma\in (mp-1,(m+1)p-1)$ with $m \in \{1,\ldots,k-1\}$ (which implies that $k \geq 2$).

Let us write
\[
\M := W^{\ell,q}(\R,v;L^{p}(\R^{d}_{+},w_{\gamma};X)) \cap L^{q}(\R,v;W^{k,p}(\R^{d}_{+},w_{\gamma};X))
\]
and
\[
\B := F^{\ell-\frac{\ell}{k}\frac{1+\gamma}{p}}_{p,q}(\R,v;L^{p}(\R^{d-1};X)) \cap
L^{q}(\R,v;B^{k-\frac{1+\gamma}{p}}_{p,p}(\R^{d-1};X)).
\]

By Theorem~\ref{thm:Hspwinterp}, Proposition~\ref{prop:Sobolev_emb_into_neg_smoothness}, Corollary~\ref{cor:lem:Hardy;Sobolev_emb} and \cite[Propositions 5.5 $\&$ 5.6]{LMV},
\begin{align*}
\M &\hookrightarrow H^{\ell(1-\frac{m}{k}),q}(\R;v;[L^{p}(\R^{d}_{+},w_{\gamma};X),W^{k,p}(\R^{d}_{+},w_{\gamma};X)]_{\frac{m}{k}}) \\
&\hookrightarrow H^{\ell(1-\frac{m}{k}),q}(\R;v;
[H^{-m,p}(\R^{d}_{+},w_{\gamma-mp};X),W^{k-m,p}(\R^{d}_{+},w_{\gamma-mp};X)]_{\frac{m}{k}}) \\
&= H^{\ell(1-\frac{m}{k}),q}(\R;v;L^{p}(\R^{d}_{+},w_{\gamma-mp};X)).
\end{align*}
Therefore, once applying Corollary~\ref{cor:lem:Hardy;Sobolev_emb},
\begin{align}\label{eq:thm:trace_max-reg_space}
\M \hookrightarrow H^{\ell(1-\frac{m}{k}),q}(\R;v;L^{p}(\R^{d}_{+},w_{\gamma-p};X)) \cap L^{q}(\R,v;W^{k-m,p}(\R^{d}_{+},w_{\gamma-mp};X)),
\end{align}
which reduces the problem to the $A_p$-weighted setting. By \cite[Theorem~2.1 $\&$ Corollary~4.9]{lindemulder2017maximal} (also see \cite[Theorem~4.4]{lindemulder2017maximal}), $\tr_{\partial\R^{d}_{+}}$ is bounded from the last space to
\[
F^{\ell(1-\frac{m}{k})-\frac{\ell(1-\frac{m}{k})}{k-m}\frac{1+\gamma-mp}{p}}_{p,q}(\R,v;L^{p}(\R^{d-1};X)) \cap
L^{q}(\R,v;B^{k-m-\frac{1+\gamma-mp}{p}}_{p,p}(\R^{d-1};X)) = \B.
\]

Finally, that there is a coretraction $\ext_{\partial\R^{d}_{+}}$ corresponding to $\tr_{\partial\R^{d}_{+}}$ simply follows from a combination of
\cite[Theorems~2.1 $\&$ 4.6 $\&$ Remark~4.7]{lindemulder2017maximal} (also see \cite[Example 5.5]{lindemulder2019intersection}) and Lemma~\ref{lemma:thm:trace_max-reg_space;trivial_embd_F}.
\end{proof}

\subsection{Identification of the temporal trace space\label{subsec:spacetrace}}

For $p\in (1, \infty)$, $q \in [1,\infty]$, $\gamma\in (-1, 2p-1)$ and $s \in (0,2)$ we use the following notation:
\begin{align}\label{eq:real_interp_Sobolev_def}
W^{s}_{p,q}(\dom,w_{\gamma}^{\dom};X) := (L^{p}(\dom,w_{\gamma}^{\dom};X),W^{2,p}(\dom,w_{\gamma}^{\dom};X))_{\frac{s}{2},q}.
\end{align}
In the case $\gamma\in (-1, p-1)$ (with general $A_p$-weight) these spaces can be identified with Besov spaces (see \cite[Proposition 6.1]{MeyVersharp}). In the case $\gamma\in (p-1, 2p-1)$ we only have embedding results (see Lemma \ref{lemma:incl_Besov_real_interp_Sobolev} below).

In the next result we identity the temporal trace space.
\begin{theorem}[Temporal trace space]\label{thm:temporal_trace_max-reg}
Let $\dom$ be either $\R^{d}_{+}$ or a bounded $C^{2}$-domain and let $J$ be either $\R$ or $(0,T)$ with $T \in (0,\infty]$. Let $X$ be a UMD space, $p,q \in (1,\infty)$, $\mu \in (-1,q-1)$ and $\gamma \in (-1,2p-1)\setminus\{p-1\}$.
If $1-\frac{1+\mu}{q} \neq \frac{1}{2}\frac{1+\gamma}{p}$, then the temporal trace operator $\tr_{t=0}:u \mapsto u(0)$ is a retraction
\begin{align}\label{eq:thm:temporal_trace_max-reg}
W^{1,q}(J,v_{\mu};L^{p}(\dom,w_{\gamma}^{\dom};X)) \cap L^{q}(J,v_{\mu};W^{2,p}(\dom,w_{\gamma}^{\dom};X)) \longra
W^{2(1-\frac{1+\mu}{q})}_{p,q}(\dom,w_{\gamma}^{\dom};X).
\end{align}
\end{theorem}

It follows from the trace method (see \cite[Section~1.2]{Luninterp} or \cite[Section~1.8]{Tr1}) that $\tr_{t=0}$ is a quotient mapping \eqref{eq:thm:temporal_trace_max-reg}.
The nontrivial fact in the above theorem is to show that it is a retraction. In order to show this we want to apply \cite[Theorem~1.1]{MeyVerTr}/\cite[Theorem~3.4.8]{PrSi}. However, these results can only be applied directly in the special case that the boundary condition vanishes in the real interpolation space. In the case $\gamma\in (-1, p-1)$ this difficulty does not arise because by using a suitable extension operator one can reduce to the case $\dom = \R^d$. To cover the remaining cases we have found a workaround which requires some preparations. The first result is the characterization of the spatial trace of the spaces defined in \eqref{eq:real_interp_Sobolev_def}. The result will be proved further below.

\begin{proposition}\label{prop:real_interp_Sob_traces}
Let $\dom$ be either $\R^{d}_{+}$ or a bounded $C^{2}$-domain.
Let $X$ be a UMD space, $p\in (1, \infty)$, $q \in [1,\infty)$, $\gamma\in (p-1, 2p-1)$ and $s \in (0,2)$. If $s > \frac{1+\gamma}{p}$, then $\tr_{\partial\dom}$ extends to a retraction
\[
W^{s}_{p,q}(\dom,w_{\gamma}^{\dom};X) \longra B^{s-\frac{1+\gamma}{p}}_{p,q}(\partial\dom).
\]
\end{proposition}
In the setting of the above proposition we define, for $s \in (0,2) \setminus \{\frac{1+\gamma}{p}\}$,
\[
W^{s}_{p,q,\Dir}(\dom,w_{\gamma}^{\dom};X) :=
\left\{\begin{array}{ll}
W^{s}_{p,q}(\dom,w_{\gamma}^{\dom};X), & s < \frac{1+\gamma}{p},\\
\{ u \in W^{s}_{p,q}(\dom,w_{\gamma}^{\dom};X) : \tr_{\partial\dom}u=0 \}, & s> \frac{1+\gamma}{p}.
\end{array}\right.
\]
For these spaces we have the following result which will be proved below as well.
\begin{proposition}\label{prop:real_interp_Sob_Dir}
Let $\dom$ be either $\R^{d}_{+}$ or a bounded $C^{2}$-domain.
Let $X$ be a UMD space, $p\in (1, \infty)$, $q \in [1,\infty)$, $\gamma\in (p-1, 2p-1)$ and $s \in (0,2) \setminus \{\frac{1+\gamma}{p}\}$. Then
\[
W^{s}_{p,q,\Dir}(\dom,w_{\gamma}^{\dom};X) = (L^{p}(\dom,w_{\gamma}^{\dom};X),W^{2,p}_{\Dir}(\dom,w_{\gamma}^{\dom};X))_{\frac{s}{2},q}.
\]
\end{proposition}

From Proposition~\ref{prop:interp_W} and reiteration (see \cite[Theorem~1.10.2]{Tr1}) we immediately obtain the following.
\begin{lemma}\label{lem:interpWnWnplus1}
Let $\dom$ be either $\R^{d}_{+}$ or a bounded $C^{2}$-domain.
Let $X$ be a UMD space, $p\in (1, \infty)$, $q \in [1,\infty]$, $\gamma\in (p-1, 2p-1)$ and $s \in (0,2) \setminus \{1\}$.
If $s = \theta+n$ with $\theta \in (0,1)$ and $n \in \{0,1\}$, then
\[
W^{s}_{p,q}(\dom,w_{\gamma}^{\dom};X) = (W^{n,p}(\dom,w_{\gamma}^{\dom};X),W^{n+1,p}(\dom,w_{\gamma}^{\dom};X))_{\theta,q}.
\]
\end{lemma}

\begin{lemma}\label{lemma:incl_Besov_real_interp_Sobolev}
Let $\dom$ be either $\R^{d}_{+}$ or a bounded $C^{2}$-domain.
Let $p\in (1, \infty)$, $q \in [1,\infty]$, $\gamma\in (p-1, 2p-1)$ and $s \in (0,2)$. Then
\begin{equation}\label{eq:lemma:incl_Besov_real_interp_Sobolev}
B^{s}_{p,q}(\dom,w_{\gamma}^{\dom};X) \hookrightarrow W^{s}_{p,q}(\dom,w_{\gamma}^{\dom};X).
\end{equation}
The inclusion is dense if $q < \infty$.
\end{lemma}

For $\gamma$ in the $A_{p}$-range $(-1,p-1)$ it holds that $B^{s}_{p,q}(\dom,w_{\gamma}^{\dom};X) = W^{s}_{p,q}(\dom,w_{\gamma}^{\dom};X)$ (which can be obtained from \cite[Proposition~6.1]{MeyVersharp}).
However, the reverse inclusion to \eqref{eq:lemma:incl_Besov_real_interp_Sobolev} does not hold for $\gamma \in (p-1,2p-1)$, see Remark~\ref{rmk:lemma:incl_Besov_real_interp_Sobolev} below.

\begin{proof}
By \cite[Theorem~3.5]{Bui_Weighted_Beov&Triebel-Lizorkin_spaces:interpolation...} and a retraction-coretraction argument using Rychkov's extension operator (see \cite{LindF,Rychk99}),
\begin{align*}
B^{s}_{p,q}(\dom,w_{\gamma}^{\dom};X) = (F^{0}_{p,1}(\dom,w_{\gamma}^{\dom};X),F^{2}_{p,1}(\dom,w_{\gamma}^{\dom};X))_{\frac{s}{2},q};
\end{align*}
these references are actually in the scalar-valued setting, but the arguments remain valid in the vector-valued setting. The inclusion now follows from \eqref{eq:TrLWemb}. Density follows from Lemma \ref{lem:Kufner}, \cite[Theorem 1.6.2]{Tr1} and the fact that $C^{\infty}_{c}(\overline{\dom};X) \subset B^{s}_{p,q}(\dom,w_{\gamma}^{\dom};X)$.
\end{proof}

\begin{lemma}\label{lemma:real_interp_Sobolev_incl_Bessel-pot}
Let $\dom$ be either $\R^{d}_{+}$ or a bounded $C^{2}$-domain.
Let $X$ be a UMD space, $p\in (1, \infty)$, $q \in [1,\infty]$, $\gamma\in (p-1, 2p-1)$ and $s \in (0,2)$. Then
\[
W^{s}_{p,q}(\dom,w_{\gamma}^{\dom};X) \hookrightarrow B^{s-1}_{p,q}(\dom,w_{\gamma-p}^{\dom};X).
\]
\end{lemma}
\begin{proof}
By Corollary~\ref{cor:lem:Hardy;Sobolev_emb} and Proposition \ref{prop:Sobolev_emb_into_neg_smoothness}
\begin{align*}
W^{s}_{p,q}(\dom,w_{\gamma}^{\dom};X)  & = ((L^p(\dom, w_{\gamma}^{\dom});X), W^{2,p}(\dom,w_{\gamma}^{\dom};X))_{\frac{s}{2},q}
\\ & \hookrightarrow  ((H^{-1,p}(\dom, w_{\gamma-p}^{\dom});X), W^{1,p}(\dom,w_{\gamma-p}^{\dom};X))_{\frac{s}{2},q} \\ & = B^{s-1}_{p,q}(\dom,w_{\gamma-p}^{\dom};X),
\end{align*}
where the last identity follows from \cite[Proposition 6.1]{MeyVersharp} and a retraction-coretraction argument.
\end{proof}

Before we proceed, we recall some trace theory for weighted Besov spaces, for which we refer to \cite{LMV2}. Let $s \in \R$, $p \in (1,\infty)$, $q \in [1,\infty)$ and $\gamma \in (-1,p-1)$.
If $s = \frac{1+\gamma}{p}+k+\theta$ with $k \in \N_{0}$ and $\theta \in (0,1)$, then $\tr_{k}: u \mapsto (\tr u, \ldots,\tr \partial_{1}^{k}u)$ is well-defined on $B^{s}_{p,q}(\R^{d}_{+},w_{\gamma};X)$.
For such $s$ we put
\begin{align*}
B^{s}_{p,q,0}(\R^{d}_{+},w_{\gamma};X) &:= \{ u \in B^{s}_{p,q}(\R^{d}_{+},w_{\gamma};X) : \tr_{k}u = 0 \}.
\end{align*}
For $s<\frac{1+\gamma}{p}$ we put $B^{s}_{p,q,0}(\R^{d}_{+},w_{\gamma};X) := B^{s}_{p,q}(\R^{d}_{+},w_{\gamma};X)$.

The following result follows from \cite{LMV2}.
\begin{lemma}\label{lem:compl_real_interp00_Ap}
Let $X$ be a UMD space, $p \in (1,\infty)$, $q \in [1,\infty]$ and $\gamma \in (-1,p-1)$.
Let $k\in \N$ and $\theta \in (0,1)$ be such that $k\theta \notin \N_{0}+\frac{1+\gamma}{p}$. Then
\begin{align*}
B^{k\theta}_{p,q,0}(\R^{d}_{+},w_{\gamma};X) &= (L^{p}(\R^{d}_{+},w_{\gamma};X),W^{k,p}_{0}(\R^{d}_{+},w_{\gamma};X))_{\theta,q}.
\end{align*}
\end{lemma}

Let ${_{0}}W^{s}_{p,q}(\R^{d}_{+},w_{\gamma};X)$ be defined as the closure of $\{ u \in C^{\infty}_{c}(\overline{\R^{d}_{+}};X) : u_{|\partial\R^{d}_{+}} = 0 \}$ in $W^{s}_{p,q}(\R^{d}_{+},w_{\gamma};X)$. The following identities hold for the real interpolation spaces.

\begin{lemma}\label{lemma:real_interp_Sobolev_Dir}
Let $X$ be a UMD space, $p\in (1, \infty)$, $q \in [1,\infty)$, $\gamma\in (p-1, 2p-1)$ and $s \in (0,2) \setminus \{\frac{1+\gamma}{p}\}$.
Then
\[
{_{0}}W^{s}_{p,q}(\R^{d}_{+},w_{\gamma};X) = (L^{p}(\R^{d}_{+},w_{\gamma};X),W^{2,p}_{\Dir}(\R^{d}_{+},w_{\gamma};X))_{\frac{s}{2},q}
\]
and the map $M$, defined above Lemma~\ref{lem:isomM}, is an isomorphism
\[
M:{_{0}}W^{s}_{p,q}(\R^{d}_{+},w_{\gamma};X) \longra B^{s}_{p,q,0}(\R^{d}_{+},w_{\gamma-p};X).
\]
\end{lemma}

As a consequence of Lemmas \ref{lem:CinftydenserealLW} and \ref{lemma:real_interp_Sobolev_Dir} for $p\in (1, \infty)$, $\gamma\in (p-1, 2p-1)$ and $s\in (0,\frac{\gamma+1}{p})$ we have
\begin{equation}\label{eq:real_interp_Sobolev_Dir}
{_{0}}W^{s}_{p,q}(\R^{d}_{+},w_{\gamma};X) = W^{s}_{p,q}(\R^{d}_{+},w_{\gamma};X)
\end{equation}

\begin{proof}
We first show that
\begin{align}\label{eq:lemma:real_interp_Sobolev_Dir;1}
(L^{p}(\R^{d}_{+},w_{\gamma};X),W^{2,p}_{\Dir}(\R^{d}_{+},w_{\gamma};X))_{\frac{s}{2},q}
\hookrightarrow {_{0}}W^{s}_{p,q}(\R^{d}_{+},w_{\gamma};X).
\end{align}
By Proposition~\ref{prop:densenonAp}, $C^{\infty}_{c}(\R^{d}_+;X) \stackrel{d}{\subset} W^{2,p}_{\Dir}(\R^{d}_{+},w_{\gamma};X)$.
Therefore, $C^{\infty}_{c}(\R^{d}_{+};X)$ is dense in $(L^{p}(\R^{d}_{+},w_{\gamma};X),W^{2,p}_{\Dir}(\R^{d}_{+},w_{\gamma};X))_{\frac{s}{2},q}$ (see \cite[Theorem 1.6.2]{Tr1}).
As
\[
(L^{p}(\R^{d}_{+},w_{\gamma};X),W^{2,p}_{\Dir}(\R^{d}_{+},w_{\gamma};X))_{\frac{s}{2},q}
\hookrightarrow W^{s}_{p,q}(\R^{d}_{+},w_{\gamma};X)
\]
clearly holds, \eqref{eq:lemma:real_interp_Sobolev_Dir;1} follows.

Next we show that $M$ is a bounded operator
\begin{align}\label{eq:lemma:real_interp_Sobolev_Dir;2}
M:{_{0}}W^{s}_{p,q}(\R^{d}_{+},w_{\gamma};X) \longra B^{s}_{p,q,0}(\R^{d}_{+},w_{\gamma-p};X).
\end{align}
Lemma~\ref{lem:isomM} and real interpolation yield that $M$ is a bounded operator
\[
M: W^{s}_{p,q}(\R^{d}_{+},w_{\gamma};X) \longra B^{s}_{p,q}(\R^{d}_{+},w_{\gamma-p};X).
\]
Since
\begin{align*}
M\{ u \in C^{\infty}_{c}(\overline{\R^{d}_{+}};X) : u_{|\partial\R^{d}_{+}} = 0 \}
&\subset \{ u \in C^{\infty}_{c}(\overline{\R^{d}_{+}};X) :
u_{|\partial\R^{d}_{+}} = (\partial_{1}u)_{|\partial\R^{d}_{+}} = 0 \} \\
&\subset B^{s}_{p,q,0}(\R^{d}_{+},w_{\gamma-p};X),
\end{align*}
\eqref{eq:lemma:real_interp_Sobolev_Dir;2} follows.

From a combination of Lemma~\ref{lem:isomM} and Lemma~\ref{lem:compl_real_interp00_Ap} it follows that $M^{-1}$ is a bounded operator
\[
M^{-1}: B^{s}_{p,q,0}(\R^{d}_{+},w_{\gamma-p};X) \longra (L^{p}(\R^{d}_{+},w_{\gamma};X),W^{2,p}_{\Dir}(\R^{d}_{+},w_{\gamma};X))_{\frac{s}{2},q}.
\]
Combining this with \eqref{eq:lemma:real_interp_Sobolev_Dir;1} and \eqref{eq:lemma:real_interp_Sobolev_Dir;2} finishes the proof.
\end{proof}

\begin{lemma}\label{lem:real_interp_Sob_traces}
Let $X$ be a UMD space, $p\in (1, \infty)$, $q \in [1,\infty)$, $\gamma\in (p-1, 2p-1)$ and $s \in (\frac{1+\gamma}{p},2)$. Then $\tr_{\partial\R^{d}_{+}}$ extends to a retraction
\[
W^{s}_{p,q}(\R^{d}_{+},w_{\gamma};X) \longra B^{s-\frac{1+\gamma}{p}}_{p,q}(\R^{d-1};X)
\]
with
\begin{align}\label{eq:lem:real_interp_Sob_traces;kernel}
{_{0}}W^{s}_{p,q}(\R^{d}_{+},w_{\gamma};X) = \{ u \in W^{s}_{p,q}(\R^{d}_{+},w_{\gamma};X) : \tr_{\partial\R^{d}_{+}}u=0 \}.
\end{align}
Moreover, there exists a coretraction $E$ corresponding to $\tr_{\partial\R^{d}_{+}}$ such that
\begin{align}\label{eq:lem:real_interp_Sob_traces;kernel_proj_equiv_norm}
\| u \|_{W^{s}_{p,q}(\R^{d}_{+},w_{\gamma};X)} \eqsim \| u \|_{B^{s}_{p,q}(\R^{d}_{+},w_{\gamma};X)},
\qquad u \in \ker(I-E\circ\tr_{\partial\R^{d}_{+}}).
\end{align}
\end{lemma}
\begin{proof}
By trace theory of weighted $B$-spaces (see \cite{LMV2} and see \cite[Section~4.1]{lindemulder2017maximal} for the anistropic setting) and Lemmas \ref{lemma:incl_Besov_real_interp_Sobolev} and \ref{lemma:real_interp_Sobolev_incl_Bessel-pot}, there is the commutative diagram
\[
\begin{tikzcd}
B^{s}_{p,q}(\R^{d}_{+},w_{\gamma};X) \arrow[r, "d",hook] \arrow[dr,"E", leftarrow]
& W^{s}_{p,q}(\R^{d}_{+},w_{\gamma};X) \arrow[r, hook]
& B^{s-1}_{p,q}(\R^{d}_{+},w_{\gamma-p};X)  \arrow[d,"\tr_{\partial\R^{d}_{+}}"] \\
 &\:\: B^{s-\frac{1+\gamma}{p}}_{p,q}(\R^{d-1};X) \arrow[r, equal]
& B^{s-1-\frac{1+\gamma-p}{p}}_{p,q}(\R^{d-1};X) \\
\end{tikzcd}
\]
for some extension operator $E$. All statements different from \eqref{eq:lem:real_interp_Sob_traces;kernel} directly follow from this.
Next we claim that
\[
\{ u \in B^{s}_{p,q}(\R^{d}_{+},w_{\gamma};X) : \tr_{\partial\R^{d}_{+}}u=0 \} \stackrel{d}{\hookrightarrow} \{ u \in W^{s}_{p,q}(\R^{d}_{+},w_{\gamma};X) : \tr_{\partial\R^{d}_{+}}u=0 \}.
\]
Indeed, if $u \in W^{s}_{p,q}(\R^{d}_{+},w_{\gamma};X)$ satisfies $\tr_{\partial\R^{d}_{+}}u=0$, then we can find $u_n\in B^{s}_{p,q}(\R^{d}_{+},w_{\gamma};X)$ such that $u_n\to u$ in $W^{s}_{p,q}(\R^{d}_{+},w_{\gamma};X)$. Now it remains to set $v_n = u_n - E \tr_{\partial\R^{d}_{+}} u_n \in B^{s}_{p,q}(\R^{d}_{+},w_{\gamma};X)$ and observe that $\tr_{\partial\R^{d}_{+}}v_n= 0$ and $E \tr_{\partial\R^{d}_{+}} u_n\to 0$ in $W^{s}_{p,q}(\R^{d}_{+},w_{\gamma};X)$.

Since $\{ u \in C^{\infty}_{c}(\overline{\R^{d}_{+}};X) : u_{\partial\R^{d}_{+}} = 0 \}$ is dense in the space on the left hand side by \cite{LMV2}, it follows from the claim that it is also dense in the space on the right hand side. This density implies \eqref{eq:lem:real_interp_Sob_traces;kernel}.
\end{proof}

\begin{proof}[Proof of Proposition~\ref{prop:real_interp_Sob_traces}]
The statement simply follows from Lemma~\ref{lem:real_interp_Sob_traces} by a standard localization argument.
\end{proof}

\begin{proof}[Proof of Proposition~\ref{prop:real_interp_Sob_Dir}]
A combination of \eqref{eq:real_interp_Sobolev_Dir} and Lemma \ref{lem:real_interp_Sob_traces} gives the desired statement for the case $\dom=\R^{d}_{+}$, from the general case follows by a standard localization argument.
\end{proof}

\begin{proof}[Proof of Theorem~\ref{thm:temporal_trace_max-reg}]
Let us first establish the asserted boundedness of $\tr_{t=0}$.
It suffices to consider the case $J=\R_{+}$, where the boundedness statement follows from \cite[Proposition~1.2.2]{Luninterp} or \cite[Section~1.8]{Tr1}.

In order to show that there is a coretraction corresponding to $\tr_{t=0}$,
it suffices to consider the case $\dom=\R^{d}_{+}$ and $J=\R$.
The case $\gamma\in (-1, p-1)$ follows from \cite[Equation (38)]{lindemulder2017maximal}, and therefore it remains to consider $\gamma\in (p-1, 2p-1)$.

Let $\delta = 2(1-\frac{1+\mu}{q})$.
In view of Theorem~\ref{thm:domain}, we can apply \cite[Theorem~1.1]{MeyVerTr} or \cite[Theorem~3.4.8]{PrSi} to $-\DD$ on $L^{p}(\R^{d}_{+},w_{\gamma};X)$, which by Proposition~\ref{prop:real_interp_Sob_Dir} gives an extension operator
\begin{align*}
\mathcal{E}_{\Dir} : W^{\delta}_{p,q,\Dir}(\R^{d}_{+},w_{\gamma};X) \longra &
W^{1,q}(\R,v_{\mu};L^{p}(\R^{d}_{+},w_{\gamma};X)) \\ & \qquad \cap L^{q}(\R,v_{\mu};W^{2,p}_{\Dir}(\R^{d}_{+},w_{\gamma};X)).
\end{align*}

If $\delta < \frac{1+\gamma}{p}$, then $W^{\delta}_{p,q,\Dir}(\R^{d}_{+},w_{\gamma};X) = W^{\delta}_{p,q}(\R^{d}_{+},w_{\gamma};X)$ and we can simply take $\mathcal{E}_{\Dir}$ as the required coretraction.

Finally, let us consider the case $\delta> \frac{1+\gamma}{p}$.
In the notation of Lemma~\ref{lem:real_interp_Sob_traces}, put $\pi:= E \circ \tr_{\partial\R^{d}_{+}}$.
Then
\begin{align}\label{eq:thm:temporal_trace_max-reg;direct_sum}
W^{\delta}_{p,q}(\R^{d}_{+},w_{\gamma};X) = \ker(I-\pi) \oplus W^{\delta}_{p,q,\Dir}(\R^{d}_{+},w_{\gamma};X)
\end{align}
under the projection $\pi$ with the norm equivalence \eqref{eq:lem:real_interp_Sob_traces;kernel_proj_equiv_norm} on $\ker(I-\pi)$.
In view of \cite{LindF}, we can apply \cite[Theorem~1.1]{MeyVerTr} or \cite[Theorem~3.4.8]{PrSi} to the realization of $I-\Delta$ in $B^{0}_{p,1}(\R^{d},w_{\gamma};X)$ with domain $B^{2}_{p,1}(\R^{d},w_{\gamma};X)$, which by real interpolation of weighted $B$-spaces (see \cite[Theorem~3.5]{Bui_Weighted_Beov&Triebel-Lizorkin_spaces:interpolation...}) gives an extension operator
\begin{align*}
\mathcal{E}_{\R^{d}}: B^{\delta}_{p,q}(\R^{d},w_{\gamma};X) \to &
W^{1,q}(\R,v_{\mu};B^{0}_{p,1}(\R^{d},w_{\gamma};X)) \cap L^{q}(\R,v_{\mu};B^{2}_{p,1}(\R^{d},w_{\gamma};X)).
\end{align*}
By extension (using for instance Rychkov's extension operator \cite{Rychk99}) and restriction we obtain an extension operator $\mathcal{E}_{\R^{d}_{+}}$ which maps $B^{\delta}_{p,q}(\R^{d}_{+},w_{\gamma};X)$ into
\begin{align*}
& W^{1,q}(\R,v_{\mu};B^{0}_{p,1}(\R^{d}_{+},w_{\gamma};X)) \cap L^{q}(\R,v_{\mu};B^{2}_{p,1}(\R^{d}_{+},w_{\gamma};X))
\\ & \qquad \hookrightarrow  W^{1,q}(\R,v_{\mu};L^{p}(\R^{d}_{+},w_{\gamma};X))
\cap L^{q}(\R,v_{\mu};W^{2,p}(\R^{d}_{+},w_{\gamma};X)),
\end{align*}
where the embedding follows from \eqref{eq:TrLWemb}.
By \eqref{eq:thm:temporal_trace_max-reg;direct_sum} and \eqref{eq:lem:real_interp_Sob_traces;kernel_proj_equiv_norm}, $\mathcal{E}:=\mathcal{E}_{\R^{d}_{+}}\pi+\mathcal{E}_{\Dir}(I-\pi)$ defines a coretraction corresponding to~$\tr_{\partial\dom}$.
\end{proof}

\begin{remark}\label{rmk:lemma:incl_Besov_real_interp_Sobolev}
Let $p\in (1, \infty)$, $q \in [1,\infty)$, $\gamma\in (-1, 2p-1) \setminus \{p-1\}$ and $s \in (0,2) \setminus \{\frac{1+\gamma}{p}\}$. Then
\begin{equation*}
W^{s}_{p,q}(\R^{d}_{+},w_{\gamma};X)  \hookrightarrow B^{s}_{p,\infty}(\R^{d}_{+},w_{\gamma};X) \quad \Longra \quad \gamma \in (-1,p-1).
\end{equation*}
\end{remark}
\begin{proof}
Assume there is the inclusion $W^{s}_{p,q}(\R^{d}_{+},w_{\gamma};X)  \hookrightarrow B^{s}_{p,\infty}(\R^{d}_{+},w_{\gamma};X)$. Considering the linear mapping $u \mapsto u \otimes x$ for some $x \in X \setminus \{0\}$, we find $W^{s}_{p,q}(\R^{d}_{+},w_{\gamma})  \hookrightarrow B^{s}_{p,\infty}(\R^{d}_{+},w_{\gamma})$. In particular,
\begin{equation}\label{eq:rmk:lemma:incl_Besov_real_interp_Sobolev;1}
W^{s}_{p,q,\Dir}(\R^{d}_{+},w_{\gamma})  \hookrightarrow B^{s}_{p,\infty,\Dir}(\R^{d}_{+},w_{\gamma}).
\end{equation}

Consider the interpolation-extrapolation scale $\{E_{\eta} : \eta \in [-1,\infty)\}$ generated by the operator $(1-\Delta_{\Dir})$ on $L^{p}(\R^{d}_{+},w_{\gamma})$ and the complex interpolation functors $[\,\cdot\,,\,\cdot\,]_{\theta}$, $\theta \in (0,1)$, the interpolation-extrapolation scale $\{E_{\eta,q} : \eta \in [-1,\infty)\}$ generated by the operator $(1-\Delta_{\Dir})$ on $L^{p}(\R^{d}_{+},w_{\gamma})$ and the real interpolation functors $(\,\cdot\,,\,\cdot\,)_{\theta,q}$, $\theta \in (0,1)$, and the interpolation-extrapolation scale $\{F_{\eta,\infty} : \eta \in [-1,\infty)\}$ generated by the operator $(1-\Delta_{\Dir})$ on $B^{0}_{p,\infty}(\R^{d}_{+},w_{\gamma})$ and the complex interpolation functors $[\,\cdot\,,\,\cdot\,]_{\theta}$, $\theta \in (0,1)$ (see \cite[Section~V.1.5]{Am}); the operator $(1-\Delta_{\Dir})$ on $B^{0}_{p,\infty}(\R^{d}_{+},w_{\gamma})$ is considered in \cite{LindF}.
By Proposition~\ref{prop:real_interp_Sob_Dir} and \cite{LindF},
\begin{equation}\label{eq:rmk:lemma:incl_Besov_real_interp_Sobolev;4}
E_{\eta,q} = W^{2\eta}_{p,q,\Dir}(\R^{d}_{+},w_{\gamma}), \qquad \eta \in (0,1) \setminus \left\{\frac{1+\gamma}{2p}\right\}.
\end{equation}
and
\begin{equation}\label{eq:rmk:lemma:incl_Besov_real_interp_Sobolev;5}
F_{\eta,\infty} = B^{2\eta}_{p,\infty,\Dir}(\R^{d}_{+},w_{\gamma}), \qquad \eta \in \left(\frac{1+\gamma}{2p}-1,\frac{1+\gamma}{2p}+1\right) \setminus \left\{\frac{1+\gamma}{2p}\right\},
\end{equation}
respectively.
Now \eqref{eq:rmk:lemma:incl_Besov_real_interp_Sobolev;1}, \eqref{eq:rmk:lemma:incl_Besov_real_interp_Sobolev;4} and \eqref{eq:rmk:lemma:incl_Besov_real_interp_Sobolev;5} imply $E_{\frac{s}{2},q} \hookrightarrow F_{\frac{s}{2},\infty}$ and by lifting we obtain
\begin{equation}\label{eq:rmk:lemma:incl_Besov_real_interp_Sobolev;2}
E_{\eta,q} \hookrightarrow F_{\eta,\infty}, \qquad \eta \in \left[\frac{s}{2} + \Z \right] \cap [-1,\infty).
\end{equation}
By the reiteration property from \cite[Theorem~V.1.5.4]{Am}, $E_{0} = [E_{-1},E_{1}]_{\frac{1}{2}}$.
So, by \cite[Theorem~1.10.3.1]{Tr1},
\begin{equation}\label{eq:rmk:lemma:incl_Besov_real_interp_Sobolev;3}
E_{0} = [E_{-1},E_{1}]_{\frac{1}{2}} \hookrightarrow
(E_{-1},E_{1})_{\frac{1}{2},\infty}.
\end{equation}
Doing a reiteration (\cite[Theorem~1.10.2]{Tr1} and \cite[Theorem~V.1.5.4]{Am}), we find
\begin{align}
(E_{-1},E_{1})_{\frac{1}{2},\infty} &= ((E_{-1},E_{1})_{\frac{s}{4},q},(E_{-1},E_{1})_{\frac{s}{4}+\frac{1}{2},q})_{1-\frac{s}{2},\infty} \nonumber \\
&= ((E_{-1},[E_{-1},E_{1}]_{\frac{1}{2}})_{\frac{s}{2},q},([E_{-1},E_{1}]_{\frac{1}{2}},E_{1})_{\frac{s}{2},q})_{1-\frac{s}{2},\infty} \nonumber \\
&= ((E_{-1},E_{0})_{\frac{s}{2},q},(E_{0},E_{1})_{\frac{s}{2},q})_{1-\frac{s}{2},\infty} \nonumber    \\
&= (E_{\frac{s}{2}-1,q},E_{\frac{s}{2},q})_{1-\frac{s}{2},\infty}   \stackrel{\eqref{eq:rmk:lemma:incl_Besov_real_interp_Sobolev;2}}{\hookrightarrow} (F_{\frac{s}{2}-1,\infty},F_{\frac{s}{2},\infty})_{1-\frac{s}{2},\infty} \nonumber \\
&= ([F_{\frac{s}{2}-1,\infty},F_{\frac{s}{2},\infty}]_{\eta_{0}-\frac{s}{2}+1},[F_{\frac{s}{2}-1},F_{\frac{s}{2}}]_{\eta_{1}-\frac{s}{2}+1})_{\lambda,\infty} \nonumber \\
&= (F_{\eta_{0},\infty},F_{\eta_{1},\infty})_{\lambda,\infty} \label{eq:rmk:lemma:incl_Besov_real_interp_Sobolev;6}
\end{align}
for any $\eta_{0},\eta_{1} \in \R$ and $\lambda \in (0,1)$ with $\frac{s}{2}-1 < \eta_{0} < \eta_{1} < \frac{s}{2}$ and $0=(1-\lambda)\eta_{0}+\lambda\eta_{1}$.
Now pick $\eta_{0},\eta_{1} \in \R$ and $\lambda \in (0,1)$ with $\frac{s}{2}-1 < \eta_{0} < \eta_{1} < \frac{s}{2}$ and $0=(1-\lambda)\eta_{0}+\lambda\eta_{1}$ such that $\eta_{0},\eta_{1} \in (\frac{1+\gamma}{2p}-1,\frac{1+\gamma}{2p})$. Then
\begin{equation}\label{eq:rmk:lemma:incl_Besov_real_interp_Sobolev;7}
(F_{\eta_{0},\infty},F_{\eta_{1},\infty})_{\lambda,\infty} \stackrel{\eqref{eq:rmk:lemma:incl_Besov_real_interp_Sobolev;5}}{=} (B^{2\eta_{0}}_{p,\infty}(\R^{d}_{+},w_{\gamma}),B^{2\eta_{1}}_{p,\infty}(\R^{d}_{+},w_{\gamma}))_{\lambda,\infty} = B^{0}_{p,\infty}(\R^{d}_{+},w_{\gamma})
\end{equation}
by real interpolation of weighted $B$-spaces (see \cite[Theorem~3.5]{Bui_Weighted_Beov&Triebel-Lizorkin_spaces:interpolation...}).
Combining \eqref{eq:rmk:lemma:incl_Besov_real_interp_Sobolev;3},
\eqref{eq:rmk:lemma:incl_Besov_real_interp_Sobolev;6} and \eqref{eq:rmk:lemma:incl_Besov_real_interp_Sobolev;7} gives
\begin{equation}\label{eq:rmk:lemma:incl_Besov_real_interp_Sobolev;8}
L^{p}(\R^{d}_{+},w_{\gamma}) \hookrightarrow B^{0}_{p,\infty}(\R^{d}_{+},w_{\gamma}).
\end{equation}

We finally show that the inclusion \eqref{eq:rmk:lemma:incl_Besov_real_interp_Sobolev;8} implies $\gamma \in (-1,p-1)$. Taking odd extensions  in \eqref{eq:rmk:lemma:incl_Besov_real_interp_Sobolev;8} (see \cite{LindF}) gives
\[
(\mathcal{S}_{\mathrm{odd}}(\R^{d}),\|\,\cdot\,\|_{L^{p}(\R^{d},w_{\gamma})}) \hookrightarrow B^{0}_{p,\infty}(\R^{d},w_{\gamma}).
\]
Now a slight modification of the argument given in \cite[Remark~3.13]{MeyVersharp} gives $C,c>0$ such that
\[
 \frac{1}{|Q|}\int_{Q}w_{\gamma}(x) \,dx \,\cdot\, \left( \frac{1}{|Q|}\int_{Q}w_{\gamma}(x)^{-\frac{1}{p-1}}\,dx \right)^{p-1} \leq C
\]
for all cubes $Q \subset \R^{d}_{+}$ with $|Q| \leq c$. A computation as in \cite[Example~9.1.7]{GrafakosM} shows that $\gamma \in (-1,p-1)$.
\end{proof}

\subsection{Weighted $L^{q}$-$L^{p}$-maximal regularity\label{subsec:weightedLqLpinh}}

Let us first introduce some notation.
Let $\dom$ be a bounded $C^{2}$-domain. Let $p,q\in (1, \infty)$, $v\in A_q(\R)$ and $\gamma\in (-1, 2p-1)$. For an interval $J \subset \R$ we set $\mathbb{D}^{q,p}_{v,\gamma}(J) := L^q(J, v;L^p(\dom,w_{\gamma}^{\dom}))$,
\[
\M^{q,p}_{v,\gamma}(J) := W^{1,q}(J,v;L^p(\dom,w_{\gamma}^\dom))\cap L^{q}(J,v;W^{2,p}(\dom,w_{\gamma}^{\dom}))
\]
and
\[
\B^{q,p}_{v,\gamma}(J) := F^{1-\frac{1}{2}\frac{1+\gamma}{p}}_{q,p}(J,v;L^{p}(\partial\dom)) \cap
L^{q}(J,v;B^{2-\frac{1+\gamma}{p}}_{p,p}(\partial\dom)).
\]
For the power weight $v=v_{\mu}$, with $\mu \in (-1,q-1)$, we simply replace $v$ by $\mu$ in the subscripts: $\mathbb{D}^{q,p}_{\mu,\gamma}(J) := \mathbb{D}^{q,p}_{v_{\mu},\gamma}(J)$, $\M^{q,p}_{\mu,\gamma}(J) := \M^{q,p}_{v_{\mu},\gamma}(J)$ and $\B^{q,p}_{\mu,\gamma}(J)  := \B^{q,p}_{v_{\mu},\gamma}(J)$.

\begin{theorem}[Heat equation]\label{thm:max-reg_bdd_domain_bd-data}
Let $\dom$ be a bounded $C^{2}$-domain.
Let $p,q\in (1, \infty)$, $v\in A_q(\R)$ and $\gamma\in (-1, 2p-1)\setminus\{p-1\}$.
For all $\lambda \geq 0$, $u \mapsto (u'+(\lambda-\Delta)u,\tr_{\partial\dom}u)$ defines an isomorphism of Banach spaces $\M^{q,p}_{v,\gamma}(\R) \longra \mathbb{D}^{q,p}_{v,\gamma}(\R) \oplus \B^{q,p}_{v,\gamma}(\R)$; in particular, for all $\lambda \geq 0$, $f \in \mathbb{D}^{q,p}_{v,\gamma}(\R)$ and $g \in \B^{q,p}_{v,\gamma}(\R)$, there exists a unique solution $u \in \M^{q,p}_{v,\gamma}(\R)$ of the parabolic boundary value problem
\begin{align*}
\left\{\begin{array}{rl}
u' + (\lambda-\Delta) u &= f, \\
\tr_{\partial\dom}u &=g. \\
\end{array}\right.
\end{align*}
Moreover, there are the estimates
\[
\|u\|_{\M^{q,p}_{v,\gamma}(\R)} \eqsim_{p,q,v,\gamma,d,\lambda} \|f\|_{\mathbb{D}^{q,p}_{v,\gamma}(\R)} + \|g\|_{\B^{q,p}_{v,\gamma}(\R)}.
\]
\end{theorem}
\begin{proof}
The required boundedness of the mapping $u \mapsto (u'+(\lambda-\Delta)u,\tr_{\partial\dom}u)$ follows from Theorem~\ref{thm:trace_max-reg_space} while the injectivity follows from Corollary~\ref{cor:max-reg_bdd_domain}. So it remains to be shown that it has a bounded right-inverse, i.e.\ there is a bounded solution operator to the associated parabolic boundary value problem.
Using Theorem~\ref{thm:trace_max-reg_space} we will reduce to the case $g=0$.
After this reduction, the desired result follows from Corollary~\ref{cor:max-reg_bdd_domain}. Finally, to give the reduction to $g=0$, write $U = u-\ext_{\partial\R^{d}_{+}} g$, where $\ext_{\partial\R^{d}_{+}}:\B^{q,p}_{v,\gamma}(\R)\to \M^{q,p}_{v,\gamma}(\R)$ is the coretraction of $\tr_{\partial\dom}$ of Theorem ~\ref{thm:trace_max-reg_space}. Then $U$ satisfies
$U' + (\lambda-\Delta) U = F$ and $\tr_{\partial\dom}U =0$ where
\[F = f - (\frac{d}{dt} +\lambda-\Delta)\ext_{\partial\R^{d}_{+}} g.\]
Now Corollary~\ref{cor:max-reg_bdd_domain} gives
\[\|U\|_{\M^{q,p}_{v,\gamma}(\R)}\leq C\|F\|_{\mathbb{D}^{q,p}_{v,\gamma}(\R)}\leq C\|f\|_{\mathbb{D}^{q,p}_{v,\gamma}(\R)} + \tilde{C} \|g\|_{\B^{q,p}_{v,\gamma}(\R)}.\]
The corresponding estimate for $u$ follows from this.
\end{proof}

As a consequence of the above theorem we obtain the following corresponding result on time intervals $J=(0,T)$ with $T \in (0,\infty]$ in the case of the power weight $v=v_{\mu}$ (with $\mu \in (-1,q-1)$), where we need to take initial values into account.

For the initial data we need to introduce the space
\[
\mathbb{I}^{q,p}_{\mu,\gamma} := W^{2(1-\frac{1+\mu}{q})}_{p,q}(\dom,w_{\gamma}^{\dom}) \stackrel{\eqref{eq:real_interp_Sobolev_def}}{=} (L^{p}(\dom,w_{\gamma}^{\dom}),W^{2,p}(\dom,w_{\gamma}^{\dom}))_{1-\frac{1+\mu}{q},q}.
\]
Recall from Lemma \ref{lemma:incl_Besov_real_interp_Sobolev} that $B^{s}_{p,q}(\dom,w_{\gamma}^{\dom};X)\hookrightarrow \mathbb{I}^{q,p}_{\mu,\gamma}$ with equality if $\gamma\in (-1, p-1)$.

Concerning the compatability condition in the space of initial-boundary data $\mathbb{IB}^{q,p}_{\mu,\gamma}(J)$ below, let us note the following. Assume $1-\frac{1+\mu}{q} > \frac{1}{2}\frac{1+\gamma}{p}$.
Then, on the one hand, by Proposition~\ref{prop:real_interp_Sob_traces}, there is a well-defined trace operator $\tr_{\partial\dom}$ on $\mathbb{I}^{q,p}_{\mu,\gamma}(J)$; in fact, $\tr_{\partial\dom}$ is a retraction from $\mathbb{I}^{q,p}_{\mu,\gamma}$ to $B^{2(1-\frac{1+\mu}{q})-\frac{1+\gamma}{p}}_{p,q}(\partial\dom;X)$.
On the other hand, as a consequence of \cite[Theorem~1.1]{MeyVerTr}, $\tr_{t=0}:g \mapsto g(0)$ is a well-defined retraction from $\mathbb{B}^{q,p}_{\mu,\gamma}(J)$ to $B^{2(1-\frac{1+\mu}{q})-\frac{1+\gamma}{p}}_{p,q}(\partial\dom;X)$.
Motivated by this we set
\begin{align*}
\mathbb{IB}^{q,p}_{\mu,\gamma}(J) := \left\{ (g,u_{0}) \in \mathbb{B}^{q,p}_{\mu,\gamma}(J) \oplus \mathbb{I}^{q,p}_{\mu,\gamma} : g(0) = \tr_{\partial\mathscr{O}}u_{0} \:\text{when $1-\frac{1+\mu}{q} > \frac{1}{2}\frac{1+\gamma}{p}$} \right\}.
\end{align*}

Now we can state the main result for the initial value problem with inhomogeneous boundary condition.
\begin{theorem}[Heat equation]\label{thm:max-reg_bdd_domain_bd-data;interval}
Let $\dom$ be a bounded $C^{2}$-domain and let $J=(0,T)$ with $T \in (0,\infty]$.
Let $p,q\in (1, \infty)$, $\mu \in (-1,q-1)$ and $\gamma\in (-1, 2p-1)\setminus\{p-1\}$ with $1-\frac{1+\mu}{q} \neq \frac{1}{2}\frac{1+\gamma}{p}$.
For all $\lambda \geq 0$,
\begin{align*}
\M^{q,p}_{\mu,\gamma}(J) \longra \mathbb{D}^{q,p}_{\mu,\gamma}(J) \oplus \mathbb{IB}^{q,p}_{\mu,\gamma}(J),\: u \mapsto (u'+(\lambda-\Delta)u,\tr_{\partial\dom}u,u(0))
\end{align*}
defines an isomorphism of Banach spaces;
in particular, for all $\lambda \geq 0$, $f \in \mathbb{D}^{q,p}_{\mu,\gamma}$ and $g \in \B^{q,p}_{\mu,\gamma}$, there exists a unique solution $u \in \M^{q,p}_{\mu,\gamma}$ of the parabolic initial-boundary value problem
\begin{equation*}\label{eq:thm:max-reg_bdd_domain_bd-data;interval;PBVP}
\left\{\begin{array}{rl}
u' + (\lambda-\Delta) u &= f, \\
\tr_{\partial\dom}u &=g, \\
u(0) &= u_{0}.
\end{array}\right.
\end{equation*}
Moreover, there are the estimates
\[
\|u\|_{\M^{q,p}_{\mu,\gamma}(J)} \eqsim_{p,q,\mu,\gamma,d,\lambda} \|f\|_{\mathbb{D}^{q,p}_{\mu,\gamma}(J)} + \|(g,u_{0})\|_{\mathbb{IB}^{q,p}_{\mu,\gamma}(J)}.
\]
\end{theorem}

In the proof of the theorem we will use the following notation:
\[
{_{0}}\mathbb{B}^{q,p}_{\mu,\gamma}(I) := \left\{\begin{array}{ll}
\mathbb{B}^{q,p}_{\mu,\gamma}(I), & 1-\frac{1+\mu}{q} < \frac{1}{2}\frac{1+\gamma}{p}, \\
\{g \in \mathbb{B}^{q,p}_{\mu,\gamma}(I): g(0)=0\},& 1-\frac{1+\mu}{q} > \frac{1}{2}\frac{1+\gamma}{p},
\end{array}\right.
\]
and ${_{0}}\M^{q,p}_{\mu,\gamma}(I) := \{u \in \M^{q,p}_{\mu,\gamma}(I): u(0)=0\}$, where $I \in \{\R_{+},\R\}$. We will furthermore use the following lemma.

\begin{lemma}\label{lemma:thm:max-reg_bdd_domain_bd-data;interval;ext0}
Let the notation and assumptions be as in Theorem~\ref{thm:max-reg_bdd_domain_bd-data;interval}.
Then operator $E_{0}$ of extension by zero from $\R_{+}$ to $\R$ is a bounded linear operator
from ${_{0}}\mathbb{B}^{q,p}_{\mu,\gamma}(\R_{+})$ to $\mathbb{B}^{q,p}_{\mu,\gamma}(\R)$.
\end{lemma}
\begin{proof}
It suffices to show that
\begin{equation*}
E_{0} \in \mathcal{B}({_{0}}F^{1-\frac{1}{2}\frac{1+\gamma}{p}}_{q,p}(\R_{+},v_{\mu};L^{p}(\partial\dom)),
F^{1-\frac{1}{2}\frac{1+\gamma}{p}}_{q,p}(\R,v_{\mu};L^{p}(\partial\dom))).
\end{equation*}
Using \cite[Theorem~1.3]{MeyVerpoint}, which says that $1_{\R_{+}}$ is a pointwise multiplier on $F^{s}_{p,q}(\R,v_{\mu};X)$ for $s \in (\frac{1+\mu}{q}-1,\frac{1+\mu}{q})$ and a Banach space $X$, this can be shown as in \cite{LMV2}. We would like to remark that this pointwise multiplier result could also be proved through a difference norm characterization as in \cite[Section~2.8.6, Proposition~1]{Tri83}, using that $F^{s}_{q,p}(\R,v_{\mu};X) \hookrightarrow L^{q}(\R,v_{\mu-sq};X)$ for $s \in (0,\frac{1+\mu}{q})$ (see \cite{MeyVersharp}).
\end{proof}

\begin{proof}[Proof of Theorem~\ref{thm:max-reg_bdd_domain_bd-data;interval}]
That $u \mapsto (u'+(\lambda-\Delta)u,\tr_{\partial\dom}u,u(0))$ is a bounded operator
\[
\M^{q,p}_{\mu,\gamma}(J) \longra \mathbb{D}^{q,p}_{\mu,\gamma}(J) \oplus \mathbb{B}^{q,p}_{\mu,\gamma}(J) \oplus \mathbb{I}^{q,p}_{\mu,\gamma}
\]
follows from a combination of Theorem~\ref{thm:trace_max-reg_space} and Theorem~\ref{thm:temporal_trace_max-reg}. That it maps to $\mathbb{D}^{q,p}_{\mu,\gamma} \oplus \mathbb{IB}^{q,p}_{\mu,\gamma}(J)$ can be seen as follows.
Of course, we only need to show that
\begin{align}\label{eq:cor:max-reg_bdd_domain_bd-data;interval;comp_cond}
\tr_{t=0}\tr_{\partial\dom}u = \tr_{\partial\dom}\tr_{t=0}u, \qquad u \in \M^{q,p}_{\mu,\gamma}(J),
\end{align}
when $1-\frac{1+\mu}{q} > \frac{1}{2}\frac{1+\gamma}{p}$. So assume $1-\frac{1+\mu}{q} > \frac{1}{2}\frac{1+\gamma}{p}$.
By a standard convolution argument and an extension and restriction argument, we see that $W^{1}_{q}(J,v_{\mu};W^{2}_{p,\gamma}(\dom))$ is dense in $\M^{q,p}_{\mu,\gamma}(J)$, from which \eqref{eq:cor:max-reg_bdd_domain_bd-data;interval;comp_cond} follows.

Injectivity of $u \mapsto (u'+(\lambda-\Delta)u,\tr_{\partial\dom}u,u(0))$ follows from the fact that $\DD$ generates a strongly continuous semigroup (see \cite{EN}) by Theorem~\ref{cor:boundeddomainLaplaceC}. So it remains to be shown that it has a bounded right-inverse, i.e.\ there is a bounded solution operator to the associated parabolic initial-boundary value problem. Using Theorem~\ref{thm:temporal_trace_max-reg} followed by Theorem~\ref{thm:trace_max-reg_space} and \eqref{eq:cor:max-reg_bdd_domain_bd-data;interval;comp_cond}, we may restrict ourselves to the case $u_{0}=0$. Furthermore, by Corollary \ref{cor:max-reg_bdd_domain} we may restrict ourselves to the case $f=0$.
By extension and restriction it is enough to treat the resulting problem for $J=\R_{+}$.
We must show that there is a bounded linear solution operator $\mathscr{S}:{_{0}}\mathbb{B}^{q,p}_{\mu,\gamma}(\R_{+}) \to
{_{0}}\M^{q,p}_{\mu,\gamma}(\R_{+}),\,g \mapsto u$ for the problem
\begin{equation}\label{eq:thm:max-reg_bdd_domain_bd-data;interval;1}
\left\{\begin{array}{rl}
u' + (\lambda-\Delta) u &= 0, \\
\tr_{\partial\dom}u &=g.
\end{array}\right.
\end{equation}
Let $E_{0} \in \mathcal{B}({_{0}}\mathbb{B}^{q,p}_{\mu,\gamma}(\R_{+}),\mathbb{B}^{q,p}_{\mu,\gamma}(\R))$ be the operator of extension by zero (see Lemma~\ref{lemma:thm:max-reg_bdd_domain_bd-data;interval;ext0}) and let $\mathscr{S}_{\R}:\mathbb{B}^{q,p}_{\mu,\gamma}(\R) \to
\M^{q,p}_{\mu,\gamma}(\R),\,g \mapsto u$ be the solution operator for the problem \eqref{eq:thm:max-reg_bdd_domain_bd-data;interval;1} on $\R$ from Theorem~\ref{thm:max-reg_bdd_domain_bd-data}.

It suffices to show that $\mathscr{S}_{\R} \circ E_{0}$ maps to ${_{0}}\mathbb{B}^{q,p}_{\mu,\gamma}(\R_{+}) $ to ${_{0}}\M^{q,p}_{\mu,\gamma}(\R)$; indeed, in that case $\mathscr{S}g:=(\mathscr{S}E_{0}g)_{|\R_{+}}$ is as desired. To do so we follow a modification of an argument given in \cite[Lemma~2.2.7]{Mey_PHD-thesis}.

Let $g \in {_{0}}\mathbb{B}^{q,p}_{\mu,\gamma}(\R_{+})$ and set $u:=\mathscr{S}_{\R}E_{0}g \in \M^{q,p}_{\mu,\gamma}(\R)$.
Pick $\phi \in C^{\infty}_{c}(\R_{+})$ with $\int_{\R}\phi(x)\,dx = 1$ and put $\phi_{n}(x):=n^{d}\phi(nx)$ for each $n \in \N_{1}$. Now consider $g_{n}:=\phi_{n}*E_{0}g \in \mathbb{B}^{q,p}_{\mu,\gamma}(\R) \cap C^{\infty}(\R;L^{p}(\partial\dom))$ and
$u_{n}:=\phi_{n}*u \in W^{\infty,q}(\R,v_{\mu};W^{2,p}(\mathscr{O},w_{\gamma}^{\dom})) \subset \M^{q,p}_{\mu,\gamma}(\R) \cap C^{\infty}(\R;L^{p}(\dom,w_{\gamma}^{\dom}))$.
Then
\begin{equation}\label{eq:thm:max-reg_bdd_domain_bd-data;interval;2}
u = \lim_{n \to \infty}u_{n} \qquad \text{in} \qquad \M^{q,p}_{\mu,\gamma}(\R)
\end{equation}
and
\[
\left\{\begin{array}{rl}
u_{n}' + (\lambda-\Delta) u_{n} &= \phi_{n}*(u' + (\lambda-\Delta) u) = 0, \\
\tr_{\partial\dom}u_{n} &= \phi_{n}*\tr_{\partial\dom}u =\phi_{n}*E_{0}g,
\end{array}\right.
\]
so that $u_{n}=\mathscr{S}_{\R}g_{n}$ by uniqueness of solutions.
Furthermore, $g_{n}(0)=0$, implying that $\tr_{\partial\dom}[u_{n}(0)] = [\tr_{\partial\dom}u_{n}](0) = g_{n}(0)=0$, so that $u_{n}(0) \in W^{2,p}_{\Dir}(\mathscr{O},w_{\gamma}^{\dom})$.
Now, as $\lambda-\Delta_{\Dir}$ is exponentially stable, we may define $v_{n} \in {_{0}}\M^{q,p}_{\mu,\gamma}(\R)$ by
\[
v_{n}(t) := \left\{\begin{array}{lr}
u_{n}(t)-e^{t(\lambda-\Delta_{\Dir})}u_{n}(0), & t \geq 0,\\
0,&  t < 0.
\end{array}\right.
\]
But then $v_{n}$ satisfies
\[
\left\{\begin{array}{rl}
v_{n}' + (\lambda-\Delta) v_{n} &= 0, \\
\tr_{\partial\dom}v_{n} &= g_{n},
\end{array}\right.
\]
so that $v_{n}=\mathscr{S}_{\R}g_{n}=u_{n}$ by uniqueness of solutions. Therefore, $u_{n} \in {_{0}}\M^{q,p}_{\mu,\gamma}(\R)$.
We may thus conclude that $u \in {_{0}}\M^{q,p}_{\mu,\gamma}(\R)$ in view of \eqref{eq:thm:max-reg_bdd_domain_bd-data;interval;2}.
\end{proof}

\begin{remark}
Theorems \ref{thm:max-reg_bdd_domain_bd-data} and \ref{thm:max-reg_bdd_domain_bd-data;interval} also remain valid in the $X$-valued setting as long as $X$ is a UMD space and $\lambda\geq \lambda_0$, where $\lambda_0$ depends on the geometry of $X$.
\end{remark}

\def\cprime{$'$}

\end{document}